\documentclass[11pt,a4paper]{article}
\usepackage{fullpage,mathrsfs,graphicx,framed,color,amssymb,amsmath,amsthm}
\usepackage{epsfig, graphicx}
\usepackage{latexsym,amsfonts,amsbsy}
\usepackage{amsmath,amsthm}
\usepackage{enumerate}
\usepackage{algorithm}
\usepackage{algpseudocode}
\usepackage{changes} 
\usepackage{color}
\usepackage[colorlinks,linkcolor=blue]{hyperref}
\usepackage[margin=1in]{geometry}
\usepackage{subcaption}
\usepackage{cleveref}
\usepackage{tikz}
\usepackage{stmaryrd}
\usetikzlibrary{positioning,calc,decorations.pathreplacing,arrows.meta}
\usepackage{xcolor}
\definecolor{mycolor}{RGB}{16,83,138}
\usepackage[normalem]{ulem}
\makeatletter 

\@addtoreset{equation}{section}
\makeatother 
\makeatletter 

\@addtoreset{equation}{section}
\makeatother 
\allowdisplaybreaks 
\makeatletter
\newenvironment{breakablealgorithm}
{
\begin{center}
\refstepcounter{algorithm}
\hrule height.8pt depth0pt \kern2pt
\renewcommand{\caption}[2][\relax]{
{\raggedright\textbf{\ALG@name~\thealgorithm} ##2\par}%
\ifx\relax##1\relax 
\addcontentsline{loa}{algorithm}{\protect\numberline{\thealgorithm}##2}%
\else 
\addcontentsline{loa}{algorithm}{\protect\numberline{\thealgorithm}##1}%
\fi
\kern2pt\hrule\kern2pt
}
}{
\kern2pt\hrule\relax
\end{center}
}
\makeatother
\newtheorem{theorem}{Theorem}[section]
\newtheorem{lemma}{Lemma}[section]
\newtheorem{corollary}{Corollary}[section]
\newtheorem{remark}{Remark}[section]

\newtheorem{proposition}{Proposition}[section]
\newtheorem{assumption}{Assumption}

\allowdisplaybreaks 

\numberwithin{equation}{section}

\begin{document}
\title{Regularity Analysis and Tensor Neural Network Methods for Quasiperiodic Elliptic Equations\footnote{This work was supported by the Strategic 
Priority Research Program of the Chinese Academy of Sciences 
(XDA0480504, XDB0620203, XDB0640000, XDB0640300), National Key
Research and Development Program of China (2023YFB3309104), 
National Key Laboratory of Computational Physics (6142A05230501), 
National Natural Science Foundations of China (1233000214),
Science Challenge Project (TZ2025007), and National Center for 
Mathematics and Interdisciplinary Science, CAS.}}
\author{Jingze Ren\footnote{School of Mathematics, Jilin University, Changchun, 
130012, Jilin, China (renjz23@mails.jlu.edu.cn)},\ \ \ 
Yifan Wang\footnote{School of Mathematical Sciences, Peking University, Beijing 100871, 
China (wangyifan1994@pku.edu.cn)},\ \ \ 
Hehu Xie\footnote{SKLMS, NCMIS, Institute of Computational Mathematics, 
Academy of Mathematics and Systems Science, Chinese Academy of
Sciences, No.55, Zhongguancun Donglu, Beijing 100190, China, and School of
Mathematical Sciences, University of Chinese Academy of Sciences, 
Beijing, 100049 (hhxie@lsec.cc.ac.cn)}\ \ \
and \ \ Qilong Zhai\footnote{School of Mathematics, Jilin University, Changchun, 
130012, Jilin, China (zhaiql@jlu.edu.cn)} }
\date{}
\maketitle

\begin{abstract}
In this paper, we propose a novel machine learning method based on an adaptive 
tensor neural network subspace for solving quasiperiodic elliptic problems. 
To this end, we first provide a theoretical analysis of the associated quasiperiodic 
and periodic function spaces and establish regularity estimates 
for the quasiperiodic elliptic problems. In particular, 
under the Diophantine condition, we derive a suitable condition 
on the source term to guarantee the regularity of the solution,  
which provides a theoretical basis for the design of numerical schemes. 
An efficient numerical method is then designed by combining the 
projection method with tensor neural networks. Leveraging the special 
structure of tensor neural networks, high-dimensional integration 
can be performed directly and with high accuracy, without relying on 
Monte Carlo methods. Finally, several numerical experiments are 
presented to demonstrate the accuracy and efficiency of the proposed method.

\end{abstract}

\section{Introduction}

In recent years, partial differential equations (PDEs) with quasiperiodic coefficients have played an increasingly important role 
in describing intriguing phenomena in physics and materials science. 
Quasicrystals are a class of aperiodic space-filling structures, 
which share certain features with periodic systems, such as crystals and periodic tilings, and exhibit long-range order. 
However, due to the absence of translational symmetry and spatial decay, quasiperiodic PDEs are naturally posed in the whole space. 
As a result, developing effective numerical methods for PDEs with quasiperiodic coefficients has become an active area of research. 
Several numerical approaches have been explored to tackle this problem.
One of the most commonly used methods is the periodic approximation method (PAM), which approximates quasiperiodic functions 
by periodic ones and thereby imposes periodic boundary conditions on quasiperiodic PDEs.
However, the accuracy of PAM is strongly influenced by the Diophantine error \cite{JiangLiZhang_2025}. 
In the quasiperiodic corrector problem, the filtering method (FM) approximates whole-space quasiperiodic elliptic PDEs by restricting them
to a finite computational domain and using a filter function whose derivatives vanish to all orders at the boundary.
This construction naturally induces modified homogeneous boundary conditions.
From a numerical performance perspective, although this method can suppress oscillations, 
it still fails to eliminate the influence of Diophantine errors \cite{BlancBris}.


Another approach is a direct method to study the quasicrystals or aperiodic crystals in the hyperspace, 
called the higher-dimensional approach. 
From this approach, a quasiperiodic structure can be viewed as a periodic structure by extending it into a higher-dimensional space.
Its symmetries can be expressed in terms of the conventional point groups and space groups of higher-dimensional 
periodic crystals \cite{JiangZhang_2015,SiLiZhangShiJiang}. 
In the higher-dimensional description, quasiperiodic structures result from irrational physical-space cuts of appropriate 
periodic hypercrystal structures.  It is the so-called cut-and-project method. 
The higher-dimensional approach reveals the hidden structural correlations.
To implement the higher-dimensional approach in the direct space, one must know the discrete lattice arrangements 
of higher-dimensional periodic structures.
The embedding spaces of $d$-dimensional quasiperiodic structures are abstract spaces whose dimensions are more than three.  

Recently, the projection method (PM) \cite{JiangZhang_2014,JiangZhang_2018} provides an effective framework for transforming quasiperiodic 
problems posed in the whole space into high-dimensional periodic problems.
Essentially, the PM technique reveals that the quasiperiodic characteristics can be defined on an irrational manifold in a higher-dimensional torus.
So far, the PM has been combined with the spectral collocation method and discrete Fourier-Bohr transformation to approximate quasiperiodic functions. 
By computing higher-dimensional periodic parent functions and incorporating them with a projection matrix, 
one can obtain approximations to quasiperiodic functions. 
The PM, combined with spectral collocation methods, has been applied to many quasiperiodic problems, 
such as \cite{BarkanEngelLifshitz, JiangZhang_2015, SiLiZhangShiJiang}, 
incommensurate quantum systems \cite{GaoXuYangYe,LiJiang}, topological insulators \cite{WangLiuHuang}, 
and grain boundaries \cite{CaoShenXu, JiangSiXu}. 
The PM transforms a lower-dimensional quasiperiodic elliptic problem into a higher-dimensional periodic problem, 
which brings significant challenges for classical numerical methods. 
Furthermore, it is easy to find that the deduced high-dimensional periodic problem has no the positive definiteness 
in the classical Sobolev space. This property brings difficulty to design numerical methods which has rigorous convergence rate.  
The main motivation of this paper is therefore to develop proper regularity analysis and 
efficient machine-learning-based algorithms for the quasiperiodic elliptic problem. 



Over the last few years, machine learning methods based on neural network (NN) to solve PDEs have attracted 
increasing attention in the computational mathematics and scientific computing communities.
Several well-known NN-based methods have been developed, including the deep Ritz \cite{EYu}, 
deep Galerkin method \cite{DGM}, PINN \cite{RaissiPerdikarisKarniadakis}, and weak adversarial networks \cite{WAN} 
for solving PDEs by designing different loss functions. 
In these methods, the loss functions typically involve the numerical integration of neural network-defined functions.
Direct numerical integration of high-dimensional functions is always subject to the ``curse of dimensionality''. 
In practice, such high-dimensional integrations are often computed using Monte Carlo methods and various sampling strategies \cite{EYu}.
However, due to the low convergence rate of Monte Carlo methods, achieving high accuracy and stable convergence remains challenging.

To improve the accuracy of high-dimensional integration, we previously proposed a type of tensor neural network (TNN) 
and the corresponding machine learning method to solve high-dimensional problems with high accuracy \cite{WangJinXie, WangLinLiaoLiuXie, WangLiaoXie}.
The key advantage of the TNN-based method is that the high-dimensional integration of TNN in the loss functions 
can be transformed into one-dimensional integrations, so that highly accurate classical quadrature rules can be implemented.
The TNN-based machine learning method has already been used to solve high-dimensional eigenvalue and boundary value problems 
using Ritz-type loss functions \cite{WangJinXie}. 
Furthermore, in \cite{WangXie}, it was shown that multiple eigenpairs can be computed by combining TNNs with the Rayleigh-Ritz process.
These results demonstrate the strong potential of TNN-based machine learning methods for solving high-dimensional problems.

The aim of this paper is to develop a TNN-based machine learning method for solving quasiperiodic elliptic problems.
To this end, we first investigate the properties of quasiperiodic function spaces and periodic function spaces, 
analyze the embedding relations between quasiperiodic function spaces and standard Sobolev spaces.
Based on these results and the Diophantine condition, we establish the regularity improvement 
in the Sobolev space for quasiperiodic elliptic problems.
These results provide the theoretical foundation for the design of numerical schemes for quasiperiodic elliptic problems.
With the help of the projection method, the low-dimensional quasiperiodic problem is transformed into a higher-dimensional 
periodic problem, which is then solved by the TNN-based machine learning method.
An approximation to the original quasiperiodic problem is then obtained by applying the projection process to 
the high-dimensional TNN function.
In order to validate the effectiveness of the proposed numerical algorithm, we present comprehensive numerical 
experiments for various quasiperiodic elliptic equations.
These results demonstrate the remarkable efficiency of our method and indicate that the method has strong potential 
for broader quasiperiodic systems.

An outline of the paper is as follows. 
In Section \ref{Section_Projection}, we introduce the quasiperiodic and periodic function spaces. 
Section \ref{Section_Regularity} is devoted to the regularity analysis of quasiperiodic elliptic problems.
In Section \ref{Section_TNN_ML}, we develop a    
TNN-based machine learning method
will be proposed to solve the resulting high-dimensional periodic problem.  
Section \ref{Section_Numerical_Examples} contains numerical examples illustrating the accuracy and efficiency of the proposed method. 
The paper ends with some concluding remarks.


\section{Quasiperiodic and periodic function spaces}\label{Section_Projection}
In this section, we introduce the low-dimensional quasiperiodic function spaces and 
the corresponding high-dimensional periodic function spaces. The isometric 
isomorphism between these two function spaces is also investigated. 
These function spaces and their relationship provide the foundation for the
regularity analysis of quasiperiodic elliptic problems in the next section.
\subsection{Quasiperiodic function spaces}
This subsection is devoted to a special class of almost periodic functions 
whose frequency module is generated by a matrix $P$ 
with $\mathbb Q$-linearly independent columns.

Let
\begin{eqnarray*}
\mathbb{P}^{d \times n} := \Big\{P
= (\boldsymbol{p}_1, \cdots, \boldsymbol{p}_n) \in \mathbb{R}^{d \times n}: \boldsymbol{p}_1, 
\cdots, \boldsymbol{p}_n \text{ are }  \mathbb{Q}\text{-linearly independent}\Big\},
\end{eqnarray*}
where a set of vectors is said to be $\mathbb{Q}$-linearly independent if, for any rational numbers 
$\mu_1, \ldots, \mu_n$, the equation $\mu_1 \boldsymbol{p}_1+\cdots+\mu_n \boldsymbol{p}_n=0$ implies 
$\mu_1=\cdots=\mu_n=0$.

For $P\in \mathbb{P}^{d \times n}$, we call $P$ a projection matrix and define the trigonometric polynomial 
space by the frequency modulus $\Lambda_P:= \{P\boldsymbol{k}:\boldsymbol{k}\in\mathbb Z^n\}$ as follows
\begin{eqnarray*}
\mathcal{T}_{P}:=\left\{u(\boldsymbol{x})
=\sum_{\boldsymbol{k}\in \mathcal K} u_{\boldsymbol{k}} e^{{\rm i}(P
\boldsymbol{k})^\top \boldsymbol{x}}: \mathcal K\subset \mathbb Z^n, {\rm crad}(\mathcal K)<\infty   \right\},
\end{eqnarray*}
where ${\rm crad}(\mathcal K)$ denotes the cardinal of set $\mathcal K$.
The functions in $\mathcal{T}_{P}$ are basic smooth quasiperiodic functions. Note that each $u\in \mathcal{T}_{P}$ 
can be written as $u(\boldsymbol{x})=U(P^\top \boldsymbol{x})$, where $U$ is a $n$-dimensional periodic function.

Let $K_T = \{\boldsymbol{x} = (x_1, \cdots, x_d) \in \mathbb{R}^d : |x_j| \leq T,\ j = 1, \cdots, d\}$.
The mean value of $f(\boldsymbol{x})$ is defined by
\begin{equation*}
\mathcal{M}\{f(\boldsymbol{x})\} = \lim_{T \to \infty} \frac{1}{(2T)^d} \int_{\boldsymbol{s} 
+ K_T} f(\boldsymbol{x}) d\boldsymbol{x},
\end{equation*}
where the limit on the right side exists uniformly for all $\boldsymbol{s} \in \mathbb{R}^d$ \cite{JiangLiZhang}.
An elementary calculation shows that for $\boldsymbol{\alpha},\boldsymbol{\beta}\in \mathbb R^d$ 
\begin{eqnarray}\label{B_P^2_orth}
\mathcal{M} \left(e^{{\rm i}\boldsymbol{\alpha}^\top \boldsymbol{x}} 
e^{-{\rm i}\boldsymbol{\beta}^\top \boldsymbol{x}}\right)=
\left\{
\begin{aligned}
1,\ \ \ \boldsymbol{\alpha}=\boldsymbol{\beta},\\
0,\ \ \ \boldsymbol{\alpha}\neq\boldsymbol{\beta}.
\end{aligned}
\right.
\end{eqnarray}

For any two quasiperiodic functions 
$u(\boldsymbol{x})=\sum_{\boldsymbol{k}\in \mathcal{K}_u} \widehat{u}_{\boldsymbol{k}} e^{{\rm i}(P\boldsymbol{k})^\top \boldsymbol{x}}$ and $v(\boldsymbol{x})
=\sum_{\boldsymbol{k}\in \mathcal{K}_v} \widehat{v}_{\boldsymbol{k}} 
e^{{\rm i}(P\boldsymbol{k})^\top \boldsymbol{x}}$ in $\mathcal{T}_{P}$, 
where $\mathcal K_u$ and $\mathcal K_v$ denote the frequency index sets of $u$ and $v$, respectively,
we define the following inner product
\begin{eqnarray*}
(u,v)_{B_P^2}:=\mathcal{M}(u\overline{v})=\sum_{\boldsymbol{k}\in \mathcal K_u\cap 
\mathcal K_v}\widehat{u}_{\boldsymbol{k}} \overline{\widehat{v}}_{\boldsymbol{k}},
\end{eqnarray*}
and the corresponding induced norm
\begin{eqnarray*}
\|u\|_{B_P^2}^2:=\mathcal{M}(|u|^2)=\sum_{\boldsymbol{k}\in \mathcal K_u}|\widehat{u}_{\boldsymbol{k}}|^2.
\end{eqnarray*}
With the help of the norm $\|\cdot\|_{B_P^2}$, 
we define the Besicovitch space where the spectrum is fixed 
as the column vectors of $P$ as follows
\begin{eqnarray*}
B_P^2(\mathbb R^d):= \overline{\mathcal{T}_{P}}^{\|\cdot\|_{B_P^2}},
\end{eqnarray*}
This definition has the following equivalent form
\begin{eqnarray*}
B_P^2(\mathbb R^d):= \left\{u(\boldsymbol{x})=\sum_{\boldsymbol{k}\in \mathbb Z^n} 
\widehat{u}_{\boldsymbol{k}} e^{{\rm i}(P\boldsymbol{k})^\top \boldsymbol{x}}: 
\sum_{\boldsymbol{k}\in \mathbb Z^n}|\widehat{u}_{\boldsymbol{k}}|^2<\infty   \right\},
\end{eqnarray*}
where the coefficient $\widehat{u}_{\boldsymbol{k}},\boldsymbol{k}\in\mathbb Z^n$ 
is called Bohr coefficient and can be obtained by the following Bohr transformation
\begin{eqnarray*}
\widehat{u}_{\boldsymbol{k}}=\mathcal{M}\left(u(\boldsymbol{x})
e^{-{\rm i}(P\boldsymbol{k})^\top \boldsymbol{x}}\right),\ \ \ \boldsymbol{k}\in\mathbb Z^n.
\end{eqnarray*}
We also define the subspace of $B_P^2(\mathbb R^d)$ consisting of functions with zero mean
\begin{eqnarray*}
B_{P,0}^2(\mathbb R^d)&:=&\left\{u\in B_P^2(\mathbb R^d):\mathcal{M}(u)=0  \right\} \\
&=&\left\{u(\boldsymbol{x})=\sum_{\boldsymbol{k}\in \mathbb Z^n\setminus\{\boldsymbol 0\}} 
\widehat{u}_{\boldsymbol{k}} e^{{\rm i}(P\boldsymbol{k})^\top \boldsymbol{x}}: \widehat{u}_{\boldsymbol{0}}=0, 
\sum_{\boldsymbol{k}\in \mathbb Z^n\setminus\{\boldsymbol 0\}}|\widehat{u}_{\boldsymbol{k}}|^2<\infty   \right\}.
\end{eqnarray*}
The Besicovitch space $B_P^2(\mathbb R^d)$ and $B_{P,0}^2(\mathbb R^d)$ are the most natural square integrable space for quasiperiodic problems.
In this quasiperiodic sense, a space with Sobolev regularity is naturally defined as follows
\begin{eqnarray*}
H_P^s(\mathbb R^d):= \left\{u(\boldsymbol{x})=\sum_{\boldsymbol{k}\in \mathbb Z^n} 
\widehat{u}_{\boldsymbol{k}} e^{{\rm i}(P\boldsymbol{k})^\top \boldsymbol{x}}: 
\sum_{\boldsymbol{k}\in \mathbb Z^n}(1+|P\boldsymbol{k}|^2)^s  |\widehat{u}_{\boldsymbol{k}}|^2<\infty   \right\},
\end{eqnarray*}
where $|\boldsymbol{x}|^2=\sum_{i=1}^d |x_i|^2$.
Our subsequent analysis requires the following homogeneous space
\begin{eqnarray*}
\overline{H}_P^s(\mathbb R^d):= \left\{u(\boldsymbol{x})=\sum_{\boldsymbol{k}\in \mathbb Z^n} 
\widehat{u}_{\boldsymbol{k}} e^{{\rm i}(P\boldsymbol{k})^\top \boldsymbol{x}}: 
\sum_{\boldsymbol{k}\in \mathbb Z^n}|P\boldsymbol{k}|^{2s}  |\widehat{u}_{\boldsymbol{k}}|^2<\infty   \right\},
\end{eqnarray*}
and the corresponding zero-mean subspaces 
$H_{P,0}^s(\mathbb R^d):=\left\{u\in H_P^s(\mathbb R^d):\mathcal{M}(u)=0  \right\}$ 
and $\overline{H}_{P,0}^s(\mathbb R^d):=\left\{u\in \overline{H}_P^s(\mathbb R^d):\mathcal{M}(u)=0  \right\}$.

For $u(\boldsymbol{x})=\sum_{\boldsymbol{k}\in \mathbb Z^n} \widehat{u}_{\boldsymbol{k}} e^{{\rm i}(P\boldsymbol{k})^\top \boldsymbol{x}}$
and
$v(\boldsymbol{x})=\sum_{\boldsymbol{k}\in \mathbb Z^n} \widehat{v}_{\boldsymbol{k}} e^{{\rm i}(P\boldsymbol{k})^\top \boldsymbol{x}}$,
we define the inner product on $H_P^s(\mathbb R^d)$ as
\begin{eqnarray*}
(u,v)_{H_P^s}
:=
\sum_{\boldsymbol{k}\in \mathbb Z^n}
(1+|P\boldsymbol{k}|^2)^s
\widehat{u}_{\boldsymbol{k}}\overline{\widehat{v}}_{\boldsymbol{k}},
\end{eqnarray*}
and the corresponding induced norm
\begin{eqnarray*}
\|u\|_{H_P^s}
:= \left(\sum_{\boldsymbol{k}\in \mathbb Z^n}
(1+|P\boldsymbol{k}|^2)^s
|\widehat{u}_{\boldsymbol{k}}|^2\right)^{1/2}.
\end{eqnarray*}
The zero-mean subspace $H_{P,0}^s(\mathbb R^d)$ 
is endowed with the inner products and norms induced 
by those of $H_P^s(\mathbb R^d)$.
Equivalently, the corresponding sums may be taken 
over $\mathbb Z^n\setminus\{\boldsymbol 0\}$.

Similarly, the zero-mean homogeneous space $\overline{H}_{P,0}^s(\mathbb R^d)$ is equipped with the inner product
\begin{eqnarray*}
(u,v)_{\overline{H}_{P,0}^s(\mathbb R^d)}
:=\sum_{\boldsymbol{k}\in \mathbb Z^n\setminus\{\boldsymbol 0\}}
|P\boldsymbol{k}|^{2s}
\widehat{u}_{\boldsymbol{k}}\overline{\widehat{v}_{\boldsymbol{k}}},
\end{eqnarray*}
and the norm
\begin{eqnarray*}
\|u\|_{\overline{H}_{P,0}^s(\mathbb R^d)}
:=
\left(
\sum_{\boldsymbol{k}\in \mathbb Z^n\setminus\{\boldsymbol 0\}}
|P\boldsymbol{k}|^{2s}
|\widehat{u}_{\boldsymbol{k}}|^2
\right)^{1/2}.
\end{eqnarray*}

Besides the physical regularity measured by the projected 
frequencies $P\boldsymbol{k}$, we also introduce spaces 
measuring the regularity of the generating function on $\mathbb T^n$. 
\begin{eqnarray*}
H^s(\mathbb R^d):= \left\{u(\boldsymbol{x})=\sum_{\boldsymbol{k}\in \mathbb Z^n} \widehat{u}_{\boldsymbol{k}} 
e^{{\rm i}(P\boldsymbol{k})^\top \boldsymbol{x}}: \sum_{\boldsymbol{k}\in \mathbb Z^n}(1+|\boldsymbol{k}|^2)^s  
|\widehat{u}_{\boldsymbol{k}}|^2<\infty   \right\},
\end{eqnarray*}
Similarly, we define the corresponding homogeneous space by
\begin{eqnarray*}
\overline{H}^s(\mathbb R^d):= \left\{u(\boldsymbol{x})=\sum_{\boldsymbol{k}\in \mathbb Z^n} \widehat{u}_{\boldsymbol{k}} 
e^{{\rm i}(P\boldsymbol{k})^\top \boldsymbol{x}}: \sum_{\boldsymbol{k}\in \mathbb Z^n}|\boldsymbol{k}|^{2s}  
|\widehat{u}_{\boldsymbol{k}}|^2<\infty   \right\},
\end{eqnarray*}
and the corresponding zero-mean spaces $H_{0}^s(\mathbb R^d):=\left\{u\in H^s(\mathbb R^d):\mathcal{M}(u)=0  \right\}$ 
and $\overline{H}_{0}^s(\mathbb R^d):=\left\{u\in \overline{H}^s(\mathbb R^d):\mathcal{M}(u)=0  \right\}$.

For
$u(\boldsymbol{x})=\sum_{\boldsymbol{k}\in \mathbb Z^n} \widehat{u}_{\boldsymbol{k}} e^{{\rm i}(P\boldsymbol{k})^\top \boldsymbol{x}}$
and
$v(\boldsymbol{x})=\sum_{\boldsymbol{k}\in \mathbb Z^n} \widehat{v}_{\boldsymbol{k}} e^{{\rm i}(P\boldsymbol{k})^\top \boldsymbol{x}}$,
we define the inner product on $H^s(\mathbb R^d)$ 
\begin{eqnarray*}
(u,v)_{H^s}
:=
\sum_{\boldsymbol{k}\in \mathbb Z^n}
(1+|\boldsymbol{k}|^2)^s
\widehat{u}_{\boldsymbol{k}}\overline{\widehat{v}_{\boldsymbol{k}}},
\end{eqnarray*}
and the induced norm 
\begin{eqnarray*}
\|u\|_{H^s}
:=
\left(
\sum_{\boldsymbol{k}\in \mathbb Z^n}
(1+|\boldsymbol{k}|^2)^s
|\widehat{u}_{\boldsymbol{k}}|^2
\right)^{1/2}.
\end{eqnarray*}
The zero-mean space $H_0^s(\mathbb R^d)$ is equipped with the inner product and norm induced by those of $H^s(\mathbb R^d)$.
Equivalently, with the above sums restricted to $\mathbb Z^n\setminus\{\boldsymbol 0\}$.

Similarly, the zero-mean homogeneous space $\overline{H}^s_0(\mathbb R^d)$ is equipped with the inner product
\begin{eqnarray*}
(u,v)_{\overline{H}^s_0}
:=
\sum_{\boldsymbol{k}\in \mathbb Z^n\setminus\{\boldsymbol 0\}}
|\boldsymbol{k}|^{2s}
\widehat{u}_{\boldsymbol{k}}\overline{\widehat{v}_{\boldsymbol{k}}},
\end{eqnarray*}
and the corresponding norm
\begin{eqnarray*}
\|u\|_{\overline{H}^s_0}
:=
\left(
\sum_{\boldsymbol{k}\in \mathbb Z^n\setminus\{\boldsymbol 0\}}
|\boldsymbol{k}|^{2s}
|\widehat{u}_{\boldsymbol{k}}|^2
\right)^{1/2}.
\end{eqnarray*}
\begin{remark}
The spaces $H^s(\mathbb R^d)$ and $H_{P}^s(\mathbb R^d)$ encode two different notions 
of regularity. $H_{P}^s(\mathbb R^d)$ measures the Sobolev regularity with respect 
to the physical variable $x\in \mathbb{R}^d$.
As will be seen from the following analysis of the pullback mapping, 
$H^s(\mathbb R^d)$ measures the Sobolev regularity of the generating function 
on $\mathbb{T}^n$.
\end{remark}

\subsection{Periodic function spaces}
This subsection introduces periodic function spaces in high-dimensional settings.
In particular, the projected periodic function space is introduced as a suitable framework for quasiperiodic elliptic problems. 
Let $\mathbb T^n=(\mathbb R/2\pi\mathbb Z)^n$ be the $n$-dimensional torus.
We first introduce several basic function spaces on $\mathbb T^n$.
Define the periodic Hilbert space $L^2(\mathbb T^n)$
\begin{eqnarray*}
L^2(\mathbb T^n):=\left\{U(\boldsymbol{y})= \sum_{\boldsymbol{k}\in\mathbb Z^n} \widehat{U}_{\boldsymbol{k}} 
e^{{\rm i}\boldsymbol{k}^\top \boldsymbol{y}} : \sum_{\boldsymbol{k}\in\mathbb Z^n}|\widehat{U}_{\boldsymbol{k}}|^2<\infty \right\},
\end{eqnarray*}
equipped with the inner product
\[
(U,V)_{L^2}
:=
\sum_{\boldsymbol{k}\in\mathbb Z^n}
\widehat U_{\boldsymbol{k}}\overline{\widehat V_{\boldsymbol{k}}},
\]
and the deduced norm
\[
\|U\|_{L^2}
:=
\left(
\sum_{\boldsymbol{k}\in\mathbb Z^n}
|\widehat U_{\boldsymbol{k}}|^2
\right)^{1/2}.
\]
It is obvious that $L^2(\mathbb T^n)$ is 
the standard space of square-integrable periodic functions on $\mathbb T^n$.
The Sobolev space $H^s(\mathbb T^n)$ is defined as
\begin{eqnarray*}
H^s(\mathbb T^n):=\left\{U(\boldsymbol{y})= \sum_{\boldsymbol{k}\in\mathbb Z^n} 
\widehat{U}_{\boldsymbol{k}} e^{{\rm i}\boldsymbol{k}^\top \boldsymbol{y}} : \sum_{\boldsymbol{k}\in\mathbb Z^n} 
(1+|\boldsymbol{k}|^2)^s|\widehat{U}_{\boldsymbol{k}}|^2<\infty \right\},
\end{eqnarray*}
endowed with the inner product
\[
(U,V)_{H^s}
:= \sum_{\boldsymbol{k}\in\mathbb Z^n}
(1+|\boldsymbol{k}|^2)^s
\widehat U_{\boldsymbol{k}}\overline{\widehat V_{\boldsymbol{k}}},
\]
and the norm
\[
\|U\|_{H^s}
:=
\left(
\sum_{\boldsymbol{k}\in\mathbb Z^n}
(1+|\boldsymbol{k}|^2)^s
|\widehat U_{\boldsymbol{k}}|^2
\right)^{1/2}.
\]
Now we come to define a closed subspace of $H^s(\mathbb T^n)$
consisting of functions with zero mean, 
or equivalently, with vanishing zeroth Fourier coefficient as 
\begin{eqnarray*}
H_{0}^s(\mathbb T^n):=\left\{U(\boldsymbol{y})= \sum_{\boldsymbol{k}\in\mathbb Z^n\setminus\{\boldsymbol 0\} } 
\widehat{U}_{\boldsymbol{k}} e^{{\rm i}\boldsymbol{k}^\top \boldsymbol{y}} : \sum_{\boldsymbol{k}\in\mathbb Z^n\setminus\{\boldsymbol 0\} } 
(1+|\boldsymbol{k}|^2)^s|\widehat{U}_{\boldsymbol{k}}|^2<\infty \right\}.
\end{eqnarray*}
The periodic space $H_0^s(\mathbb T^n)$ is a Hilbert space equipped 
with the inner product and norm deduced by those of $H^s(\mathbb T^n)$.

We further define the zero-mean homogeneous Sobolev space 
\begin{eqnarray*}
\overline{H}_{0}^s(\mathbb T^n):=\left\{U(\boldsymbol{y})= \sum_{\boldsymbol{k}\in\mathbb Z^n\setminus\{\boldsymbol 0\}} 
\widehat{U}_{\boldsymbol{k}} e^{{\rm i}\boldsymbol{k}^\top \boldsymbol{y}} : \sum_{\boldsymbol{k}\in\mathbb Z^n\setminus\{\boldsymbol 0\}} 
|\boldsymbol{k}|^{2s}|\widehat{U}_{\boldsymbol{k}}|^2<\infty \right\}.
\end{eqnarray*}
Then, $\overline H_0^s(\mathbb T^n)$ is a Hilbert space with inner product
\[
(U,V)_{\overline H_0^s}
:=
\sum_{\boldsymbol{k}\in\mathbb Z^n\setminus\{\boldsymbol 0\}}
|\boldsymbol{k}|^{2s}
\widehat U_{\boldsymbol{k}}\overline{\widehat V_{\boldsymbol{k}}},
\]
and the corresponding norm
\[
\|U\|_{\overline H_0^s}
:=
\left(
\sum_{\boldsymbol{k}\in\mathbb Z^n\setminus\{\boldsymbol 0\}}
|\boldsymbol{k}|^{2s}
|\widehat U_{\boldsymbol{k}}|^2
\right)^{1/2}.
\]

We then introduce several function spaces associated with the quasiperiodic problem. 
These will be referred to as the $P$-projected Sobolev-type spaces. We begin with
\begin{eqnarray*}
H_{P}^s(\mathbb T^n):=\left\{U(\boldsymbol{y})= \sum_{\boldsymbol{k}\in\mathbb Z^n} \widehat{U}_{\boldsymbol{k}} e^{{\rm i}
\boldsymbol{k}^\top \boldsymbol{y}} : \sum_{\boldsymbol{k}\in\mathbb Z^n} (1+|P\boldsymbol{k}|^2)^s|\widehat{U}_{\boldsymbol{k}}|^2<\infty \right\},
\end{eqnarray*}
which is a Hilbert space endowed with the inner product
\[
(U,V)_{H_{P}^s}
:=\sum_{\boldsymbol{k}\in\mathbb Z^n}
(1+|P\boldsymbol{k}|^2)^s
\widehat U_{\boldsymbol{k}}\overline{\widehat V_{\boldsymbol{k}}},
\]
and norm
\[
\|U\|_{H_{P}^s}
:=
\left(
\sum_{\boldsymbol{k}\in\mathbb Z^n}
(1+|P\boldsymbol{k}|^2)^s
|\widehat U_{\boldsymbol{k}}|^2
\right)^{1/2}.
\]
This space measures regularity through the projected frequencies $P\boldsymbol{k}$ rather than the frequencies $\boldsymbol{k}$.
Define the zero-mean subspace of $H_{P}^s(\mathbb T^n)$ as follows
\begin{eqnarray*}
H_{P,0}^s(\mathbb T^n):=\left\{U(\boldsymbol{y})= \sum_{\boldsymbol{k}\in\mathbb Z^n\setminus\{\boldsymbol 0\}} 
\widehat{U}_{\boldsymbol{k}} e^{{\rm i}\boldsymbol{k}^\top \boldsymbol{y}} : 
\sum_{\boldsymbol{k}\in\mathbb Z^n\setminus\{\boldsymbol 0\}}(1+|P\boldsymbol{k}|^2)^s|\widehat{U}_{\boldsymbol{k}}|^2<\infty \right\}.
\end{eqnarray*}
Then, $H_{P,0}^s(\mathbb T^n)$ is a Hilbert space equipped
with the inner product and norm induced by those of $H_{P}^s(\mathbb T^n)$.

We further define the homogeneous counterpart of $H_{P,0}^s(\mathbb T^n)$, 
through the Fourier weight $|P\boldsymbol{k}|^{2s}$ as follows
\begin{eqnarray*}
\overline{H}_{P,0}^s(\mathbb T^n):=\left\{U(\boldsymbol{y})= \sum_{\boldsymbol{k}\in\mathbb Z^n\setminus\{\boldsymbol 0\}} 
\widehat{U}_{\boldsymbol{k}} e^{{\rm i}\boldsymbol{k}^\top \boldsymbol{y}} : 
\sum_{\boldsymbol{k}\in\mathbb Z^n\setminus\{\boldsymbol 0\}}|P\boldsymbol{k}|^{2s}|\widehat{U}_{\boldsymbol{k}}|^2<\infty \right\},
\end{eqnarray*}
which is a Hilbert space endowed with the inner product
\[
(U,V)_{\overline H_{P,0}^s}
:=
\sum_{\boldsymbol{k}\in\mathbb Z^n\setminus\{\boldsymbol 0\}}
|P\boldsymbol{k}|^{2s}
\widehat U_{\boldsymbol{k}}\overline{\widehat V_{\boldsymbol{k}}},
\]
and the corresponding norm
\[
\|U\|_{\overline H_{P,0}^s}
:=
\left(
\sum_{\boldsymbol{k}\in\mathbb Z^n\setminus\{\boldsymbol 0\}}
|P\boldsymbol{k}|^{2s}
|\widehat U_{\boldsymbol{k}}|^2
\right)^{1/2}.
\]
The space $\overline H_{P,0}^s(\mathbb T^n)$ plays a fundamental role in the analysis of the quasiperiodic problems.
This space arises naturally when a low-dimensional quasiperiodic problem is transformed into the corresponding high-dimensional periodic problem.
It also serves as the natural energy space for the resulting high-dimensional periodic problem. 
For this reason, the remainder of this subsection is devoted to the study of the properties of $\overline H_{P,0}^s(\mathbb T^n)$. 
\begin{lemma}
For any $s\in\mathbb R$, $\overline H_{P,0}^{-s}(\mathbb T^n)\cong \big(\overline{H}_{P,0}^{s}(\mathbb T^n)\big)^\prime$.
\end{lemma}

\begin{proof} 
Let us define the operator $\mathcal T$ as follows 
\begin{eqnarray*}
(\mathcal{T}F)(U):=\sum_{\boldsymbol{k}\in \mathbb Z^n\setminus\{\boldsymbol 0\}} \overline{\widehat{F}_{\boldsymbol{k}}}\widehat{U}_{\boldsymbol{k}}. 
\end{eqnarray*}
Then we will show that the operator $\mathcal T$ defines an isometric isomorphism between  
$\overline{H}_{P,0}^{-s}(\mathbb T^n)$ and  
$\big(\overline{H}_{P,0}^{s}(\mathbb T^n)\big)^\prime$. 

First, we prove that $\mathcal{T}F\in \big(\overline{H}_{P,0}^{s}(\mathbb T^n)\big)^\prime$ 
for any $F\in \overline{H}_{P,0}^{-s}(\mathbb T^n)$. 
For any $U\in \overline{H}_{P,0}^{s}(\mathbb T^n)$, the following inequalities hold 
\begin{eqnarray*}
|(\mathcal{T}F)(U)|&=&\sum_{\boldsymbol{k}\in \mathbb Z^n\setminus\{\boldsymbol 0\}}\overline{\widehat{F}_{\boldsymbol{k}}}\widehat{U}_{\boldsymbol{k}}
=\sum_{\boldsymbol{k}\in \mathbb Z^n\setminus\{\boldsymbol 0\}}\left(|P\boldsymbol{k}|^{-s}
\overline{\widehat{F}}_{\boldsymbol{k}}\right) \left(|P\boldsymbol{k}|^s\widehat{U}_{\boldsymbol{k}}\right)\\
&\leq&\left(\sum_{\boldsymbol{k}\in \mathbb Z^n\setminus\{\boldsymbol 0\}}|P\boldsymbol{k}|^{-2s}
\Big|\overline{\widehat{F}}_{\boldsymbol{k}}\Big|^2\right)^{\frac{1}{2}}
\left(\sum_{\boldsymbol{k}\in \mathbb Z^n\setminus\{\boldsymbol 0\}}|P\boldsymbol{k}|^{2s}\big|\widehat{U}_{\boldsymbol{k}}\big|^2\right)^{\frac{1}{2}}\\
&=&\|F\|_{\overline{H}_{P}^{-s}} \|U\|_{\overline{H}_{P}^{s}}.
\end{eqnarray*}
Then since $U$ is arbitrary, the desired boundedness follows
\begin{eqnarray*}
\|\mathcal{T}F\|_{\big(\overline{H}_{P,0}^{s}\big)^\prime}
=\sup_{0\neq U\in \overline{H}_{P,0}^{s}}
\frac{|(\mathcal{T}F)(U)|}{\|U\|_{\overline{H}_{P}^{s}}}
\leq \|F\|_{\overline{H}_{P}^{-s}}<\infty.
\end{eqnarray*}

Next, we prove that the upper bound of the above inequality can be obtained. 
Define $W=\sum_{\boldsymbol{k}\in \mathbb Z^n\setminus\{\boldsymbol 0\}}|P\boldsymbol{k}|^{-2s}
\widehat{F}_{\boldsymbol{k}}e^{{\rm i}\boldsymbol{k}^\top \boldsymbol{y}}$, 
it is easy to verify $\|W\|_{\overline{H}_{P}^{s}}=\|F\|_{\overline{H}_{P}^{-s}}<\infty$, and 
\begin{eqnarray*}
\|\mathcal{T}F\|_{\big(\overline{H}_{P,0}^{s}\big)^\prime} \geq \frac{|(\mathcal{T}F)(W)|}{\|W\|_{\overline{H}_{P}^{s}}}
= \|F\|_{\overline{H}_{P}^{-s}}.
\end{eqnarray*}

Finally, we prove that $\mathcal{T}$ is a surjection.
By the Riesz representation theorem, for any $\mathcal{L}\in \big(\overline{H}_{P,0}^{s}\big)^\prime$, 
there exists a unique $G\in \overline{H}_{P,0}^{s}$, such that
\begin{eqnarray*}
\mathcal{L}(U)=(U,G)_{\overline{H}_{P}^{s}}=\sum_{\boldsymbol{k}\in \mathbb Z^n\setminus\{\boldsymbol 0\}} 
|P\boldsymbol{k}|^{2s} \widehat{U}_{\boldsymbol{k}}\overline{\widehat{G}}_{\boldsymbol{k}}.
\end{eqnarray*}
Define Fourier coefficients of $F$ as follows
\begin{eqnarray*}
\widehat{F}_{0}=0,\ \ \ \widehat{F}_{\boldsymbol{k}}=|P\boldsymbol{k}|^{2s}\widehat{G}_{\boldsymbol{k}},
\ \ \ \boldsymbol{k}\in \mathbb Z^n\setminus\{\boldsymbol 0\}.
\end{eqnarray*}
It can be obtained through simple calculation that $\|F\|_{\overline{H}_{P}^{-s}}=\|G\|_{\overline{H}_{P}^{s}}<\infty$, 
and $(\mathcal{T}F)(U)=\mathcal{L}(U)$ for any $\forall U\in \overline{H}_{P,0}^{s}(\mathbb T^n)$. 
Then the proof is complete. 
\end{proof}
In the following discussion, for notational convenience, we introduce the symbol $\hookrightarrow$ to 
denote continuous embedding, and the symbol $\not\hookrightarrow$ to denote that no continuous embedding holds.

\begin{lemma}\label{not_embedding_L2}
For any \(s>0\), $\overline{H}_{P,0}^s(\mathbb T^n)\not\hookrightarrow L^2_{0}(\mathbb T^n)$.  
\end{lemma}

\begin{proof}
We first show that $\boldsymbol{0}\in\mathbb R^d$ is an accumulation point of the set
\begin{eqnarray}\label{Definition_SubGroup}
\Gamma=\left\{ P\boldsymbol{k}:\boldsymbol{k}\in\mathbb Z^n \right\}\subset \mathbb R^d.
\end{eqnarray}

We prove it by contradiction.
Suppose that $\boldsymbol{0}$ is not an accumulation point of $\Gamma$.
Then there exists \(\varepsilon>0\) such that
\begin{eqnarray*}
\Gamma\cap B_\varepsilon(\mathbf 0)=\{\mathbf 0\}.
\end{eqnarray*}
Since \(\Gamma\) is an additive subgroup of \(\mathbb R^d\), translation shows that every point of \(\Gamma\) is isolated.
Hence $\Gamma$ is a discrete subgroup of $\mathbb R^d$.
Since the columns of $P$ are rationally linearly independent, the surjective homomorphism 
$\varphi:\mathbb Z^n\rightarrow \Gamma$ has trivial kernel and is therefore injective, and thus $\mathbb Z^n\cong \Gamma$.
But a discrete subgroup of \(\mathbb R^d\) has rank at most \(d\), whereas \(\mathbb Z^n\) has rank \(n>d\), 
which contradict with $\mathbb Z^n\cong \Gamma$. 

We now prove the existence of a sequence $\{U_j\}$ such that $\|U_j\|_{L_0^2}=1$ and $\|U_j\|_{\overline{H}_{P,0}^s}\rightarrow 0$.
Since $\boldsymbol{0}$ is an accumulation point of $\Gamma$, there exists a sequence $\{\boldsymbol{k}_j\}\subset \mathbb Z^n\setminus\{\boldsymbol 0\} $ such that
\begin{eqnarray*}
|P\boldsymbol{k}_j|\rightarrow 0,\ \ \ {\rm as}\ j\rightarrow\infty.
\end{eqnarray*}
Define
\begin{eqnarray*}
U_j:=e^{{\rm i}\boldsymbol{k}_j^\top y},\ \ \ j\in \mathbb N.
\end{eqnarray*}
Then it is easy to verify that $\|U_j\|_{L_0^2}=1$ and
\begin{eqnarray*}
\|U_j\|_{\overline{H}_{P,0}^s}^2 = |P\boldsymbol{k}_j|^{2s}  \rightarrow 0,\ \ \ {\rm as}\ j\rightarrow\infty.
\end{eqnarray*}
Therefore there is no constant $C>0$ such that 
\begin{eqnarray*}
\|U\|_{L^2_0}\leq C\|U\|_{\overline{H}_{P,0}^s},\ \ \ \forall U\in \overline{H}_{P,0}^s(\mathbb T^n),
\end{eqnarray*}
and hence $\overline{H}_{P,0}^s(\mathbb T^n)\not\hookrightarrow L^2_{0}(\mathbb T^n)$.
This completes the proof. 
\end{proof}
Lemma \ref{not_embedding_L2} shows that, unlike standard Sobolev spaces, the space $\overline{H}_{P,0}^s(\mathbb T^n)$ 
does not admit a continuous embedding into $L^2_{0}(\mathbb T^n)$.
Moreover, for $s\neq t$, it does not continuous embedding from $\overline{H}_{P,0}^s(\mathbb T^n)$ into $\overline{H}_{P,0}^t$.
Nevertheless, the following continuous embedding relation between the corresponding interpolation spaces still holds.
This relation plays a crucial role in the regularity analysis in Section \ref{Section_Regularity}.

\begin{lemma}\label{lemma_cap_embedding}
Assume that $s<t$, for every $r\in[s,t]$, we have the continuous embedding
\begin{eqnarray*}
\overline{H}_{P,0}^{s}(\mathbb T^n)\cap \overline{H}_{P,0}^{t}(\mathbb T^n)\hookrightarrow \overline{H}_{P,0}^{r}(\mathbb T^n),
\end{eqnarray*}
and for any $F\in \overline{H}_{P,0}^{s}(\mathbb T^n)\cap \overline{H}_{P,0}^{t}(\mathbb T^n)$, the following estimate holds
\begin{eqnarray}\label{eq_lemma_cap_embedding}
\|F\|_{\overline{H}_{P,0}^{r}}\leq \|F\|_{\overline{H}_{P,0}^{s}}^{1-\theta} \|F\|_{\overline{H}_{P,0}^{t}}^{\theta},
\end{eqnarray}
where $\theta=\frac{r-s}{t-s}$.
\end{lemma}

\begin{proof}
The lemma clearly holds when $r=s$ or $r=t$.
For the case $r\in (s,t)$, we have 
\begin{eqnarray*}
r = (1-\theta) s+\theta t,
\end{eqnarray*}
and $\theta\in(0,1)$.
Then, using H\"{o}lder's inequality leads to 
\begin{eqnarray*}
\|F\|_{\overline{H}_{P,0}^{r}}^2&=&\sum_{\boldsymbol{k}\in \mathbb Z^n\setminus\{\boldsymbol 0\}} |P\boldsymbol{k}|^{2r}|\widehat{F}_{\boldsymbol{k}}|^2
=\sum_{\boldsymbol{k}\in \mathbb Z^n\setminus\{\boldsymbol 0\}} (|P\boldsymbol{k}|^{2s}|\widehat{F}_{\boldsymbol{k}}|^2)^{1-\theta}  
(|P\boldsymbol{k}|^{2t}|\widehat{F}_{\boldsymbol{k}}|^2)^{\theta}\\
&\leq&\|F\|_{\overline{H}_{P,0}^{s}}^{2(1-\theta)} \|F\|_{\overline{H}_{P,0}^{t}}^{2\theta},
\end{eqnarray*}
which means (\ref{eq_lemma_cap_embedding}) is proved. 

The continuous embedding can be established directly by the weighted 
arithmetic-geometric mean inequality together with the following argument.
\begin{eqnarray*}
\|F\|_{\overline{H}_{P,0}^{r}}^2&\leq& \|F\|_{\overline{H}_{P,0}^{s}}^{2(1-\theta)} \|F\|_{\overline{H}_{P,0}^{t}}^{2\theta}\\
&\leq& (1-\theta) \|F\|_{\overline{H}_{P,0}^{s}}^2 + \theta \|F\|_{\overline{H}_{P,0}^{t}}^2\\
&\leq& \|F\|_{\overline{H}_{P,0}^{s}}^2 + \|F\|_{\overline{H}_{P,0}^{t}}^2 
=: \|F\|_{\overline{H}_{P,0}^{s}\cap \overline{H}_{P,0}^{t}}^2.
\end{eqnarray*}
Then we complete the proof. 
\end{proof}

For convenience in the subsequent regularity analysis of the variational problem, we present the following corollary, 
which follows directly from Lemma \ref{lemma_cap_embedding}.
\begin{corollary}
Under the assumptions of Lemma \ref{lemma_cap_embedding}, the following inequality holds
\begin{eqnarray}\label{eq_corollary_cap_embedding}
\|F\|_{\overline{H}_{P,0}^{r}}\leq \|F\|_{\overline{H}_{P,0}^{s}}+ \|F\|_{\overline{H}_{P,0}^{t}}.
\end{eqnarray}
\end{corollary}

\begin{proof}
This is a direct consequence of \eqref{eq_lemma_cap_embedding} with the mean inequality.
\end{proof}

Since matrix $P\in \mathbb R^{d\times n}$ is a bounded linear operator from $\mathbb R^n$ to $\mathbb R^d$, 
it is also a bounded linear operator from $\mathbb Z^n$ to $\mathbb Z^d$.
Then the following inequality relationship holds
\begin{eqnarray}\label{Pk_leq_k}
|P\boldsymbol{k}|\leq C_{\lambda} |\boldsymbol{k}|,\ \ \ \boldsymbol{k}\in \mathbb Z^n,
\end{eqnarray}
where 
\begin{eqnarray*}
C_{\lambda}:= \sqrt{\lambda_{\max}(P^\top P)},
\end{eqnarray*}
and $\lambda_{\max}(P^\top P)$ denotes the maximum eigenvalue of $P^\top P$.
Obviously, we have the following embedding relations.

\begin{lemma}
For any $s>0$, $\overline{H}_{0}^s(\mathbb T^n)\hookrightarrow   \overline{H}_{P,0}^s(\mathbb T^n)$.
\end{lemma}

\begin{proof}
For any $F\in \overline{H}_{0}^s(\mathbb T^n)$, we have 
\begin{eqnarray*}
\|F\|_{\overline{H}_{P,0}^s}^2&=& \sum_{\boldsymbol{k}\in \mathbb Z^n\setminus\{\boldsymbol 0\}} |P\boldsymbol{k}|^{2s}|\widehat{F}_{\boldsymbol{k}}|^2
\leq C_{\lambda}^{2s}\sum_{\boldsymbol{k}\in \mathbb Z^n\setminus\{\boldsymbol 0\}} |\boldsymbol{k}|^{2s}|\widehat{F}_{\boldsymbol{k}}|^2\\
&\leq& C_{\lambda}^{2s}\sum_{\boldsymbol{k}\in \mathbb Z^n\setminus\{\boldsymbol 0\}} 
(1+|\boldsymbol{k}|)^{2s}|\widehat{F}_{\boldsymbol{k}}|^2 
= C_{\lambda}^{2s}\|F\|_{\overline{H}_{0}^s(\mathbb T^n)}^2,
\end{eqnarray*}
and the proof is complete. 
\end{proof}

However, for $s<0$, it is generally difficult to establish a continuous 
embedding of some standard Sobolev space $\overline{H}_{0}^t$ 
into $ \overline{H}_{P,0}^{s}$.
To obtain such an embedding, we introduce the following assumption.

\begin{assumption}\label{assumption_P_Diophantine}
A projection matrix $P\in \mathbb P$ satisfies the Diophantine condition, 
if there exist constants $c>0$ and $\tau>0$ such 
that for any $\boldsymbol{k}\in \mathbb Z^n\setminus\{\boldsymbol 0\}$, the following inequality holds
\begin{eqnarray}\label{Diophantine_condition_P}
|P \boldsymbol{k}|\geq c|\boldsymbol{k}|^{-\tau}.
\end{eqnarray}
\end{assumption}

This assumption is standard in the study of quasiperiodic problems, and many results in Diophantine 
approximation are available concerning the admissible range of $\tau$.
For example, when $d=1,n=2$, and $P=[1\ \alpha]$, if $\alpha$ is algebraic, then \eqref{Diophantine_condition_P} 
holds for every $\tau>1$.
This is a consequence of Roth's theorem \cite{Roth}.

Under Assumption \ref{assumption_P_Diophantine}, we have the following continuous embedding result.

\begin{lemma}\label{lemma_Diophantine_Embedding}
If projection matrix $P$ satisfies the Diophantine 
condition \eqref{Diophantine_condition_P}, then for any $s>0$, we have
\begin{eqnarray*}
H_{0}^{\tau s}(\mathbb T^n) \hookrightarrow \overline{H}_{P,0}^{-s}(\mathbb T^n).
\end{eqnarray*}
\end{lemma}

\begin{proof}
For any $F\in \overline{H}_{P,0}^{-s}(\mathbb T^n)$, due to \eqref{Diophantine_condition_P}, 
the following inequalities hold
\begin{eqnarray*}
\|F\|_{\overline{H}_{P}^{-s}}^2&=&\sum_{\boldsymbol{k}\in \mathbb Z^n\setminus\{\boldsymbol 0\}}
|P\boldsymbol{k}|^{-2s}|\widehat{F}_{\boldsymbol{k}}|^2
\leq c^{-2}\sum_{\boldsymbol{k}\in \mathbb Z^n\setminus\{\boldsymbol 0\}}|\boldsymbol{k}|^{2\tau s}|\widehat{F}_{\boldsymbol{k}}|^2\\
&\leq& c^{-2}\sum_{\boldsymbol{k}\in \mathbb Z^n\setminus\{\boldsymbol 0\}}
(1+|\boldsymbol{k}|^2)^{\tau s}|\widehat{F}_{\boldsymbol{k}}|^2=c^{-2}\|F\|_{H^{\tau s}}^2.
\end{eqnarray*}
This completes the proof.
\end{proof}

\subsection{Isomorphism between quasiperiodic and periodic function spaces}  
In this subsection, we study the relationship between quasiperiodic function spaces and the associated periodic function spaces.
These spaces are connected by a pullback mapping.
For a trigonometric polynomial
\begin{eqnarray*}
U(\boldsymbol{y})=\sum_{k\in \mathcal K} c_k e^{{\rm i}\boldsymbol{k}^\top  \boldsymbol{y}}, 
\ \ \ \boldsymbol{y}\in \mathbb{T}^n,
\end{eqnarray*}
we define its pullback by
\begin{eqnarray*}
(\mathcal{J}_{P} U)(\boldsymbol{x}):=U(P^\top \boldsymbol{x}),
\ \ \ \boldsymbol{x}\in \mathbb{R}^d.
\end{eqnarray*}
Since
\begin{eqnarray*}
U(P^\top  \boldsymbol{x})
=\sum_{\boldsymbol{k}\in \mathcal{K}} c_{\boldsymbol{k}} 
e^{{\rm i}\boldsymbol{k}^\top (P^\top\boldsymbol{x})}
=\sum_{\boldsymbol{k}\in \mathcal{K}} c_{\boldsymbol{k}} 
e^{{\rm i}(P\boldsymbol{k})^\top \boldsymbol{x}},
\end{eqnarray*}
the pullback mapping \(\mathcal{J}_P\) sends trigonometric polynomials 
on \(\mathbb{T}^n\) into \(\mathcal{T}_P\).
Furthermore, \cite[Theorem 4.1]{JiangLiZhang} shows that 
the Fourier coefficient $U_{\boldsymbol{k}}$ of $U\in L^2(\mathbb T^n)$ 
is equal to the Bohr coefficient 
$u_{\boldsymbol{k}}$ of $u\in B_{P}^2(\mathbb R^d)$ 
for $\boldsymbol{k}\in \mathbb Z^n$, 
and thus $\mathcal J_{P}: L^2(\mathbb T^n) 
\rightarrow B_{P}^2(\mathbb R^d)$ is an isometric isomorphism.
Then, it is easy to derive the following isomorphism
between the periodic function spaces on $\mathbb{T}^n$ 
and the quasiperiodic function spaces on $\mathbb{R}^d$.

\begin{proposition}
The pullback mapping \(\mathcal{J}_P\) extends uniquely by density to the 
following isometric isomorphisms:
\begin{eqnarray*}
\mathcal{J}_P &:& L^2(\mathbb{T}^n)\longrightarrow B_P^2(\mathbb{R}^d),\\
\mathcal{J}_P &:& H^s(\mathbb{T}^n)\longrightarrow H^s(\mathbb R^d),\\
\mathcal{J}_P &:& \overline H_0^s(\mathbb{T}^n)\longrightarrow \overline H_{0}^s(\mathbb R^d),\\
\mathcal{J}_P &:& H_P^s(\mathbb{T}^n)\longrightarrow H_{P}^s(\mathbb R^d),\\
\mathcal{J}_P &:& \overline H_{P,0}^s(\mathbb{T}^n)\longrightarrow \overline H_{P,0}^s(\mathbb R^d).
\end{eqnarray*}
\end{proposition}

The above proposition is important not only for numerical approximation, 
but also for the analysis of quasiperiodic problems.
Through the pullback isomorphism $\mathcal{J}_P$, the original low-dimensional 
quasiperiodic problem can be transformed 
into a high-dimensional periodic problem on $\mathbb T^n$.
As a result, both the theoretical analysis and the numerical approximation 
can be carried out in the periodic setting, 
where standard tools from Fourier analysis and periodic PDE theory are available.
This provides a unified and often simpler framework than working directly 
with the quasiperiodic formulation in $\mathbb R^d$.

\section{Quasiperiodic elliptic problems}\label{Section_Regularity}
In this section, we introduce the quasiperiodic elliptic problem under consideration and the resulting 
high-dimensional periodic problem obtained by the projection method. Furthermore, the corresponding 
regularity results and error estimates for the Fourier series subspace approximation method are established. 
\subsection{Low-dimensional quasiperiodic elliptic problems}
In this paper, we are concerned with the quasiperiodic elliptic equation: 
Find quasiperiodic solution $u$ such that 
\begin{eqnarray}\label{Elliptic_Equation}
\mathcal{L} u(\boldsymbol{x}) = f(\boldsymbol{x}),\ \boldsymbol{x} \in \mathbb{R}^d,
\end{eqnarray}
where the second order elliptic operator $\mathcal{L}: \overline{H}_{P,0}^2(\mathbb{R}^d) 
\rightarrow B_{P,0}^2(\mathbb{R}^d)$ is defined as
\begin{eqnarray*}
\mathcal{L} u(\boldsymbol{x}) = -\mathrm{div}(\alpha(\boldsymbol{x}) \nabla u(\boldsymbol{x})).
\end{eqnarray*}
Here, the source term $f(\boldsymbol{x}) \in B_{P,0}^2(\mathbb{R}^d)$. 
Furthermore, the quasiperiodic coefficients $\alpha(\boldsymbol{x})$ is uniformly elliptic, that is, 
there exists two constants  $\alpha_0>0$ and $\alpha_1>0$ such that 
for all $\boldsymbol{x} \in \mathbb{R}^d$
\begin{eqnarray}\label{Lower_Upper_Bound_alpha}
0<\alpha_0 \leq \alpha(\boldsymbol{x}) \leq \alpha_1<\infty. 
\end{eqnarray}

By multiplying \eqref{Elliptic_Equation} with an arbitrary test function 
$v \in \overline{H}_{P,0}^{1}(\mathbb{R}^d)$ and integrating over $\mathbb{R}^d$ 
with integration by parts, taking into account that the boundary integral vanishes 
(see \cite[Lemma 3.4]{JiangLiZhang_2024}) 
leads to the variational formulation associated with \eqref{Elliptic_Equation}: 
Find $u \in \overline{H}_{P,0}^{1}(\mathbb{R}^d)$ such that
\begin{equation}\label{Weak_Form1}
\mathcal B(u, v) = (f, v),\ \forall v \in \overline{H}_{P,0}^{1}(\mathbb{R}^d),
\end{equation}
where the bilinear form $\mathcal{B}(\cdot, \cdot): \overline{H}_{P,0}^{1}(\mathbb{R}^d) 
\times \overline{H}_{P,0}^{1}(\mathbb{R}^d) \to \mathbb{R}$ is defined as
\begin{eqnarray*}
\mathcal{B}(u, v) := (\alpha \nabla u, \nabla v),
\ \ \ \forall u, v \in \overline{H}_{P,0}^{1}(\mathbb{R}^d).
\end{eqnarray*}
In Subsection \ref{sec_high}, we study the existence and uniqueness of the corresponding 
high-dimensional periodic problem associated with \eqref{Weak_Form1}.
Then the corresponding properties of the solution to variational problem \eqref{Weak_Form1} 
are immediate consequences of the corresponding properties 
for the high-dimensional variational problem and the pullback isomorphism $\mathcal{J}_P$.

\subsection{Projection method and the high-dimensional periodic problem}\label{sec_high}
By using the pullback mapping $\mathcal{J}_P$, we can analyze the properties 
of the periodic variational problem on $\mathbb T^n$ 
and then pull back to the quasiperiodic problem on $\mathbb R^d$.
For the function on $\mathbb T^n$, we take the corresponding symbols as follows
\begin{eqnarray}
U(P^\top  \boldsymbol x) = u(\boldsymbol x), \ \  
F(P^\top  \boldsymbol x) = f(\boldsymbol x),\ \  
A(P^\top  \boldsymbol x) = \alpha(\boldsymbol x),
\end{eqnarray}
where $A(\boldsymbol{y})$ has the following upper and lower bounds
\begin{eqnarray}\label{bound_A}
0<\alpha_0\leq A(\boldsymbol{y})\leq \alpha_1<\infty,\ \ \ \forall\boldsymbol{y}\in \mathbb T^n.
\end{eqnarray}

Consider the high-dimensional periodic variational problem:
Find $U\in \overline H_{P,0}^1(\mathbb  T^n)$ such that 
\begin{eqnarray}\label{Projection_Problem}
a(U,V) = \langle F,V \rangle,\ \ \ \ \forall V\in \overline H_{P,0}^1(\mathbb  T^n),
\end{eqnarray}
where the bilinear form $a(\cdot,\cdot)$ is defined as follows 
\begin{eqnarray}\label{eq_bilinear_Tn}
a(U,V) = (AP\nabla U, P\nabla V).
\end{eqnarray}
Based on the definition of the bilinear $a(\cdot,\cdot)$, we can define 
a type of norm as follows
\begin{eqnarray}
\|V\|_a = \sqrt{a(V,V)},\ \ \ \ \forall V\in \overline H^1_{P,0}(\mathbb T^n). 
\end{eqnarray}
Through single calculation, the bilinear form defined by \eqref{eq_bilinear_Tn} satisfies the coercivity and continuity on $\overline H^1_{P,0}(\mathbb  T^n)$
\begin{eqnarray}
a(U,U)&\geq& \alpha_0\| U\|^2_{\overline H^1_{P,0}},\ \ \ \forall U\in \overline H^1_{P,0}(\mathbb T^n), \label{Ellipticity_a_Periodic} \\
a(U,V)&\leq& \alpha_1\| U\|_{\overline H^1_{P,0}}  \| V\|_{\overline H^1_{P,0}}, 
\ \ \ \forall U\in \overline H^1_{P,0}(\mathbb T^n),\ \ \forall V\in \overline H^1_{P,0}(\mathbb T^n),\label{Boundedness_a_Periodic}
\end{eqnarray}
where the constants $\alpha_0$ and $C$ come from (\ref{Lower_Upper_Bound_alpha}).
Then, the following existence and uniqueness result holds for the variational problem \eqref{Projection_Problem}.
\begin{lemma}\label{exist_unique_simple}
Let $A$ satisfies the upper and lower bound condition \eqref{bound_A}.
Then for any $F\in \overline{H}_{P,0}^{-1}(\mathbb T^n)$, 
the variational problem \eqref{Projection_Problem} 
has a unique solution $U\in \overline{H}_{P,0}^{1}(\mathbb T^n)$. 
Moreover, the following estimate holds
\begin{eqnarray*}
\|U\|_{\overline{H}_{P,0}^{1}}\leq \alpha_0^{-1} \|F\|_{\overline{H}_{P,0}^{-1}}.
\end{eqnarray*}
\end{lemma}

\begin{proof}
Since $a(\cdot,\cdot)$ is coercive and continuous on $\overline{H}_{P,0}^1 (\mathbb T^n)$, 
the existence and uniqueness of the solution are immediate by the Lax-Milgram theorem.
And the upper estimate follows by taking $V=U$ as a test function in \eqref{Projection_Problem}.
\end{proof}

In addition to the solution itself, we are also concerned with its 
finite-dimensional approximation for the problem (\ref{Projection_Problem}). 
For this aim, let us define the finite-dimensional subspace spanned by the Fourier basis functions
\begin{eqnarray*}
X_{K} = {\rm span}\left\{ e^{{\rm i}\boldsymbol{k}^\top \boldsymbol{y}}:0<|\boldsymbol{k}|\leq K  \right\},
\end{eqnarray*}
together with the associated interpolation operator 
$\Pi_K:\overline{H}_{P,0}^s(\mathbb T^n)\rightarrow X_{K}$ by
\begin{eqnarray*}
\Pi_K U=\sum_{0<|\boldsymbol{k}|\leq K} \widehat{U}_{\boldsymbol{k}} e^{{\rm i}\boldsymbol{k}^\top \boldsymbol{y}},\ \ \ 
{\rm for}\ U=\sum_{\boldsymbol{k}\in \mathbb Z^n\setminus\{\boldsymbol 0\}} 
\widehat{U}_{\boldsymbol{k}} e^{{\rm i}\boldsymbol{k}^\top \boldsymbol{y}}\in \overline{H}_{P,0}^s(\mathbb T^n).
\end{eqnarray*}
Based on the finite-dimensional subspace $X_K$, 
the standard Galerkin approximation scheme for the variational 
problem \eqref{Projection_Problem} can be defined as: Find $U_K\in X_K$ such that
\begin{eqnarray}\label{Projection_Problem_Galerkin}
a(U_K,V_K) = (F,V_K),\ \ \ \ \forall V_K\in X_K.
\end{eqnarray}
The following lemma establishes the strong convergence of the Galerkin approximation $U_K$ to $U$ 
in the $\overline{H}_{P,0}^1(\mathbb T^n)$-norm. 
\begin{lemma}
Let $U$ and $U_K$ denote the solutions of \eqref{Projection_Problem} and \eqref{Projection_Problem_Galerkin}, respectively.
Then $U_K=\Pi_K U$. In addition,
\begin{eqnarray*}
\|U-U_K\|_{\overline{H}_{P,0}^1}\rightarrow 0,\ \ \ {\rm as}\ K\rightarrow\infty.
\end{eqnarray*}
\end{lemma}
\begin{proof}
Choosing $V=e^{{\rm i}\boldsymbol{k}^\top \boldsymbol{y}}$ in \eqref{Projection_Problem} and $V_K=e^{{\rm i}\boldsymbol{k}^\top \boldsymbol{y}}$ 
in \eqref{Projection_Problem_Galerkin} for $0<|\boldsymbol{k}|\leq K$, 
and then comparing the resulting identities, we obtain
\begin{eqnarray*}
\widehat{U}_{\boldsymbol{k}}=(\widehat{U_K})_{\boldsymbol{k}}, \ \ \ 0<|\boldsymbol{k}|\leq K.
\end{eqnarray*}
Hence $U_K=\Pi_K U$ and 
\begin{eqnarray*}
\|U-U_K\|_{\overline{H}_{P,0}^1}^2 = \|U-\Pi_K U\|_{\overline{H}_{P,0}^1}^2
=\sum_{|\boldsymbol{k}|>K}|P\boldsymbol{k}|^{-2} |\widehat{U}_{\boldsymbol{k}}|^2\rightarrow 0,\ \ \ {\rm as}\ K\rightarrow\infty.
\end{eqnarray*}
This completes the proof. 
\end{proof}

\begin{remark}\label{Remark_1}
Note that both the variational problem \eqref{Projection_Problem} 
and its corresponding Galerkin approximation 
problem (\ref{Projection_Problem_Galerkin}) 
are studied in the space $\overline{H}_{P,0}^1(\mathbb T^n)$. 
The key reason for doing this is exactly the statement of Lemma \ref{not_embedding_L2}: 
$\overline{H}_{P,0}^1(\mathbb T^n)$ cannot be continuously embedded into $L_0^2(\mathbb T^n)$.
Moreover, it indicates that the Poincar\'{e} inequality does not hold, 
and the eigenvalues of the following problem accumulate at $\boldsymbol{0}$: 
Find $(\lambda,U)\in \mathbb R\times \overline{H}_{P,0}^1(\mathbb T^n)$ such that
\begin{eqnarray*}
(P\nabla U,P\nabla V)=\lambda (U,V),
\ \ \ \forall V\in  \overline H^1_{P,0}(\mathbb T^n).
\end{eqnarray*}
\end{remark}

It follows from Lemma \ref{not_embedding_L2} that 
the solution $U\in \overline H^1_{P,0}(\mathbb T^n)$ 
of the variational problem \eqref{Projection_Problem} 
does not necessarily belong to $L_0^2(\mathbb T^n)$.
This naturally motivates us to investigate what conditions 
on $A$ and $F$ can ensure that the solution $U$ admits further regularity.
For convenience, we first introduce the following 
constant-coefficient directional derivatives
\begin{eqnarray*}
Y_j:=p_j\cdot\nabla,\ \ \ j=1,\cdots d,
\end{eqnarray*}
where $p_j\in \mathbb R^n$ denotes the $j$-th row of the matrix $P$.
We also introduce the multi-index notation $\alpha =(\alpha_1,\cdots,\alpha_d)\in \mathbb N_0^d$, $|\alpha|=\alpha_1+\cdots+\alpha_d$, 
$\alpha!=\alpha_1!\cdots\alpha_d!$ and $Y^{\alpha}:=Y^{\alpha_1}_1\cdots Y^{\alpha_d}_d$.

Associated with the directional derivatives $Y^\alpha$, we introduce 
the following norm for $r\geq 0$, and $s\in\mathbb R$
\begin{eqnarray*}
\|U\|_{Y^{\alpha,r}_s}^2&:=& \sum_{|\alpha|=r} \|Y^\alpha U\|_{\overline{H}_{P,0}^s}^2.
\end{eqnarray*}
This norm enjoys the properties stated in the following lemma.
\begin{lemma}\label{lemma_direct_norm_equal}
The norm $\|\cdot\|_{Y^{\alpha,r}_s}$ and $\|\cdot\|_{\overline{H}_{P,0}^r}$ are equivalent. 
More precisely, we have
\begin{eqnarray*}
(r!)^{-\frac{1}{2}} \|U\|_{\overline{H}_{P,0}^{r+s}} \leq \|U\|_{Y^{\alpha,r}_s}\leq \|U\|_{\overline{H}_{P,0}^{r+s}}.
\end{eqnarray*}
\end{lemma}

\begin{proof}
For any $\boldsymbol{x}\in \mathbb R^d$, and any integer $r\geq 0$, by the polynomial expansion, we obtain
\begin{eqnarray*}
|\boldsymbol{x}|^{2r}=\sum_{|\alpha|=r}\frac{r!}{\alpha!}|\boldsymbol{x}^{\alpha}|^2.
\end{eqnarray*}
Then it is easy to have 
\begin{eqnarray*}
\frac{1}{r!}|\boldsymbol{x}|^{2r}\leq \sum_{|\alpha|=r}|\boldsymbol{x}^{\alpha}|^2\leq |\boldsymbol{x}|^{2r}.
\end{eqnarray*}
Taking $\boldsymbol{x}=P\boldsymbol{k}$ leads to the following inequality
\begin{eqnarray*}
\frac{1}{r!} \|U\|_{\overline{H}_{P,0}^{r+s}}^2 &\leq& \|U\|_{Y^{\alpha,r}_s}^2
= \sum_{|\alpha|=r}\|Y^{\alpha} U\|_{\overline{H}_{P,0}^s}^2\\
&=& \sum_{\boldsymbol{k}\in \mathbb Z^n\setminus\{\boldsymbol 0\}}|P\boldsymbol{k}|^{2s} 
\Big(\sum_{|\alpha|=r}|(P\boldsymbol{k})^{\alpha}|^2 \Big) |\widehat{U}_{\boldsymbol{k}}|^2
\leq \|U\|_{\overline{H}_{P,0}^{r+s}}^2.
\end{eqnarray*}
This completes the proof of the lemma.
\end{proof}

In order to derive higher regularity for $U$, 
in addition to assuming that $A$ is bounded above and below by positive constants, 
we impose the following additional assumption.

\begin{assumption}\label{assume_A}
Assume that there exists $N\geq 0$ such that the Fourier coefficients of $A$ 
satisfy the following decay rate 
\begin{eqnarray*}
|\widehat{A}_{\boldsymbol{k}}|\leq C_{A}(1+|\boldsymbol{k}|)^{-N},\ \ \ \forall \boldsymbol{k}\in\mathbb Z^n,
\end{eqnarray*}
where $C_{A}$ is a positive constant independent of $\boldsymbol{k}$.
\end{assumption}

Assumption \ref{assume_A} implies that 
the directional derivatives of $A$ can be bounded in terms of $N$.
More precisely, we have the following lemma.

\begin{lemma}\label{lemma_bounded_Y_alpha_A}
Let $A$ satisfy Assumption \ref{assume_A}, and suppose that $N>n+m+1$.
Then, for any $\beta\in \mathbb N_0^d$ with $|\beta|\leq m+1$, we have
\begin{eqnarray*}
\|Y^{\beta}A\|_{L^\infty}\leq C_{\lambda,A,|\beta|}<\infty,
\end{eqnarray*}
where
\begin{eqnarray*}
C_{\lambda,A,|\beta|}:=C_{\lambda}^{|\beta|} C_A\sum_{\boldsymbol{k}\in\mathbb Z^n}(1+|\boldsymbol{k}|)^{|\beta|-N},\ \ \ |\beta|<N-n.
\end{eqnarray*}
\end{lemma}

\begin{proof}
By a direct computation with the Fourier expansion, we obtain
\begin{eqnarray*}
Y^{\beta}A=\sum_{\boldsymbol{k}\in\mathbb Z^n} {\rm i}^{|\beta|} (P\boldsymbol{k})^{\beta} 
\widehat{A}_{\boldsymbol{k}} e^{{\rm i}\boldsymbol{k}^\top \boldsymbol{y}}.
\end{eqnarray*}
Combining with \eqref{Pk_leq_k}, Assumption \ref{assume_A} and $N-|\beta|>n$ leads to 
the following inequalities 
\begin{eqnarray*}
\|Y^{\beta}A\|_{L^\infty}&\leq& \sum_{\boldsymbol{k}\in\mathbb Z^n} |P\boldsymbol{k}|^{|\beta|} |\widehat{A}_{\boldsymbol{k}}|
\leq C_{\lambda}^{|\beta|} \sum_{\boldsymbol{k}\in\mathbb Z^n} |\boldsymbol{k}|^{|\beta|} |\widehat{A}_{\boldsymbol{k}}|\\
&\leq& C_{\lambda}^{|\beta|} \sum_{\boldsymbol{k}\in\mathbb Z^n} (1+|\boldsymbol{k}|)^{|\beta|} |\widehat{A}_{\boldsymbol{k}}|
\leq C_{\lambda}^{|\beta|} C_A \sum_{\boldsymbol{k}\in\mathbb Z^n} (1+|\boldsymbol{k}|)^{|\beta|-N}<\infty.
\end{eqnarray*}
This completes the proof of the lemma.
\end{proof}

\begin{theorem}\label{integer_regularity}
Let $A$ satisfies the upper and lower bound condition \eqref{bound_A} 
as well as Assumption \ref{assume_A}, and suppose that $m\in \mathbb N$ and $N>n+m+1$.
Then for every $F\in \overline{H}_{P,0}^{-1}(\mathbb T^n)\cap \overline{H}_{P,0}^{m}(\mathbb T^n)$, 
the variational problem \eqref{Projection_Problem} 
has a unique solution $U\in \overline{H}_{P,0}^{1}(\mathbb T^n)
\cap \overline{H}_{P,0}^{m+2}(\mathbb T^n)$.
Moreover, the following estimate holds
\begin{eqnarray*}
\|U\|_{\overline{H}_{P,0}^{m+2}}\leq C_m \left(\|F\|_{\overline{H}_{P,0}^{-1}} 
+ \|F|_{\overline{H}_{P,0}^{m}}\right),
\end{eqnarray*}
where 
\begin{eqnarray*}
C_m=\alpha_0^{-1}\sqrt{(m+1)!} \sqrt{d} M_m,
\end{eqnarray*}
and 
\begin{eqnarray*}
M_0 &=& \max\{\alpha_0^{-1}C_{\lambda,A,1}, 1\},\\
M_m&=&\max\left\{\alpha_0^{-1}C_{\lambda,A,m+1}+2\sum_{j=1}^m\sqrt{\binom{j+d-1}{d-1}} C_{\lambda,A,m+1-j}C_{j-1},\right.\\
&&\quad\quad\quad \left. 1+\sum_{j=1}^m\sqrt{\binom{j+d-1}{d-1}} C_{\lambda,A,m+1-j}C_{j-1}\right\},
\ \ \ m\geq 1. 
\end{eqnarray*}
\end{theorem}

\begin{proof}
The proof is divided into two steps. 
In Step 1, we establish uniform $\overline{H}_{P,0}^{r+2}$-estimates 
for the Galerkin approximation $U_K$ by induction on $r$. 
In Step 2, we pass to the limit as $K\to\infty$ and identify the limit with the solution $U$.

\medskip
\noindent\textbf{Step 1. Uniform estimates for $U_K$.}
Since the constant-coefficient directional derivatives $Y^\alpha$ commute with $\nabla$, Leibniz's formula gives
\begin{eqnarray}\label{eq_Leibniz}
(AP\nabla (Y^{\alpha} U_K),P\nabla V_K)=\langle Y^{\alpha} F, V_K \rangle-\sum_{0<\beta\leq \alpha} 
\big((Y^{\beta} A) P\nabla(Y^{\alpha-\beta}U_K),P\nabla V_K \big).
\end{eqnarray}
The it is desired to prove 
for each $r=0,1,\dots,m$,
\begin{eqnarray}\label{induction_conclusion}
\|U_K\|_{\overline{H}_{P,0}^{r+2}}\leq C_r \left(\|F\|_{\overline{H}_{P,0}^{-1}} 
+ \|F|_{\overline{H}_{P,0}^{r}} \right).
\end{eqnarray}
We will prove (\ref{induction_conclusion}) by induction on $r$.

We begin with the case $r=0$.
Let $|\alpha|=1$, and choose $V_K=Y^{\alpha}U_K$ in \eqref{eq_Leibniz}.
Then the coercivity of the bilinear form implies 
\begin{eqnarray*}
\alpha_0\|Y^{\alpha}U_K\|_{\overline{H}_{P,0}^1}^2 \leq \|Y^{\alpha}F\|_{\overline{H}_{P,0}^{-1}} \|Y^{\alpha}U_K\|_{\overline{H}_{P,0}^1}
+ \|Y^\alpha A\|_{L^\infty} \|U_K\|_{\overline{H}_{P,0}^1} \|Y^{\alpha}U_K\|_{\overline{H}_{P,0}^1}.
\end{eqnarray*}
Hence, after simplification, the above inequalities together with 
Lemma \ref{lemma_bounded_Y_alpha_A} yields 
\begin{eqnarray}\label{induction_r_0_pre}
\|Y^{\alpha}U_K\|_{\overline{H}_{P,0}^1} \leq \alpha_0^{-1}
\left(\|Y^{\alpha}F\|_{\overline{H}_{P,0}^{-1}} 
+ C_{\lambda,A,1}\|U_K\|_{\overline{H}_{P,0}^1}\right).
\end{eqnarray}
By Lemma \ref{lemma_direct_norm_equal}, we have $\|Y^{\alpha}F\|_{\overline{H}_{P,0}^{-1}}\leq \|F\|_{L^2}$.
Combining this with Lemma \ref{exist_unique_simple} and \eqref{induction_r_0_pre} 
leads to the following estimates 
\begin{eqnarray}\label{induction_r_0}
\|Y^{\alpha}U_K\|_{\overline{H}_{P,0}^1} \leq \alpha_0^{-1} 
\max\{\alpha_0^{-1}C_{\lambda,A,1}, 1\}( \|F\|_{\overline{H}_{P,0}^{-1}} + \|F\|_{L^2} ).
\end{eqnarray}
After squaring both sides of \eqref{induction_r_0} 
and summing over $\alpha$ with $|\alpha|=1$, it follows from Lemma \ref{lemma_direct_norm_equal} that
\begin{eqnarray*}
\|U_K\|_{\overline{H}_{P,0}^{2}}\leq \alpha_0^{-1} \sqrt{d} \max\left\{\alpha_0^{-1}C_{\lambda,A,1}, 1\right\} 
\left(\|F\|_{\overline{H}_{P,0}^{-1}} + \|F\|_{L^2}\right).
\end{eqnarray*}
Thus, the desired result is valid for the case $r=0$.

Assume that the estimate \eqref{induction_conclusion} holds for $r=0,1,\dots,m-1$.
We now prove that it also holds for $r=m$.
Let $|\alpha|=m+1$, and again choose $V_K=Y^{\alpha}U_K$ in \eqref{eq_Leibniz}.
Then
\begin{eqnarray*}
\alpha_0 \|Y^{\alpha}U_K\|_{\overline{H}_{P,0}^1}^2\leq \|Y^{\alpha}F\|_{\overline{H}_{P,0}^{-1}} \|Y^{\alpha}U_K\|_{\overline{H}_{P,0}^1}
+ \sum_{0<\beta\leq\alpha} \|Y^{\beta}A\|_{L^\infty} \| Y^{\alpha-\beta} U_K\|_{\overline{H}_{P,0}^{1}} \|Y^{\alpha}U_K\|_{\overline{H}_{P,0}^1}.
\end{eqnarray*}
By Lemma \ref{lemma_bounded_Y_alpha_A}, after simplification and rearrangement, we obtain
\begin{eqnarray}\label{induction_r_m_pre}
&&\|Y^{\alpha}U_K\|_{\overline{H}_{P,0}^1}\leq \alpha_0^{-1} 
\left(\|Y^{\alpha}F\|_{\overline{H}_{P,0}^{-1}} 
+ \sum_{0<\beta\leq\alpha} C_{\lambda,A,|\beta|} 
\| Y^{\alpha-\beta} U_K\|_{\overline{H}_{P,0}^{1}}\right)\nonumber\\
&&=\alpha_0^{-1} \left( \|Y^{\alpha}F\|_{\overline{H}_{P,0}^{-1}} 
+ C_{\lambda,A,m+1}\|U_K\|_{\overline{H}_{P,0}^{1}} 
+ \sum_{0<\beta<\alpha} C_{\lambda,A,|\beta|} 
\| Y^{\alpha-\beta} U_K\|_{\overline{H}_{P,0}^{1}}\right). 
\end{eqnarray}
The first term on the right-hand side of \eqref{induction_r_m_pre} admits the upper bound
\begin{eqnarray*}
\|Y^{\alpha}F\|_{\overline{H}_{P,0}^{-1}}^2 = \sum_{\boldsymbol{k}\in\mathbb Z^n\setminus\{\boldsymbol 0\}} 
|P\boldsymbol{k}|^{-2} |(P\boldsymbol{k})^{\alpha}|^2 |\widehat{F}_{\boldsymbol{k}}|^2 
\leq C \sum_{\boldsymbol{k}\in\mathbb Z^n\setminus\{\boldsymbol 0\}} |P\boldsymbol{k}|^{2r} |\widehat{F}_{\boldsymbol{k}}|^2 = C \|F\|_{\overline{H}_{P,0}^m}^2.
\end{eqnarray*}
Since $|\alpha|=m+1$, the condition $0<\beta<\alpha$ implies that $0<|\alpha-\beta|\leq m$.
Therefore, the last term on the right-hand side of \eqref{induction_r_m_pre} satisfies
\begin{eqnarray*}
\sum_{0<\beta<\alpha} C_{\lambda,A,|\beta|} \| Y^{\alpha-\beta} U_K\|_{\overline{H}_{P,0}^{1}}
&=&\sum_{0< |\gamma|\leq m} C_{\lambda,A,m+1-|\gamma|} \| Y^{\gamma} U_K\|_{\overline{H}_{P,0}^{1}}\\
&=&\sum_{j=1}^m \sum_{|\gamma|=j} C_{\lambda,A,m+1-|\gamma|} \|Y^{\gamma} U_K\|_{\overline{H}_{P,0}^{1}}\\
&\leq& \sum_{j=1}^m \left(\sum_{|\gamma|=j} C_{\lambda,A,m+1-|\gamma|}^2 \right)^{\frac{1}{2}}  
\left(\sum_{|\gamma|=j} \|Y^{\gamma} U_K\|_{\overline{H}_{P,0}^{1}}^2 \right)^{\frac{1}{2}}\\
&\leq& \sum_{j=1}^m \sqrt{\binom{j+d-1}{d-1}} C_{\lambda,A,m+1-j} \|U\|_{\overline{H}_{P,0}^{j+1}}.
\end{eqnarray*}
By the induction hypothesis and Lemma \ref{lemma_cap_embedding} for $j=1,\cdots,m$, we have
\begin{eqnarray*}
\|U_K\|_{\overline{H}_{P,0}^{j+1}} &\leq& C_{j-1} 
\left(\|F\|_{\overline{H}_{P,0}^{-1}} + \|F\|_{\overline{H}_{P,0}^{j-1}}\right)\\
&\leq& C_{j-1} \left(2\|F\|_{\overline{H}_{P,0}^{-1}} + \|F\|_{\overline{H}_{P,0}^{m}}\right).
\end{eqnarray*}
Accordingly, \eqref{induction_r_m_pre} may be rearranged to the following
form
\begin{eqnarray}\label{induction_r_m}
\|Y^{\alpha}U_K\|_{\overline{H}_{P,0}^1}\leq \alpha_0^{-1} M_m 
\left(\|F\|_{\overline{H}_{P,0}^{-1}} + \|F|_{\overline{H}_{P,0}^{m}}\right).
\end{eqnarray}
Squaring both sides of inequality \eqref{induction_r_m} and summing 
over $\alpha$ with $|\alpha|=m+1$, it follows from Lemma \ref{lemma_direct_norm_equal} that
\begin{eqnarray*}
\|U_K\|_{\overline{H}_{P,0}^{m+2}}\leq \alpha_0^{-1} \sqrt{(m+1)!} M_m 
\left(\|F\|_{\overline{H}_{P,0}^{-1}} + \|F|_{\overline{H}_{P,0}^{m}}\right).
\end{eqnarray*}
This proves \eqref{induction_conclusion} for $r=m$, and hence completes the induction step.

\medskip
\noindent\textbf{Step 2. Passage to the limit as $K\rightarrow \infty$.}
By Step 1, the sequence $\{U_K\}$ is bounded in $\overline{H}_{P,0}^{m+2}(\mathbb T^n)$. 
Hence there exist a subsequence, still denoted by $\{U_K\}$, and some $\widetilde U\in \overline{H}_{P,0}^{m+2}(\mathbb T^n)$ such that
\begin{eqnarray}\label{weakly_convergence}
U_{K}\rightharpoonup \widetilde{U},\ \ \ {\rm weakly\ in}\ \overline{H}_{P,0}^{m+2}(\mathbb T^n).
\end{eqnarray}
On the other hand, we already know that
\begin{eqnarray}\label{strongly_convergence}
U_K\rightarrow U,\ \ \ {\rm strongly\ in}\ \overline{H}_{P,0}^{1}(\mathbb T^n).
\end{eqnarray}
It remains to identify $\widetilde U$ with $U$.
For each fixed $\boldsymbol{k}\in\mathbb Z^n\setminus\{0\}$, let us define
\begin{eqnarray*}
L_{\boldsymbol{k}}(V):=\widehat{V}_{\boldsymbol{k}}.
\end{eqnarray*}
Since
\begin{eqnarray*}
|\widehat{V}_{\boldsymbol{k}}|\leq |P\boldsymbol{k}|^{-(m+2)} \|V\|_{\overline{H}_{P,0}^{m+2}},
\end{eqnarray*}
the functional $L_{\boldsymbol{k}}$ is continuous on $\overline{H}_{P,0}^{m+2}(\mathbb T^n)$. 
Therefore, the weak convergence in $\overline{H}_{P,0}^{m+2}$ indicates that
\begin{eqnarray*}
(\widehat{U_K})_{\boldsymbol{k}} \to \widehat{\widetilde U}_{\boldsymbol{k}}.
\end{eqnarray*}
Moreover, $L_{\boldsymbol{k}}$ is also continuous on $\overline{H}_{P,0}^{1}(\mathbb T^n)$, so the strong convergence in $\overline{H}_{P,0}^{1}$ yields
\begin{eqnarray*}
(\widehat{U_K})_{\boldsymbol{k}} \to \widehat U_{\boldsymbol{k}}.
\end{eqnarray*}
Consequently, $\widehat{\widetilde U}_{\boldsymbol{k}}=\widehat U_{\boldsymbol{k}}$ for every $\boldsymbol{k}\in\mathbb Z^n\setminus\{0\}$, and hence $\widetilde U=U$.
Thus $U\in \overline{H}_{P,0}^{m+2}(\mathbb T^n)$.

Finally, by weak lower semicontinuity of the norm,
\begin{eqnarray*}
\|U\|_{\overline{H}_{P,0}^{m+2}}= \liminf_{j\rightarrow\infty} \|U_{K_j}\|_{\overline{H}_{P,0}^{m+2}} 
\leq C_m \left(\|F\|_{\overline{H}_{P,0}^{-1}} + \|F|_{\overline{H}_{P,0}^{m}}\right).
\end{eqnarray*}
This completes the proof.
\end{proof}

In the proof of Theorem \ref{integer_regularity}, the case where $m$ is an integer was established.
By interpolation space theory, the corresponding regularity result for non-integer regularity can be obtained, 
as stated in the following corollary.

\begin{corollary}\label{real_regularity}
Under the assumptions of Theorem \ref{integer_regularity}, let $s\geq 0$.
If $F\in \overline{H}_{P,0}^{-1}(\mathbb T^n)\cap \overline{H}_{P,0}^{s}(\mathbb T^n)$, then the unique variational 
solution $U\in H_{P,0}^{1}(\mathbb T^n)$ of \eqref{Projection_Problem} satisfies $U\in \overline{H}_{P,0}^{s+2}(\mathbb T^n)$.
\end{corollary}

\begin{proof}
The corollary follows from the previously established integer-order regularity in 
Theorem \ref{integer_regularity} and the Stein-Weiss interpolation 
theorem \cite[Thoerem 5.4.1]{Interpolation}.

Indeed, under the Fourier characterization,
\begin{eqnarray*}
\overline{H}_{P,0}^r(\mathbb T^n) \cong \ell^2(\mathbb Z^n\setminus\{\boldsymbol 0\},|P\boldsymbol{k}|^{2r}),\ \ \ \forall r\in\mathbb R.
\end{eqnarray*}
And for any $s\not\in \mathbb N_0$, denote $m=\lfloor s \rfloor\in\mathbb N_0$ and $\theta=s-m\in(0,1)$, we have 
\begin{eqnarray*}
s=(1-\theta)m+\theta (m+1),
\end{eqnarray*}
and
\begin{eqnarray*}
|P\boldsymbol{k}|^{2s}=(|P\boldsymbol{k}|^{2m})^{1-\theta}(|P\boldsymbol{k}|^{2(m+1)})^{\theta}.
\end{eqnarray*}
Hence,
\begin{eqnarray*}
(\overline{H}_{P,0}^{m}, \overline{H}_{P,0}^{m+1})_{\theta,2}=\overline{H}_{P,0}^{s},\ \ \ 
(\overline{H}_{P,0}^{m+2}, \overline{H}_{P,0}^{m+3})_{\theta,2}=\overline{H}_{P,0}^{s+2}.
\end{eqnarray*}
Let $T$ denote the solution operator, namely $U=TF$.
By Theorem \ref{integer_regularity}, $T$ is bounded from $\overline{H}_{P,0}^{m}(\mathbb T^n)$ to $\overline{H}_{P,0}^{m+2}(\mathbb T^n)$ 
and from $\overline{H}_{P,0}^{m+1}(\mathbb T^n)$ to $\overline{H}_{P,0}^{m+3}(\mathbb T^n)$.
Therefore, the interpolation theorem yields that $T$ is also bounded from $\overline{H}_{P,0}^{s}(\mathbb T^n)$ to $\overline{H}_{P,0}^{s+2}(\mathbb T^n)$. 
Consequently, for every
\begin{eqnarray*}
F\in \overline{H}_{P,0}^{-1}(\mathbb T^n)\cap \overline{H}_{P,0}^{s}(\mathbb T^n),
\end{eqnarray*}
the corresponding variational solution satisfies
\begin{eqnarray*}
U=TF\in \overline{H}_{P,0}^{s+2}(\mathbb T^n).
\end{eqnarray*}
This completes the proof.
\end{proof}
Theorem \ref{integer_regularity} and Corollary \ref{real_regularity} 
yield regularity improvement results for the variational problem \eqref{Projection_Problem}.
However, if the solution possesses only $\overline{H}_{P,0}^s$-regularity, this is still not sufficient for numerical approximation.
The following theorem makes this precise by showing that the convergence may be arbitrarily slow.
\begin{theorem}\label{bad_lower_bound}
For any monotonically decreasing sequence $\{\eta_m\}_{m=1}^{\infty}$ that tends to $0$, 
there exists $U\in \overline{H}_{P,0}^s(\mathbb T^n)$ such that for $K\geq 1$, the following inequality holds
\begin{eqnarray*}
\|U-\Pi_{K} U\|_{\overline{H}_{P}^t}\geq \eta_{K},
\end{eqnarray*}
where $s,t\in\mathbb R$.
\end{theorem}

\begin{proof}
First, we prove that the theorem holds for $s\leq t$.
Define a sequence $\{a_m\}_{m=1}^{\infty}$ as follows
\begin{eqnarray*}
a_m=\eta_m^2-\eta_{m+1}^2\geq 0,\ \ \ m\in\mathbb Z^+.
\end{eqnarray*}
Then 
\begin{eqnarray*}
\sum_{m=1}^\infty a_m=\eta_1^2,\ \ \ \sum_{m\geq K} a_m=\eta_{K}^2.
\end{eqnarray*}
Take $\boldsymbol{\ell}\in \mathbb Z^n\setminus\{\boldsymbol 0\}$ and $|\boldsymbol{\ell}|=1$. Since $P\in\mathbb P^{d\times n}$, 
we know $|P\boldsymbol{\ell}|\neq 0$.
Define the function $U$, which is only dependent on the Fourier 
coefficients at frequency $\boldsymbol{k}_m=m\boldsymbol{\ell},m\in\mathbb Z^+$ 
and its nonzero Fourier coefficients are given by
\begin{eqnarray*}
\widehat{U}_{\boldsymbol{k}_m}=\frac{\sqrt{a_m}}{|P\boldsymbol{k}_m|^t},\ \ \ m\in\mathbb Z^+.
\end{eqnarray*}
Then the following inequality holds 
\begin{eqnarray*}
\|U\|_{\overline{H}_{P}^s}^2&=&\sum_{\boldsymbol{k}\in \mathbb Z^n\setminus\{\boldsymbol 0\}}|P\boldsymbol{k}|^{2s}|\widehat{U}_{\boldsymbol{k}}|^2
=\sum_{m=1}^{\infty}|P\boldsymbol{k}_m|^{2s}\frac{a_m}{|P\boldsymbol{k}_m|^{2t}}\\
&=&|P\boldsymbol{\ell}|^{2(s-t)}\sum_{m=1}^{\infty}a_m m^{2(s-t)}\leq |P\boldsymbol{\ell}|^{2(s-t)}\sum_{m=1}^{\infty}a_m
=|P\boldsymbol{\ell}|^{2(s-t)}\eta_1^2<\infty,
\end{eqnarray*}
which means 
$U\in \overline{H}_{P,0}^s(\mathbb T^n)$.
And for $K\geq 1$, we have
\begin{eqnarray*}
\|U-\Pi_{K} U\|_{\overline{H}_{P}^t}^2&=&\sum_{|\boldsymbol{k}|\geq K}|P\boldsymbol{k}|^{2t}|\widehat{U}_{\boldsymbol{k}}|^2\\
&=&\sum_{m\geq K}|P\boldsymbol{k}_m|^{2t}\frac{a_m}{|P\boldsymbol{k}_m|^{2t}}=\sum_{m\geq K} a_m=\eta_K^2.
\end{eqnarray*}

Taking square roots on both sides of the above equality leads to 
\begin{eqnarray*}
\|U-\Pi_{K} U\|_{\overline{H}_{P}^t}=\eta_K,
\end{eqnarray*}
which is stronger than the conclusion of the theorem for the case $s\leq t$.

Next, we prove the case $s>t$ in the theorem.
For $j\in\mathbb Z,j\geq 1$, since the conclusion of Diophantine approximation, 
there exists a sequence $\{\boldsymbol{k}_j\}$ with $|\boldsymbol{k}_j|$ increasing monotonically in $j$ such that
\begin{eqnarray*}
|P\boldsymbol{k}_j|^{2(s-t)}<2^{-j},\ \ \ j\in\mathbb Z,\ \ j\geq 2.
\end{eqnarray*}
Let $N_1=0$ and $N_j=|\boldsymbol{k}_j|$ with $j\in\mathbb Z,j\geq 2$, and define a sequence $\{b_j\}_{j=2}^{\infty}$ as follows
\begin{eqnarray*}
b_j = \eta_{N_{j-1}}^2-\eta_{N_{j}}^2\geq 0,\ \ \ j\in\mathbb Z,\ \ j\geq 2.
\end{eqnarray*}
Then it is obvious that 
\begin{eqnarray*}
\sum_{j=2}^{\infty}b_j=\eta_{N_1}^2,\ \ \ \sum_{j\geq K}^{\infty}b_j=\eta_{N_{K-1}}^2.
\end{eqnarray*}
Define $U$ to contain only the frequencies $\{\boldsymbol{k}_j\}$, with its nonzero Fourier coefficients given by
\begin{eqnarray*}
\widehat{U}_{\boldsymbol{k}_j}=\frac{\sqrt{b_j}}{|P\boldsymbol{k}_j|^t},\ \ \ j\in\mathbb Z, \ \ j\geq 2.
\end{eqnarray*}
The property that $u\in \overline{H}_{P,0}^s(\mathbb T^n)$ 
can be deduced by the following argument
\begin{eqnarray*}
\|U\|_{\overline{H}_{P}^s}^2&=&\sum_{\boldsymbol{k}\in \mathbb Z^n\setminus\{\boldsymbol 0\}}|P\boldsymbol{k}|^{2s}|\widehat{U}_{\boldsymbol{k}}|^2
=\sum_{j=2}^{\infty}|P\boldsymbol{k}_j|^{2s}\frac{b_j}{|P\boldsymbol{k}_j|^{2t}}\\
&=&\sum_{j=2}^{\infty}|P\boldsymbol{k}_j|^{2(s-t)}b_j\leq \sum_{j=2}^{\infty} 2^{-j} b_j\leq \sum_{j=2}^{\infty} b_j =\eta_{N_1}^2<\infty.
\end{eqnarray*}
For any $K\geq 1$, there exists a unique $\ell\in \mathbb Z,\ell\geq 2$, 
such that $N_{\ell-1}< K\leq N_{\ell}$.
Then we have the following lower bound 
\begin{eqnarray*}
\|U-\Pi_{K} U\|_{\overline{H}_{P}^t}^2&=&\sum_{|\boldsymbol{k}|\geq K} 
|P\boldsymbol{k}|^{2t}|\widehat{U}_{\boldsymbol{k}}|^2=\sum_{j\geq \ell} |P\boldsymbol{k}_j|^{2t}\frac{b_j}{|P\boldsymbol{k}_j|^{2t}}\\
&=& \sum_{j\geq \ell} b_j=\eta_{N_{\ell-1}}^2\geq \eta_{K}^2.
\end{eqnarray*}
By taking square roots on both sides of the above inequality, 
we obtain the desired result for the case $s>t$. 
Thus, the proof is complete. 
\end{proof}
It is easy to find that, in Theorem \ref{bad_lower_bound}, 
the sequence $\{\eta_{m}\}_{m=1}^{\infty}$ may decay arbitrarily 
slowly, even more slowly than any polynomial rate.
Taking $t=1$, this means that, under the $\overline{H}_{P,0}^s$-regularity 
the Fourier spectral method does not in general admit any algebraic rate of 
convergence in the energy norm.
Taking $t=0$, no algebraic rate of convergence can be obtained in the $L^2$-norm.
This motivates us to further improve the regularity of $U$ so that $U$ belongs to the 
standard Sobolev space. Theorem \ref{theorem_exists_unique_Sobolev} will provide 
a fairly natural sufficient condition guarantee this regularity improvement.

The proof of Theorem \ref{theorem_exists_unique_Sobolev} requires the following inequality.
For clarity, we first present the proof of the inequality.
\begin{lemma}\label{lemma_ineq}
For any $r\in\mathbb R$, the following inequality holds for any $\boldsymbol{k},\boldsymbol{\ell}\in \mathbb Z^n$
\begin{eqnarray*}
(1+|\boldsymbol{k}|^2)^{\frac{r}{2}}\leq C_r(1+|\boldsymbol{k}-\boldsymbol{\ell}|)^{|r|}(1+|\boldsymbol{\ell}|^2)^{\frac{r}{2}},
\end{eqnarray*}
where $C_r=2^{\frac{|r|}{2}}$ which is independent of $\boldsymbol{k}$ and $\boldsymbol{\ell}$.
\end{lemma}

\begin{proof}
It is easy to have the following inequalities
\begin{eqnarray}\label{ineq_1+|x|}
(1+|\boldsymbol{x}|^2)^{\frac{1}{2}}\leq 1+|\boldsymbol{x}|\leq \sqrt{2}(1+|\boldsymbol{x}|^2)^{\frac{1}{2}},\ \ \ \forall \boldsymbol{x}\in \mathbb R^n.
\end{eqnarray}
First, we prove that the Lemma holds for $r\geq 0$.
It is easy to verify that the following inequality holds
\begin{eqnarray}\label{ineq_1+|k|}
1+|\boldsymbol{k}|\leq (1+|\boldsymbol{k}-\boldsymbol{\ell}|)(1+|\boldsymbol{\ell}|).
\end{eqnarray}
Combining \eqref{ineq_1+|x|} and (\ref{ineq_1+|k|}) leads to following inequalities 
\begin{eqnarray*}
(1+|\boldsymbol{k}|^2)^{\frac{1}{2}}\leq 1+|\boldsymbol{k}|\leq (1+|\boldsymbol{k}-\boldsymbol{\ell}|)
(1+|\boldsymbol{\ell}|)\leq \sqrt{2}(1+|\boldsymbol{k}-\boldsymbol{\ell}|)(1+|\boldsymbol{\ell}|^2)^{\frac{1}{2}}.
\end{eqnarray*}
Then, raising both sides of the inequality to the power $r$ can prove the desired result 
for the case of $r\geq 0$.

Next, we prove the case for $r<0$.
Denote $s=-r>0$. By exchanging $\boldsymbol{k}$ and $\boldsymbol{\ell}$ in \eqref{ineq_1+|k|}, we obtain
\begin{eqnarray*}
1+|\boldsymbol{\ell}|\leq (1+|\boldsymbol{k}-\boldsymbol{\ell}|)(1+|\boldsymbol{k}|).
\end{eqnarray*}
Combined with \eqref{ineq_1+|x|}, the following inequalities hold
\begin{eqnarray*}
(1+|\boldsymbol{\ell}|^2)^{\frac{1}{2}}\leq 1+|\boldsymbol{\ell}|\leq (1+|\boldsymbol{k}-\boldsymbol{\ell}|)(1+|\boldsymbol{k}|)
\leq \sqrt{2}(1+|\boldsymbol{k}-\boldsymbol{\ell}|)(1+|\boldsymbol{k}|^2)^{\frac{1}{2}}.
\end{eqnarray*}
Raising both sides of the inequality to the power $s$ and then multiplying by $(1+|\boldsymbol{\ell}|^2)^{-\frac{s}{2}}(1+|\boldsymbol{k}|^2)^{-\frac{s}{2}}$, 
we prove the desired result for $r<0$. Thus, the proof of this lemma is complete.
\end{proof}

\begin{theorem}\label{theorem_exists_unique_Sobolev}
Let $s\geq 0$.
Assume that $A$ satisfies \eqref{bound_A} and Assumption \ref{assume_A} with $N>n+s+\tau$,
and that the projection matrix $P$ satisfies the Diophantine condition \eqref{Diophantine_condition_P}. 
Then for any $F\in H_{0}^{2\tau+s}(\mathbb T^n)$, 
there exists a unique solution $U\in H_{0}^{s}(\mathbb T^n)$ for the variational problem \eqref{Projection_Problem}.
\end{theorem}

\begin{proof}
By comparing the Fourier coefficients of equation \eqref{Projection_Problem}, 
the following equation holds
\begin{eqnarray*}
\widehat{F}_{\boldsymbol{k}}=\sum_{\boldsymbol{\ell}\in\mathbb Z^n\setminus\{\boldsymbol 0\}} 
\widehat{A}_{\boldsymbol{k}-\boldsymbol{\ell}} (P\boldsymbol{k})\cdot (P\boldsymbol{\ell}) \widehat{U}_{\boldsymbol{\ell}},
\ \ \ \boldsymbol{k}\in\mathbb Z^n\setminus\{\boldsymbol 0\}.
\end{eqnarray*}
Multiplying both sides of the above identity by $|P\boldsymbol{k}|^{-1}$ and doing rearrangement, 
we obtain the following equation 
\begin{eqnarray}\label{eq_fourier_ceff}
|P\boldsymbol{k}|^{-1}\widehat{F}_{\boldsymbol{k}}=\sum_{\boldsymbol{\ell}\in\mathbb Z^n\setminus\{\boldsymbol 0\}} 
\widehat{A}_{\boldsymbol{k}-\boldsymbol{\ell}} 
\frac{(P\boldsymbol{k})\cdot (P\boldsymbol{\ell})}{|P\boldsymbol{k}||P\boldsymbol{\ell}|} |P\boldsymbol{\ell}|\widehat{U}_{\boldsymbol{\ell}}.
\end{eqnarray}
Let us define two functions $W,G\in L_0^2(\mathbb T^n)$ 
with following Fourier coefficients 
\begin{eqnarray*}
\widehat{W}_{\boldsymbol{k}}=|P\boldsymbol{k}|\widehat{U}_{\boldsymbol{k}},
\ \ \ \widehat{G}_{\boldsymbol{k}}=|P\boldsymbol{k}|^{-1}\widehat{F}_{\boldsymbol{k}},
\ \ \ \boldsymbol{k}\in \mathbb Z^n\setminus\{\boldsymbol 0\}. 
\end{eqnarray*}
We use $w$ and $g$ to denote the images of $W$ and $G$ under the Fourier coefficient 
isomorphism from $L_0^2(\mathbb T^n)$ to $\ell^2(\mathbb Z^n\setminus\{\boldsymbol 0\})$, namely,
\begin{eqnarray*}
w=\{\widehat{W}_{\boldsymbol{k}}\},\ \ \ g=\{\widehat{G}_{\boldsymbol{k}}\}.
\end{eqnarray*}
We further define an infinite-dimensional matrix $\mathcal{C}$ whose entries are given by
\begin{eqnarray*}
\mathcal{C}_{\boldsymbol{k}\boldsymbol{\ell}}=\widehat{A}_{\boldsymbol{k}-\boldsymbol{\ell}}\frac{(P\boldsymbol{k})
\cdot (P\boldsymbol{\ell})}{|P\boldsymbol{k}||P\boldsymbol{\ell}|},
\ \ \ \boldsymbol{k},\boldsymbol{\ell}\in \mathbb Z^n\setminus\{\boldsymbol 0\}.
\end{eqnarray*}
Then the problem \eqref{eq_fourier_ceff} is equivalent to the following 
infinite-dimensional linear equation
\begin{eqnarray*}
\mathcal{C} w=g.
\end{eqnarray*}

First, we estimate the off-diagonal decay rate of $\mathcal{C}^{-1}$.
Since $a(U,U)=(w,\mathcal{C}w)_{\ell^2}$, the following inequalities hold
\begin{eqnarray*}
\alpha_0\|w\|_{\ell^2}\leq (w,\mathcal{C}w)_{\ell^2}\leq\alpha_1\|w\|_{\ell^2},
\end{eqnarray*}
and then $\mathcal{C}$ is invertible.
It follows from the assumptions of the theorem that 
\begin{eqnarray*}
|\mathcal{C}_{\boldsymbol{k}\boldsymbol{\ell}}|=\left|\widehat{A}_{\boldsymbol{k}-\boldsymbol{\ell}}
\frac{(P\boldsymbol{k})\cdot (P\boldsymbol{\ell})}{|P\boldsymbol{k}||P\boldsymbol{\ell}|}\right|
\leq |\widehat{A}_{\boldsymbol{k}-\boldsymbol{\ell}}|\leq C_{N}(1+|\boldsymbol{k}-\boldsymbol{\ell}|)^{-N}.
\end{eqnarray*}
Then, it follows from Jaffard's Theorem \cite{Jaffard,Jaffard_Eng} that the entries of $\mathcal{C}^{-1}$ 
have the same off-diagonal decay rate, that is
\begin{eqnarray*}
|\mathcal{C}^{-1}_{\boldsymbol{k}\boldsymbol{\ell}}|\leq C^{\prime}_{N}(1+|\boldsymbol{k}-\boldsymbol{\ell}|)^{-N}.
\end{eqnarray*}

Next, we prove that the following inequality holds for every $r\in \mathbb R$
\begin{eqnarray*}
\|\mathcal{C}^{-1}z\|_{h^r}\lesssim \|z\|_{h^r}.
\end{eqnarray*}
Here $z$ is an arbitrary sequence in $h^r(\mathbb Z^n\setminus\{\boldsymbol 0\})$, 
where $h^r(\mathbb Z^n\setminus\{\boldsymbol 0\})$ is the sequence space associated 
with the Fourier coefficients of functions in $H_0^r(\mathbb T^n)$.
Define an infinite-dimensional matrix $\mathcal K=(\mathcal{K}_{\boldsymbol{k}\boldsymbol{\ell}})$ by
\begin{eqnarray*}
\mathcal{K}_{\boldsymbol{k}\boldsymbol{\ell}}=(1+|\boldsymbol{k}-\boldsymbol{\ell}|)^r|\mathcal{C}_{\boldsymbol{k}\boldsymbol{\ell}}^{-1}|,
\end{eqnarray*}
then $\mathcal{K}$ satisfies the following off-diagonal decay estimate
\begin{eqnarray*}
|\mathcal{K}_{\boldsymbol{k}\boldsymbol{\ell}}|\leq C_N^{\prime} (1+|\boldsymbol{k}-\boldsymbol{\ell}|)^{-(N-r)}.
\end{eqnarray*}
For $N>n+r$, we have 
\begin{eqnarray*}
\sum_{\boldsymbol{\ell}\in\mathbb Z^n\setminus\{\boldsymbol 0\}} 
|\mathcal{K}_{\boldsymbol{k}\boldsymbol{\ell}}|\leq C_N^{\prime} 
\sum_{\boldsymbol{\ell}\in\mathbb Z^n\setminus\{\boldsymbol 0\}} (1+|\boldsymbol{k}-\boldsymbol{\ell}|)^{-(N-r)}
\leq C_N^{\prime} \sum_{\boldsymbol{m}\in\mathbb Z^n\setminus\{\boldsymbol 0\}} (1+|\boldsymbol{m}|)^{-(N-r)}<\infty.
\end{eqnarray*}
By the same argument, we also have $\sum_{\boldsymbol{k}\in\mathbb Z^n\setminus\{\boldsymbol 0\}} |\mathcal{K}_{\boldsymbol{k}\boldsymbol{\ell}}|<\infty$. 
Then, it follows from Schur's test; see \cite[Theorem 5.2]{Schur} and the original paper \cite{Schur1911}, that $\mathcal{K}$ is a bounded linear operator and 
\begin{eqnarray*}
\|\mathcal{K} z\|_{\ell^2}\leq C_{\mathcal{K}} \|z\|_{\ell^2}.
\end{eqnarray*}
By Lemma \ref{lemma_ineq}, together with a straightforward argument, for $\boldsymbol{k}\in\mathbb Z^n\setminus\{\boldsymbol 0\}$, we have
\begin{eqnarray*}
(1+|\boldsymbol{k}|^2)^{\frac{r}{2}}\left|\sum_{\boldsymbol{\ell}\in\mathbb Z^n\setminus\{\boldsymbol 0\}} \mathcal{C}_{\boldsymbol{k}\boldsymbol{\ell}}^{-1} z_{\boldsymbol{\ell}}\right|
&\leq& \sum_{\boldsymbol{\ell}\in\mathbb Z^n\setminus\{\boldsymbol 0\}} (1+|\boldsymbol{k}|^2)^{\frac{r}{2}} |\mathcal{C}_{\boldsymbol{k}\boldsymbol{\ell}}^{-1}| |z_{\boldsymbol{\ell}}|\\
&\leq& C_r \sum_{\boldsymbol{\ell}\in\mathbb Z^n\setminus\{\boldsymbol 0\}}  (1+|\boldsymbol{k}-\boldsymbol{\ell}|)^{|r|} |\mathcal{C}_{\boldsymbol{k}\boldsymbol{\ell}}^{-1}| (1+|\boldsymbol{\ell}|^2)^{\frac{r}{2}}|z_{\boldsymbol{\ell}}|\\
&=& C_r \sum_{\boldsymbol{\ell}\in\mathbb Z^n\setminus\{\boldsymbol 0\}} \mathcal{K}_{\boldsymbol{k}\boldsymbol{\ell}} v_{\boldsymbol{\ell}},
\end{eqnarray*}
where $v=\{(1+|\boldsymbol{\ell}|^2)^{\frac{r}{2}} |z_{\boldsymbol{\ell}}|\}$. Then it comes to the following estimate 
\begin{eqnarray*}
\|\mathcal{C}^{-1}z\|_{h^r}^2&=&\sum_{\boldsymbol{k}\in\mathbb Z^n\setminus\{\boldsymbol 0\}} \left(1+|\boldsymbol{k}|^2\right)^{r}
\left|\sum_{\boldsymbol{\ell}\in\mathbb Z^n\setminus\{\boldsymbol 0\}} \mathcal{C}_{\boldsymbol{k}\boldsymbol{\ell}}^{-1} z_{\boldsymbol{\ell}}\right|^2\\
&\leq&   C_r^2 \sum_{\boldsymbol{k}\in\mathbb Z^n\setminus\{\boldsymbol 0\}} \left(\sum_{\boldsymbol{\ell}\in\mathbb Z^n\setminus\{\boldsymbol 0\}} 
\mathcal{K}_{\boldsymbol{k}\boldsymbol{\ell}} v_{\boldsymbol{\ell}}\right)^2\\
&=& C_r^2 \|\mathcal{K}v\|_{\ell^2}^2\leq C_r^2 C_{\mathcal{K}}^2 \|v\|_{\ell^2}^2 
= C_r^2 C_{\mathcal{K}}^2 \|z\|_{h^r}^2,
\end{eqnarray*}
which is $\|\mathcal{C}^{-1}z\|_{h^r}\lesssim \|z\|_{h^r}$.

Finally, together with the conclusion of Lemma \ref{lemma_Diophantine_Embedding}, 
this yields the following inequality
\begin{eqnarray*}
\|U\|_{H^s}\lesssim \|W\|_{H^{s+\tau}} \lesssim \|G\|_{H^{s+\tau}} \lesssim \|F\|_{H^{s+2\tau}}.
\end{eqnarray*}
Thus, the proof of the theorem is complete.
\end{proof}

\begin{remark}
We can illustrate by giving examples that Theorem \ref{theorem_exists_unique_Sobolev} is sharp.
Assume $A\equiv A_0$ is a constant function, then through Fourier analysis, we can obtain
\begin{eqnarray*}
\widehat{U}_{\boldsymbol{k}}=\frac{\widehat{F}_{\boldsymbol{k}}}{A_0|P\boldsymbol{k}|^2},\ \ \ \boldsymbol{k}\in \mathbb Z^n\setminus\{\boldsymbol 0\}.
\end{eqnarray*} 
Then we have the following result 
\begin{eqnarray*}
F\in H^{2\tau}(\mathbb T^n) \Rightarrow F\in \overline{H}_{P}^{-2}(\mathbb T^n)	
\Leftrightarrow 
\sum_{\boldsymbol{k}\in \mathbb Z^n\setminus\{\boldsymbol 0\}} \frac{|\widehat{F}_{\boldsymbol{k}}^2|}{|P\boldsymbol{k}|^4}<\infty
\Leftrightarrow \sum_{\boldsymbol{k}\in \mathbb Z^n\setminus\{\boldsymbol 0\}} |\widehat{U}_{\boldsymbol{k}}|^2<\infty
\Leftrightarrow U\in L^2_{0}(\mathbb T^n).
\end{eqnarray*}
This corresponds to the case $s=0$ of Theorem \ref{theorem_exists_unique_Sobolev}.
\end{remark}

Theorem \ref{theorem_exists_unique_Sobolev} is particularly important for numerical approximation.
Indeed, once the solution $U$ is improved to the standard Sobolev space, 
the classical approximation theory for Fourier projections becomes applicable, 
which yields quantitative convergence rates.
In particular, the arbitrarily slow convergence phenomenon exhibited in Theorem \ref{bad_lower_bound} 
no longer occurs under the additional assumptions of Theorem \ref{theorem_exists_unique_Sobolev}.
The following corollary illustrates this conclusion.

\begin{corollary}\label{cor_projection_rate}
Let the assumptions of Theorem \ref{theorem_exists_unique_Sobolev} be satisfied, 
and assume that $F\in H_0^{2\tau+s+1}(\mathbb T^n)$.
Then the corresponding solution \(U\) to \eqref{Projection_Problem} belongs 
to \(H_0^{s+1}(\mathbb T^n)\), and the Fourier projection \(\Pi_K U\) 
satisfies the following estimates
\begin{eqnarray*}
\|U-\Pi_K U\|_{\overline{H}_{P,0}^1}&\leq& C K^{-s}\|U\|_{H_0^{s+1}},\\
\|U-\Pi_K U\|_{L_0^2}&\leq& C K^{-(s+1)}\|U\|_{H_0^{s+1}}.
\end{eqnarray*}
\end{corollary}

\begin{proof}
By Theorem \ref{theorem_exists_unique_Sobolev} with \(s\) replaced by \(s+1\), we have
\begin{eqnarray*}
U\in H_0^{s+1}(\mathbb T^n).
\end{eqnarray*}
Then the following estimates follow from the standard approximation property 
of the Fourier projection \(\Pi_K\):
\begin{eqnarray*}
\|U-\Pi_K U\|_{H_0^1}&\leq& C K^{-s}\|U\|_{H_0^{s+1}},\\
\|U-\Pi_K U\|_{L^2_0}&\leq& C K^{-(s+1)}\|U\|_{H_0^{s+1}}.
\end{eqnarray*}
Since \( H_0^1(\mathbb T^n) \hookrightarrow \overline{H}_{P,0}^1(\mathbb T^n)\), 
the desired results follows and the proof is complete. 
\end{proof}

\begin{remark}
It is worth noting that, although the convergence rate derived here is similar to that in \cite[Theorem 3.1]{JiangLiZhang}, 
a crucial difference lies in the treatment of regularity. 
The convergence result in \cite{JiangLiZhang} assumes a priori that the parent function $U$ 
of the exact solution and the source term $F$ already possess the Sobolev regularity required 
for the approximation analysis.
For quasiperiodic problems, however, such assumptions are not entirely natural, 
since the required regularity should preferably be justified from the variational problem 
itself rather than imposed directly in the high-dimensional periodic setting.

A main contribution of the present work is to provide a detailed study of the 
regularity of quasiperiodic problems and to derive a regularity improvement result 
for the exact solution under suitable assumptions on the source term.
More precisely, we identify natural sufficient conditions on $A$ and $F$ which 
guarantee that the solution $U$ belongs to the Sobolev space needed for the numerical approximation.
\end{remark}

\section{Tensor neural network based machine learning method}\label{Section_TNN_ML}
In this section, we propose a TNN-based machine learning method for solving the 
high-dimensional periodic problem \eqref{Projection_Problem}.
An approximate solution to the original quasiperiodic problem is then obtained by applying 
the pullback mapping to the resulting high-dimensional TNN solution.
\subsection{Tensor neural network architecture}\label{Subsection_TNN}
In this subsection, we introduce the TNN structure and its approximation 
properties and together with several techniques for improving numerical stability. 
The approximation property and computational complexity of related 
integration associated with the TNN structure have been discussed in \cite{WangJinXie}.

The TNN is built with $n$ subnetworks, and each subnetwork is 
a continuous mapping from a bounded closed set $\mathbb T\subset\mathbb R$ 
to $\mathbb R^p$, which can be expressed as
\begin{eqnarray}\label{def_FNN}
\Phi_i(y_i;\theta_i)&=&\big(\phi_{i,1}(y_i;\theta_i), \phi_{i,2}(y_i;\theta_i),\cdots,\phi_{i,p}(y_i;\theta_i)\big)^\top  , 
\end{eqnarray}
where $i=1, \cdots, n$, each $y_i$ denotes the one-dimensional input, 
$\theta_i$ denotes the parameters of the $i$-th subnetwork, typically the weights and biases. 
As illustrated in Figure \ref{TNNstructure}, the TNN structure is composed of $p$ Feedforward 
Neural Networks (FNNs) for spatial basis functions $\Phi_i(y_i;\theta_i)$, $i=1,2,\cdots, n$. 
\begin{figure}[htb!]
\centering
\includegraphics[width=7.7cm,height=8cm]{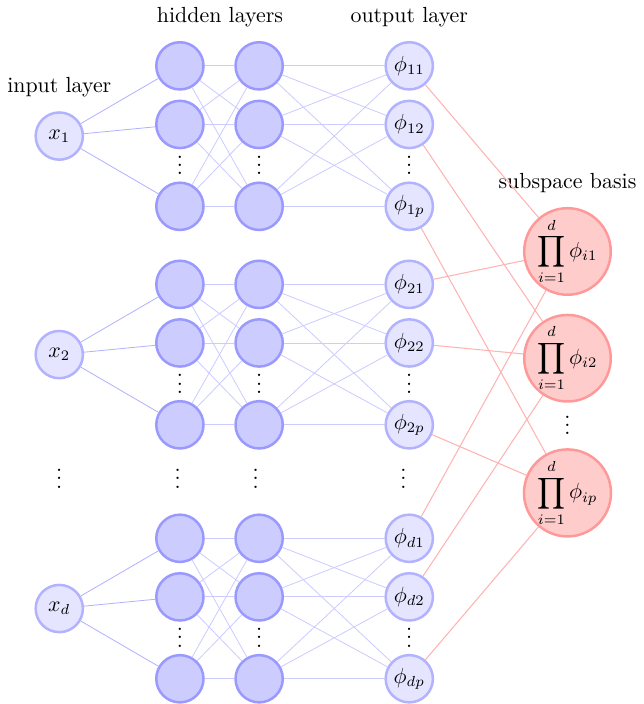}
\caption{Architecture of TNN. Blue arrows mean linear transformation
(or affine transformation). Each ending node of red arrows is obtained by taking the
scalar multiplication of all starting nodes of red arrows that end in this ending node.
The final output of TNN is derived from the linear combination
of all red nodes.}\label{TNNstructure}
\end{figure}

In order to improve the numerical stability, we normalize each 
$\phi_{i,j}(y_i)$ and use the following normalized TNN structure:
\begin{eqnarray}\label{def_TNN_normed}
\Psi(\boldsymbol{y};c,\theta )&=&\sum_{j=1}^p{c_j}\widehat{\phi}_{1,j}(y_1;\theta_1)
\cdots \widehat{\phi}_{n,j}(y_n;\theta_n)\nonumber\\
&=&\sum_{j=1}^p{c_j}\prod_{i=1}^n{\widehat{\phi }_{i,j}(y_i;\theta_i)}
=:\sum_{j=1}^pc_j\varphi_j(\boldsymbol{y};\theta), 
\end{eqnarray}
where each $c_j$ is a scaling parameter with respect to the normalized rank-one function, 
$c=\{c_j\}_{j=1}^{p}$ denotes the linear coefficients for the 
basis system built by the rank-one functions $\varphi _j(\boldsymbol{y};\theta )$, $j=1, \cdots, p$. 
For $i=1, \cdots, n$,  $j=1,\cdots, p$,  $\widehat\phi_{i,j}(y_i;\theta_i)$ 
are $L^2$-normalized function as follows:
\begin{eqnarray}\label{normalized}
\widehat{\phi}_{i,j}(y_i;\theta_i)=\frac{\phi_{i,j}(y_i;\theta_i)}
{\left\|\phi_{i,j}(y_i;\theta_i) \right\|_{L^2(\mathbb T)}}. 
\end{eqnarray}
For simplicity of notation, $\phi_{i,j}(y_i;\theta_i)$ denotes 
the normalized function in the following parts.

Due to the isomorphism relation between $L^2(\Omega_1\times\cdots\times\Omega_n)$ 
and the tensor product space $L^2(\Omega_1)\otimes\cdots\otimes L^2(\Omega_{n})$ 
with $\Omega_i=\mathbb T$ for $i=1, \cdots, n$,   
the process of approximating the function $f(x)\in L^2(\Omega_1\times\cdots\times\Omega_{n})$ 
with the TNN defined by (\ref{def_TNN_normed}) is actually to search for a correlated CP decomposition 
to approximate $F(\boldsymbol{y})$ in the space $L^2(\Omega_1)\otimes\cdots\otimes L^2(\Omega_{n})$ 
with rank not greater than $p$.
Thanks to the low-rank structure, we will find that 
the polynomial mapping acting on the TNN and its derivatives can be done with small 
scale computational work \cite{WangJinXie}. In order to show the validity of solving PDEs by the TNN,  
we introduce the following approximation result to the functions of 
the space $L^2(\Omega_1\times\cdots\times\Omega_{n})$ in the sense 
of $H^m(\Omega_1\times\cdots\times\Omega_n)$-norm. 
\begin{theorem}\cite[Theorem 1]{WangJinXie}\label{theorem_approximation}
Assume that each $\Omega_i$ is an interval in $\mathbb R$ for $i=1, \cdots, n$, $\Omega=\Omega_1\times\cdots\times\Omega_n$,
and the function $f(\boldsymbol{x})\in H^m(\Omega)$. 
Then for any tolerance $\varepsilon>0$, there exist a
positive integer $p$ and the corresponding TNN defined by (\ref{def_TNN_normed})
such that the following approximation property holds
\begin{equation}\label{eq:L2_app}
\|f(\boldsymbol{x})-\Psi(\boldsymbol{x};c,\theta)\|_{H^m(\Omega)}<\varepsilon.
\end{equation}
\end{theorem}

\subsection{Adaptive TNN subspace and loss functions}\label{Subsection_TNN_ML}
This subsection is devoted to designing a type of adaptive optimization of TNN parameters for 
solving the quasiperiodic elliptic problem (\ref{Elliptic_Equation}) with high accuracy. 
For this aim, we first introduce the solving procedure 
and then show the way to compute the high-dimensional integration included in the loss functions.
The basic idea is to generate a subspace using the output functions of TNN as the basis functions,  
where we find an approximate solution in some sense.  Then the training steps are adopted to update the 
TNN subspace to improve its approximation property \cite{WangLinLiaoLiuXie}. 

Based on the definition (\ref{def_TNN_normed}), we define the TNN subspace as follows: 
\begin{eqnarray}\label{Subspace_V_p}
\mathcal V_{p}:={\rm span}\Big\{\varphi_{j}(\boldsymbol{y};\theta), j=1, \cdots, p\Big\},
\end{eqnarray}
where $\varphi_{j}(\boldsymbol{y};\theta)$ is the rank-one function defined in (\ref{def_TNN_normed}).

Then the subspace approximation procedure can be adopted to get the solution $U_p\in \mathcal V_p$. 
For this purpose, in this paper, the subspace approximation $U_p\in \mathcal V_p$ 
is obtained by solving the following discrete Galerkin problem: 
Find $U_p\in \mathcal V_p$ such that
\begin{eqnarray}\label{Galerkin_Problem}
a(U_p, V_p)=(F, V_p),\ \  \forall V_p\in\mathcal V_p.
\end{eqnarray}
The following theorem shows that the solution of discrete Galerkin problem \eqref{Galerkin_Problem} 
has the best approximation property with respect to the $\overline H_{P,0}^1$-norm.

\begin{theorem}\label{Theorem_Error_Estimate}
Let $U$ is the solution of variational problem \eqref{Projection_Problem} 
and  $U_p$ is the solution of discrete Galerkin problem \eqref{Galerkin_Problem}, 
then the following inequality holds
\begin{eqnarray}\label{eq_Cea}
\| U-U_p \|_{\overline H_{P,0}^1}\leq \frac{\alpha_1}{\alpha_0} \inf_{V_p\in \mathcal V_p}\|  U-V_p \|_{\overline H_{P,0}^1}
\end{eqnarray}
\end{theorem}

\begin{proof}
The following equality can be directly obtained from  \eqref{Projection_Problem} and \eqref{Galerkin_Problem} 
\begin{eqnarray*}
a(U-U_p,V_p)=0,\ \ \ \forall V_P\in\mathcal V_p.
\end{eqnarray*}
Then from (\ref{Ellipticity_a_Periodic})-(\ref{Boundedness_a_Periodic}),  for any $V_p\in \mathcal V_P$, the following inequality holds
\begin{eqnarray*}
\alpha_0 \| U-U_p \|_{\overline H_{P,0}^1}^2 &\leq& a(U-U_p,U-U_p)\\
&=&a(U-U_p, U-V_p)+a(U-U_p,V_p-U_p)\\
&=&a(U-U_p, U-V_p)\\
&\leq& \alpha_1 \| U-U_p  \|_{\overline H_{P,0}^1} \| U-V_p \|_{\overline H_{P,0}^1}.
\end{eqnarray*}
From this, we can obtain
\begin{eqnarray*}
\| U-U_p \|_{\overline H_{P,0}^1}\leq \frac{\alpha_1}{\alpha_0} \|  U-V_p \|_{\overline H_{P,0}^1}.
\end{eqnarray*}
Since the arbitrariness of $V_p$, \eqref{eq_Cea} can be derived, and the theorem is proved.
\end{proof}

Based on the basis functions of the subspace $\mathcal V_p$ in (\ref{Subspace_V_p}), we can assemble the 
corresponding stiffness matrix and right hand side term as in the finite element method. 
By solving the deduced linear system, we obtain the Galerkin approximation 
$\Psi$ \eqref{def_TNN_normed} in subspace $\mathcal V_p$.

After obtaining the approximation $\Psi\in \mathcal V_p$, the subspace $\mathcal V_p$ is 
updated by optimizing the loss function in the 
training steps which will be defined in Subsection \ref{SubSection_TNN}. 
Then following the idea of the adaptive finite element methods or 
moving mesh methods, we update the subspace $\mathcal V_p$ adaptively 
during the training process of the machine learning method. 
By ``adaptively'', we mean that the NN parameters and thus the fixed-structure 
NN subspace are updated based on some criterion, which in this case is the 
optimization of the loss function, to achieve better accuracy performance. 
For example, if the Ritz-type loss function 
\begin{eqnarray}\label{Ritz_Loss}
\mathcal L(\Psi) = \frac{1}{2}a(\Psi,\Psi)-(F,\Psi)
\end{eqnarray}
is used, then the training process is actually to 
minimize the error $\|U-\Psi\|_a$ of the NN approximation $\Psi$ in the energy norm.
\begin{theorem}\label{Theorem_Ritz_Optimization}
Let $\mathcal V_{\rm NN}$ denote the neural network-based trial space.  
If the NN function $U_{\rm NN}\in  \mathcal V_{\rm NN}\subset \overline H_{P,0}^1(\mathbb T^n)$ 
is defined by the following optimization problem 
\begin{eqnarray}\label{Ritz_Problem}
U_{\rm NN} = \arg\inf_{V\in\mathcal V_{\rm NN}}
\left[\frac{1}{2}a(V,V)-(F,V)\right], 
\end{eqnarray} 
then the following optimal property holds 
\begin{eqnarray}\label{Ritz_Optimal}
\left\|U-U_{\rm NN}\right\|_a
= \inf_{V\in \mathcal V_{\rm NN}}\left\|U-V\right\|_a.
\end{eqnarray}
\end{theorem}
\begin{proof}
It is easy to know that the exact solution $U$ satisfy the following problem 
\begin{eqnarray}
U = \arg\inf_{V\in \overline H^1_{P}}\left[\frac{1}{2}a(V,V)-(F,V)\right].
\end{eqnarray}
From the definitions for the exact solution $U$ and $U_{\rm NN}$, we have the following optimality
\begin{eqnarray*}\label{Equality_1}
\frac{1}{2}\left\|U_{\rm NN}\right\|_a^2-\left(F,U_{\rm NN}\right) 
= \inf\limits_{V\in \mathcal V_{\rm NN}}\left(\frac{1}{2}\left\|V\right\|_a^2-\left(F, V\right)\right).
\end{eqnarray*}
Then  the following equality for $U$ and $U_{\rm NN}$ holds 
\begin{eqnarray*}\label{Equality_2}
\frac{1}{2}\left\|U_{\rm NN}\right\|_a^2-a(U, U_{\rm NN})
+\frac{1}{2}\left\|U\right\|_a^2
= \inf\limits_{V\in \mathcal V_{\rm NN}}
\left(\frac{1}{2}\left\|V\right\|_a^2-a(U, V)+\frac{1}{2}\left\|U\right\|_a^2\right).
\end{eqnarray*}
The following equality hold 
\begin{eqnarray*}
\frac{1}{2}\left\|U-U_{\rm NN}\right\|_a^2 
= \inf_{U\in \mathcal V_{\rm NN}}\frac{1}{2}\left\|U-V\right\|_a^2,
\end{eqnarray*}
which leads to the optimal approximation. 
\end{proof}
Theorem \ref{Theorem_Ritz_Optimization} shows that the minimum point 
of the Ritz type of loss function 
is equivalent to arrive the optimal error estimates among 
all the choices of NN parameters $\theta$ and $c$.  
In Algorithm \ref{Algorithm_1}, the NN subspace is updated by minimizing 
the loss function which defines the target of the training steps.  

Inspired by the standard PINN framework, one may naturally consider using a residual-type loss function in the training process.
However, for quasiperiodic problems, such a loss function is not always appropriate.
To illustrate this point, we present a few intuitive examples.
\begin{lemma}\label{Theorem_Residual_Los}
If $A$ is a constant function, then for every $\varepsilon>0$, there exist function $V\in \overline H_{P,0}^1(\mathbb T^n)$ 
such that 
\begin{eqnarray}
\frac{\|{\rm div}\left(P^\top  AP\nabla V\right)\|_0^2}{\| V\|_{\overline H^1_{P,0}}^2} \leq \varepsilon. 
\end{eqnarray}
\end{lemma}
\begin{proof}
For $V\in \overline H^1_{P,0}(\mathbb T^n)$, the following equation holds
\begin{eqnarray*}
\| V\|_{\overline H^1_{P,0}}^2&=&\|P\nabla V\|_0^2\\
&=&\left(\sum_{\boldsymbol{k}\in\mathbb Z^n\setminus\{\boldsymbol 0\}}
{\rm i}\widehat V_{\boldsymbol{k}}P\boldsymbol{k}e^{{\rm i}\boldsymbol{k}^\top  \boldsymbol{y}}, 
\sum_{\boldsymbol{k}\in\mathbb Z^n\setminus\{\boldsymbol 0\}}{\rm i} 
\widehat V_{\boldsymbol{k}}P\boldsymbol{k}e^{{\rm i}\boldsymbol{k}^\top  \boldsymbol{y}}\right)\\
&=&\sum_{\boldsymbol{k}\in\mathbb Z^n\setminus\{\boldsymbol 0\}}
|\widehat V_{\boldsymbol{k}} |^2\|P\boldsymbol{k}\|^2,
\end{eqnarray*}
where the last equation is due to \eqref{B_P^2_orth}.
Similarly, 
\begin{eqnarray*}
&&\left\|{\rm div}\left(P^\top  AP\nabla V\right)\right\|_0^2\\
&=&\left(\sum_{\boldsymbol{k}\in\mathbb Z^n\setminus\{\boldsymbol 0\}}
{\rm i}\widehat V_{\boldsymbol{k}}\ {\rm div}\left(P^\top  A P\boldsymbol{k}e^{{\rm i}\boldsymbol{k}^\top \boldsymbol{y}}\right), 
\sum_{\boldsymbol{k}\in\mathbb Z^n\setminus\{\boldsymbol 0\}}{\rm i} 
\widehat V_{\boldsymbol{k}}\ {\rm div}\left(P^\top  AP\boldsymbol{k}e^{{\rm i}\boldsymbol{k}^\top \boldsymbol{y}}\right)\right)\\
&=&\left(\sum_{\boldsymbol{k}\in\mathbb Z^n\setminus\{\boldsymbol 0\}}
\widehat V_{\boldsymbol{k}}\left(\boldsymbol{k}^\top  P^\top  A P\boldsymbol{k}\right)e^{{\rm i}\boldsymbol{k}^\top \boldsymbol{y}}, 
\sum_{\boldsymbol{k}\in\mathbb Z^n\setminus\{\boldsymbol 0\}}
\widehat V_{\boldsymbol{k}}\left(P^\top  AP\boldsymbol{k}\right)e^{{\rm i}\boldsymbol{k}^\top \boldsymbol{y}}\right)\\
&=&\sum_{\boldsymbol{k}\in\mathbb Z^n\setminus\{\boldsymbol 0\}}
|\widehat V_{\boldsymbol{k}} |^2\|A\|^2\|P\boldsymbol{k}\|^4.
\end{eqnarray*}
From Remark \ref{Remark_1}, we can know the set $\Gamma$ defined in (\ref{Definition_SubGroup}) 
has the accumulation point  $\mathbf 0\in \mathbb R^d$. Then for arbitrary small $\varepsilon>0$, there exist 
$\bar{\boldsymbol{k}} \in \mathbb Z^n$ such that 
\begin{eqnarray*}
\|P\bar{\boldsymbol{k}}\|^2 < \frac{\varepsilon}{\alpha_1^2}, 
\end{eqnarray*} 
where the constant $\alpha_1$ is defined in (\ref{Lower_Upper_Bound_alpha}).  
Let us chose $\bar V=e^{{\rm i}\bar{\boldsymbol{k}}^\top \boldsymbol y}$. The following inequality holds 
\begin{eqnarray*}
\frac{\|{\rm div}\left(P^\top  AP\nabla \bar V\right)\|_0^2}{\| V\|_{\overline H^1_{P,0}}^2} 
\leq \|A\|^2 \|P\bar{\boldsymbol{k}}\|^2 <\varepsilon.
\end{eqnarray*}
Then the proof is complete. 
\end{proof} 
\begin{remark}
In Lemma \ref{Theorem_Residual_Los}, we assume that $A$ is constant only to make the underlying mechanism 
behind the potential failure of the PINN loss more transparent to the reader.
In fact, the same conclusion remains valid whenever the eigenvalue problem
\begin{eqnarray*}
LV:=-{\rm div}(P^\top  AP\nabla V)  =\lambda V
\end{eqnarray*}
admits a sequence of eigenvalues converging to $0$.
Indeed, if $(\lambda_j, V_j)$ is such a sequence, then
\begin{eqnarray*}
\frac{\|LV_j\|_0^2}{\|V\|_{\overline H^1_{P,0}}^2}=\lambda_j\rightarrow 0,
\end{eqnarray*}
and hence conclusion of Lemma \ref{Theorem_Residual_Los} still follows.

Such a situation is not exceptional. For instance, it typically occurs when the coefficient $A$ is decoupled in one direction.
More precisely, when $A$ depends only on $n-1$ variables in an $n$-dimensional problem.
In that case, it is well known that the operator $L$ admits a fiber decomposition, from which one can obtain a sequence of point spectrum tending to $0$.

For more general coefficients $A$, the coupling between different frequencies makes it difficult 
to derive a general result and requires additional analytical techniques
This issue will be further explored in our future work.
\end{remark}
In PINN, the residual type of loss function is defined as follows 
\begin{eqnarray}\label{loss_Residual}
\mathcal L(U_p)=\|F+{\rm div}(P^\top  A P\nabla U_p)\|_0^2.
\end{eqnarray}
Then combining (\ref{Projection_Problem}), we have 
\begin{eqnarray}\label{Posteriori_Error_Estimate_2}
\mathcal L(U_p)=\|{\rm div}(P^\top  A P\nabla U)-{\rm div}(P^\top  A P\nabla U_p)\|_0^2
=\|{\rm div}(P^\top  A P\nabla(U-U_P))\|_0^2.
\end{eqnarray}
Based on this equality, Lemma \ref{Theorem_Residual_Los} implies that the residual type of loss function (\ref{loss_Residual}) 
can not control the error $\|U-U_p\|_{\overline{H}_{P,0}^1}^2$, and hence cannot control $\|U-U_p\|_a^2$, either. 
This suggests that the choice of the loss function requires particular care.

\subsection{Procedure of TNN-based machine learning method}\label{SubSection_TNN}
In this subsection, we introduce the adaptive TNN subspace method to solve  
high-dimensional problem (\ref{Projection_Problem}) and then the original 
quasiperiodic problem (\ref{Elliptic_Equation}) by 
executing the projection method described in Section \ref{Section_Projection}. 
Then the machine learning method can be described as follows: 
\begin{enumerate}
\item Solve the higher-dimensional periodic problem (\ref{Projection_Problem}) 
using TNN-based machine learning method.  

\item Compute an approximate solution to the original quasiperiodic elliptic problem \eqref{Elliptic_Equation} 
by projecting back through the pullback isomorphism $\mathcal{J}_P$.
\end{enumerate}
In the first step, in order to solve the high-dimensional problems (\ref{Projection_Problem}), 
we build the TNN function $\Psi(\boldsymbol{y}; c, \Theta)$ 
described in Section \ref{Subsection_TNN} as the approximation to the functions 
in $\overline H_{P,0}^1(\mathbb T^n)$. 
In order to make $\Psi(\boldsymbol{y};c, \Theta)$ 
satisfy the periodic boundary condition on $\mathbf{y}$, 
we choose the $\sin(x)$ as the activation function, and fix all entries in the 
weight matrix of the first layer for each FNN subnetwork on the corresponding dimension 
of $\mathbf{y}$ as the multiples of $2\pi$. 
Meanwhile, the corresponding biases are registered as trainable parameters, 
and so are the remaining parameters for $\mathbf{y}$. 
Under these settings, $\phi_{i,j}(y_i, \theta)$ 
satisfies the periodic boundary condition on $y_i$ automatically.
In addition, since the integral of the solution of equation (\ref{Projection_Problem}) is $0$ 
with respect to $\mathbf{y}$, we construct a new tensor neural network function 
based on $\Psi(\boldsymbol{y}; c, \Theta)$ as
\begin{eqnarray}\label{hat-psi}
\widehat{\Psi}(\boldsymbol{y}; c, \Theta) = \Psi( \boldsymbol{y}; c, \Theta) 
- \int_{\mathbb T^n} \Psi(\boldsymbol{y}; c, \Theta)d\boldsymbol{y} . 
\end{eqnarray}
which satisfies the following condition
\begin{eqnarray}\label{Integration_Constraint}
\int_{\mathbb T^n} \widehat{\Psi}( \boldsymbol{y}; c, \Theta)d\boldsymbol{y} = 0.
\end{eqnarray}
The TNN function $\widehat{\Psi}(\boldsymbol{y}; c, \Theta)$ now satisfies 
the constraints of the solution for the equation (\ref{Projection_Problem}). 
Notice that it is usually tricky to make the trial functions 
satisfy the integration constraint (\ref{Integration_Constraint}) 
of the projection problem (\ref{Projection_Problem}), 
since the accuracy of the high-dimensional integration is always hard to be guaranteed.  
However, due to the special structure of TNN, 
the constraint here is easy and natural to deal with for TNN-based method.

For the training procedure, Theorem \ref{Theorem_Ritz_Optimization} and Lemma \ref{Theorem_Residual_Los} suggest 
that we prefer to primarily use the following Ritz-type loss function associated with \eqref{Galerkin_Problem} 
for the periodic elliptic problem \eqref{Projection_Problem}
\begin{eqnarray}\label{loss_Ritz}
L(\theta) = \int_{\mathbb T^n}
\left(\frac{1}{2}\|A^{\frac{1}{2}}P\nabla \Psi(\boldsymbol{y};c, \Theta)\|^2
-F\Psi(\boldsymbol{y};c, \Theta) \right)d\boldsymbol{y}.
\end{eqnarray}
In our numerical experiments, we observe that, once a sufficiently accurate initial approximation is obtained 
through the Ritz-type loss function, further optimization using the PINN-type residual loss \eqref{loss_Residual} 
may yield additional improvement in accuracy.

After the $\ell$-th training step, the TNN $\Psi (\boldsymbol{y};c,\theta^{(\ell)})$ 
belongs to the following subspace:
\begin{eqnarray*}
\mathcal V_{p}^{(\ell )}:=\mathrm{span}\left\{ \varphi _j(\boldsymbol{y};\theta ^{(\ell )}),
\ \ \  j=1,\cdots ,p \right\},
\end{eqnarray*}
where $\varphi_{j}(\boldsymbol{y};\theta^{(\ell )})$ is the rank-one function defined 
in (\ref{def_TNN_normed}) after  $\ell$ training steps. 

Different from the standard FNN-based machine learning method,  
where the Monte Carlo integration is usually the indispensable option, 
the quadrature scheme with fixed quadrature points can be used 
to do the numerical integrations in this paper. 
Fortunately, based on TNN structure in the loss functions (\ref{loss_Residual}) 
and (\ref{loss_Ritz}),  Theorem 3 in \cite{WangJinXie} shows that these numerical 
integrations here does not encounter ``curse of dimensionality''
since the computational work can be bounded by the polynomial scale of dimension $n$.  
Due to the high accuracy of the high-dimensional integrations with Gauss points, 
all the computational process can be done with high accuracy, 
which is the core reason why the TNN-based method achieves high accuracy.

According to the definition of TNN and the corresponding space $\mathcal V_{p}^{(\ell )}$, 
it is easy to know that the neural network parameters $\theta^{(\ell)}$ 
determine the space $\mathcal V_p^{(\ell)}$, while the coefficient parameters $c^{(\ell)}$ 
determine the direction of TNN in the space $\mathcal V_p^{(\ell)}$. 
We can divide the optimization step into two sub-steps. Firstly, 
assume the neural network parameters $\theta^{(\ell)}$ are fixed. 
The optimal coefficient parameters $c^{(\ell+1)}$  
can be obtained by solving the Galerkin problem (\ref{Galerkin_Problem}), i.e., 
by solving the following linear equation 
\begin{eqnarray}\label{linear_system}
A^{(\ell)}c^{(\ell+1)}=B^{(\ell)},
\end{eqnarray}
where the matrix $A^{(\ell)}\in \mathbb R^{p\times p}$ and the vector 
$B^{(\ell)}\in\mathbb R^{p\times 1}$ are assembled as follows: 
\begin{eqnarray*}
A_{m,n}^{(\ell )}=a(\varphi _{n}^{(\ell )}, \varphi _{m}^{(\ell )} ),\ \  
B_{m}^{(\ell )}=(F, \varphi _{m}^{(\ell )}), \ \  1\le m,n\le p.
\end{eqnarray*}
\begin{remark}
In the process of solving equation (\ref{linear_system}) to obtain 
the coefficients $c^{(\ell +1)}$, the equation may become ill-conditioned 
due to the linear dependence of the basis functions 
$\varphi_{j}(x,t;\theta), j=1, \cdots, p$. In such cases, we can employ techniques 
such as singular value decomposition (SVD) or other methods to obtain the coefficients $c^{(\ell +1)}$. 
In this article, we solve this problem by solving the linear system (\ref{linear_system}) 
using ridge regression as follows:
\begin{eqnarray}\label{reg_linear_system}
\left( (A^{(\ell )})+\lambda E \right) c^{(\ell +1)}=B^{(\ell )},
\end{eqnarray}
where $ \lambda $ is the regularization parameter, which is chosen empirically 
as $1 \times 10^{-5}$ and $ E $ is the identity matrix. 
\end{remark}

Secondly, when the coefficient parameters $c^{(\ell+1)}$ are fixed, 
the neural network parameters $\theta^{(\ell+1)}$ can be updated by 
the optimization steps included, such as Adam or L-BFGS for the 
loss function $\mathcal{L}^{(\ell+1)}(c^{(\ell + 1)},\theta ^{(\ell)})$ 
defined in (\ref{loss_Residual}) or (\ref{loss_Ritz}). 
Thus, the TNN-based machine learning method for the high-dimensional periodic problem 
can be defined in Algorithm \ref{Algorithm_1}. 
This type of parameter optimization significantly enhances the efficiency of the training process 
and effectively improves the accuracy of the TNN method.
\begin{breakablealgorithm}
\caption{Adaptive TNN subspace method for (\ref{Projection_Problem})}\label{Algorithm_1}
\begin{enumerate}
\item Initialization:
Build the initial TNN $\Psi(\boldsymbol{y};c^{(0)},\theta^{(0)})$ defined in 
Section \ref{Subsection_TNN}, set the loss function $\mathcal{L}(\Psi (\boldsymbol{y}; c,\theta))$, 
the maximum number of training steps $M$, 
learning rate $\gamma$, and the iteration counter $\ell = 0$.
\item Define the subspace $\mathcal V_{p}^{(\ell)}$ as follows:
\begin{eqnarray*}
\mathcal V_{p}^{(\ell )}:=\mathrm{span}\left\{ \varphi_j(\boldsymbol{y};\theta ^{(\ell )}), 
j=1,\cdots ,p \right\} .
\end{eqnarray*}
Assemble the stiffness matrix $A^{(\ell)} \in \mathbb{R}^{p \times p}$ and right-hand side term
$B^{(\ell)} \in \mathbb{R}^{p}$ on $\mathcal V_{p}^{(\ell)}$:
\begin{eqnarray*}
A_{m,n}^{(\ell )}=a(\varphi _{n}^{(\ell )},\varphi _{m}^{(\ell)}), 
\ \ B_{m}^{(\ell )}=(F,\varphi _{m}^{(\ell )})\ \  1\le m,n\le p. 
\end{eqnarray*}
\item Solve the following linear system to determine the coefficient vector $c \in \mathbb{R}^{p}$
\begin{eqnarray*}
A^{(\ell)}c = B^{(\ell)}.
\end{eqnarray*}
Update $\Psi(\boldsymbol{x};c^{(\ell+1)},\theta^{(\ell)})$ 
with the coefficient vector: $c^{(\ell+1)} = c$.
\item Update the network parameters from $\theta^{(\ell)}$ to $\theta^{(\ell+1)}$, by optimizing the loss 
function $\mathcal{L}^{(\ell+1)}(c^{(\ell + 1)},\theta ^{(\ell)})$ 
defined by (\ref{loss_Residual}) or (\ref{loss_Ritz}) through gradient-based training steps.
\item Iteration: set $\ell = \ell + 1$. If $\ell < M$, go to Step 2. Otherwise, terminate.
\end{enumerate}
\end{breakablealgorithm}
In the third step, based on the high-dimensional function 
$\Psi(\boldsymbol{y}; c^{(\ell+1)}, \theta^{(\ell)})$, 
we can build the neural network function 
$\Psi(P^\top  \boldsymbol{x};c^{(\ell+1)}, \theta^{(\ell)})$
as the approximation to the solution 
$u(\boldsymbol{x})\in \overline H_{QP}^1\left(\mathbb R^d\right)$ 
for the quasiperiodic elliptic problem (\ref{Elliptic_Equation}).  
From Theorem \ref{Theorem_Error_Estimate}, 
we can easily obtain the following optimal approximation property 
\begin{eqnarray*}
\| U- \Psi(\boldsymbol{y}; c^{(\ell+1)}, \theta^{(\ell)})\|_{\overline{H}_{P,0}^1} 
\leq \frac{\alpha_1}{\alpha_0} \inf_{V_p\in \mathcal V_p^{(\ell)}}\|  U-V_p  \|_{\overline{H}_{P,0}^1}.
\end{eqnarray*}

\subsection{Quadrature scheme for TNN loss function}\label{Section_Integration}
In this subsection, we introduce the quadrature scheme for computing the integrations 
in the loss function (\ref{loss_Residual}) or (\ref{loss_Ritz}), 
which involves TNN functions and tensor-product-type coefficient. 
As for the general case, please refer to \cite{WangJinXie}, 
where the method to compute the numerical 
integrations of polynomial composite functions of TNN and their derivatives are designed. 
In the loss function (\ref{loss_Ritz}) 
\begin{equation*}
\begin{aligned}
&L(\theta) = \left(\int_{\mathbb T^n}\left(\frac{1}{2}
\sum_{i=1}^n\sum_{j=1}^n\sum_{\ell=1}^n A(\boldsymbol{y})P_{i,j}P_{i,\ell}
\frac{\partial \Psi(\boldsymbol{y};c,\Theta) }{\partial y_j}
\frac{\partial \Psi(\boldsymbol{y}; c, \Theta) }{\partial y_\ell}
-F(\boldsymbol{y})\Psi(\boldsymbol{y};c, \Theta)\right)d\boldsymbol{y}\right)^{\frac{1}{2}},  
\end{aligned}
\end{equation*}
we need to compute a $n$-dimensional integration. 
With the tensor-product-type structure of TNN and the multiscale coefficients, 
we are able to compute the integration with polynomial scale computational 
complexity and high accuracy by transforming the high-dimensional integration 
into the multiplication of $n$ one-dimensional integrations. 
As an example, we consider the case of $d=2$ and two $\mathbb Q$-independent periods, i.e., $n=4$ and   
and compute one of the terms of the $L^2$ square expansion, i.e., 
\begin{equation}\label{integration_square_expansion}
\frac{1}{2}\int_{\mathbb T^n}
\sum_{j=1}^4 \sum_{\ell=1}^4 A(\boldsymbol{y})P_{i,j}P_{i,\ell}
\frac{\partial \Psi(\boldsymbol{y}; c, \Theta) }{\partial y_j}
\frac{\partial \Psi(\boldsymbol{y}; c, \Theta) }{\partial y_\ell}d\boldsymbol{y}. 
\end{equation}
Furthermore, we assume that the coefficient $A$ has  following tensor-product structure, 
\begin{equation*}
A(\mathbf{y}) = \sum_{e=1}^{p_a} \alpha_e\prod_{i=1}^4\varphi_{i,e}(y_i).
\end{equation*}
This assumption is quite common for the $n$-dimensional periodic problems. 
Recall that by (\ref{hat-psi}) and (\ref{def_TNN_normed}), $\widehat{\Psi}_j$ is also tensor-product-type 
\begin{equation*}
\begin{aligned}
\widehat{\Psi}_j := \sum_{\ell=1}^p c_{\ell} \prod_{i=1}^4 \phi_{i,\ell}(y_i).
\end{aligned}
\end{equation*}
Hence, (\ref{integration_square_expansion}) can be computed by
\begin{eqnarray}\label{intergration_square_tensor_form}
&&\frac{1}{2}\int_{\mathbb T^4}
\sum_{j=1}^n\sum_{\ell=1}^n A(\boldsymbol{y})P_{i,j}P_{i,\ell}
\frac{\partial \Psi(\boldsymbol{y}, \Theta) }{\partial y_j}
\frac{\partial \Psi(\boldsymbol{y}, \Theta) }{\partial y_\ell}d\boldsymbol{y} \nonumber\\
&=& \frac{1}{2}\sum_{i=1}^n\sum_{j=1}^n\sum_{\ell=1}^nP_{i,j}P_{i,\ell}
\int_{\mathbb T^4}\left(\sum_{e=1}^{p} \alpha_{e} 
\prod_{k=1}^4 \varphi_{k,e}(y_k) \cdot 
\sum_{m=1}^p c_m \prod_{n \neq j} \phi_{n,m}(y_n)
\frac{\partial \phi_{j,m}(y_j)}{\partial y_j}\right.\nonumber\\
&&\times \left.\sum_{r=1}^pc_r\prod_{s\neq \ell}\phi_{s,r}(y_s)
\frac{\partial\phi_{\ell,r}(y_\ell)}{\partial y_\ell} \right)dy_1 dy_2 dy_3 dy_4\nonumber\\
&=& \frac{1}{2}\sum_{i=1}^n\sum_{j=1}^n\sum_{\ell=1}^n\sum_{e=1}^{p_\alpha}\sum_{m=1}^p\sum_{r=1}^p P_{i,j}P_{i,\ell}
\alpha_ec_mc_r \int_{\mathbb T^4}\prod_{k=1}^4 \varphi_{k,e}(y_k) \nonumber\\
&&\times  \prod_{n \neq j} \phi_{n,m}(y_n)\frac{\partial \phi_{j,m}(y_j)}{\partial y_j}\cdot\prod_{s\neq \ell}\phi_{s,r}(y_s)
\frac{\partial\phi_{\ell,r}(y_\ell)}{\partial y_\ell}dy_1 dy_2 dy_3 dy_4.
\end{eqnarray}
The expansion (\ref{intergration_square_tensor_form}) gives the hint to 
design the efficient numerical scheme to compute the high-dimensional integration (\ref{integration_square_expansion}). Each integration term in 
(\ref{intergration_square_tensor_form}) can be transformed to the one-dimensional 
integration which can be calculated with high accuracy with the classical quadrature 
schemes. For example, the integration term in (\ref{integration_square_expansion}) 
for $j=2$ and $\ell=3$ can be decomposed to the one-dimensional integration as follows 
\begin{eqnarray*}
&&\int_{\mathbb T^4}\prod_{k=1}^4 \varphi_{k,e}(y_k)\cdot  \prod_{n \neq 2} 
\phi_{n,m}(y_n)\frac{\partial \phi_{2,m}(y_2)}{\partial y_2}\cdot\prod_{s\neq 3}\phi_{s,r}(y_s)
\frac{\partial\phi_{3,r}(y_3)}{\partial y_3}dy_1 dy_2 dy_3 dy_4\\
&&=\int_{\mathbb T}\varphi_{1,e}(y_1)\phi_{1,m}(y_1)\phi_{1,r}(y_1)dy_1
\int_{\mathbb T}\varphi_{2,e}(y_2)\frac{\partial \phi_{2,m}(y_2)}{\partial y_2}
\phi_{2,r}(y_2)dy_2\\
&&\quad\times\int_{\mathbb T} \varphi_{3,e}(y_3)\phi_{3,m}(y_3)
\frac{\phi_{3,r}(y_3)}{\partial y_3}dy_3 
\int_{\mathbb T}\varphi_{4,e}(y_4)\phi_{4,m}(y_4)\phi_{4,r}(y_4)dy_4. 
\end{eqnarray*}
Therefore, the computational work of the high-dimensional integration 
of (\ref{integration_square_expansion}) can be reduced to the polynomial scale of dimension $n$. 
Following the same procedure, we are able to compute all the terms of the $L^2$ 
square expansion of loss function (\ref{loss_Residual}) and (\ref{loss_Ritz}) 
with high accuracy and acceptable computational complexity. 
Solving the minimization problems using optimization algorithms to find the optimized parameters 
of the neural networks becomes much more reasonable and numerically stable 
when the accuracy of computing loss functions is guaranteed, 
which is the initial motivation for designing the TNN structure.

\section{Numerical examples}\label{Section_Numerical_Examples}
In this section, we provide several examples to validate the efficiency and accuracy of 
the proposed TNN-based machine learning method for solving the 
quasiperiodic problems.

The incommensurate coefficient $\alpha$ is constructed according to Lindemann-Weierstrass Theorem \cite[Thoerem 1.4]{Lindemann}, 
so that the resulting projection matrix $P$ has $\mathbb Q$-independent column vectors, 
and the corresponding quasiperiodic problem is thereby determined.

In all numerical experiments in this section, we use the TNN $\Psi (\boldsymbol{y};c,\Theta)$ 
defined in \eqref{def_TNN_normed}, to approximate the solution of the high-dimensional periodic problem, 
and let $\Phi (\boldsymbol{x};c,\Theta):=\Psi (P^\top \boldsymbol{x};c,\Theta)$ 
the approximation to the low-dimensional quasiperiodic problem obtained by applying 
the pullback mapping $\mathcal J_P$ to $\Psi (\boldsymbol{y};c,\Theta)$.
The following two notations about the approximate solution 
$\Psi (\boldsymbol{y};c,\Theta)$ and the exact solution $U$ 
are used to measure the convergence behavior and accuracy of the examples in this section:
\begin{itemize}
\item Relative $L^2$ norm error 
$$e_{L^2}:=\frac{\left\|U-\Psi (\boldsymbol{y};c,\Theta)\right\|_{L^2}}
{\left\|u\right\|_{L^2}},$$ 
Here $\|\cdot\|_{L^2}$ denotes the $L^2(\mathbb T^n)$ norm.
\item Relative $L^2$ error at the test points 
$$e_{\rm test}:=\frac{\sqrt{\sum_{k=1}^K{( \Psi (\boldsymbol{x}^k;c,\Theta) -u( \boldsymbol{x}^k))^2}}}
{\sqrt{\sum_{k=1}^K{(u( \boldsymbol{x}^k))^2}}},$$
where the selected test points $\{(\boldsymbol{x}^k)\}$ are $K=5000$ randomly sampled points drawn uniformly from $\mathbb T^n$. 
\end{itemize}

For all examples, each FNN has three hidden layers with 50 neurons in each layer and uses
sine function as the activation function.
The rank parameter for the TNN is set to be $p = 20$.
Moreover, the TNN is constructed so that it automatically satisfies the periodic boundary conditions when approximating $U(\boldsymbol{y})$. 

During computing the loss function, the interval $\mathbb T$ 
is subdivided into $100$ subintervals with $4$ Gauss points selected within each subinterval for building the quadrature scheme. 
During the optimization of the TNN, the learning rate was set to 0.003 for the Adam optimizer and to 1 for the LBFGS optimizer.

All the experiments were done on an NVIDIA RTX A6000 GPU with 48 GB of memory. 

\subsection{Example 1}
In the first example, we solve the quasiperiodic elliptic problem (\ref{Elliptic_Equation}) 
with following coefficient  
\[
\alpha = \cos(2\pi x) + \cos(2\pi \sqrt{2} x) + 6.
\]
The right hand side term $f(\boldsymbol{x})$ is chosen such that the exact solution is 
\[
u = \sin(2\pi x) + \sin(2\pi \sqrt{2} x).
\]
In order to use the projection method, we take the following projection matrix 
\[
P = \begin{bmatrix} 
1 & \sqrt{2} 
\end{bmatrix}.
\]
Then $\mathbb T = [0, 1]$ and the corresponding two dimensional function $A$ and $U$ are defined as follows 
\[
A(\boldsymbol{y}) = \cos(2\pi y_1) + \cos(2\pi\sqrt{2} y_2) + 6, \ \ \ \ 
U(\boldsymbol{y}) = \sin(2\pi y_1) + \sin(2\pi\sqrt{2} y_2).
\]

First, we use the loss function defined in (\ref{loss_Ritz}) 
for Step 4 of Algorithm \ref{Algorithm_1}, training the TNN for 1,000 iterations with Adam and subsequently for 500 iterations with LBFGS.
The relative $L^2$ norm error of the approximation $\Phi(\boldsymbol{x};c,\Theta)$ 
to the exact solution $u(\boldsymbol{x})$ is \(1.0513 \times 10^{-5}\).  
The corresponding numerical results are presented in Figure \ref{fig_errors_ex1_ritz}. 
\begin{figure}[ht]
\centering
\includegraphics[width=6cm,height=4.5cm]{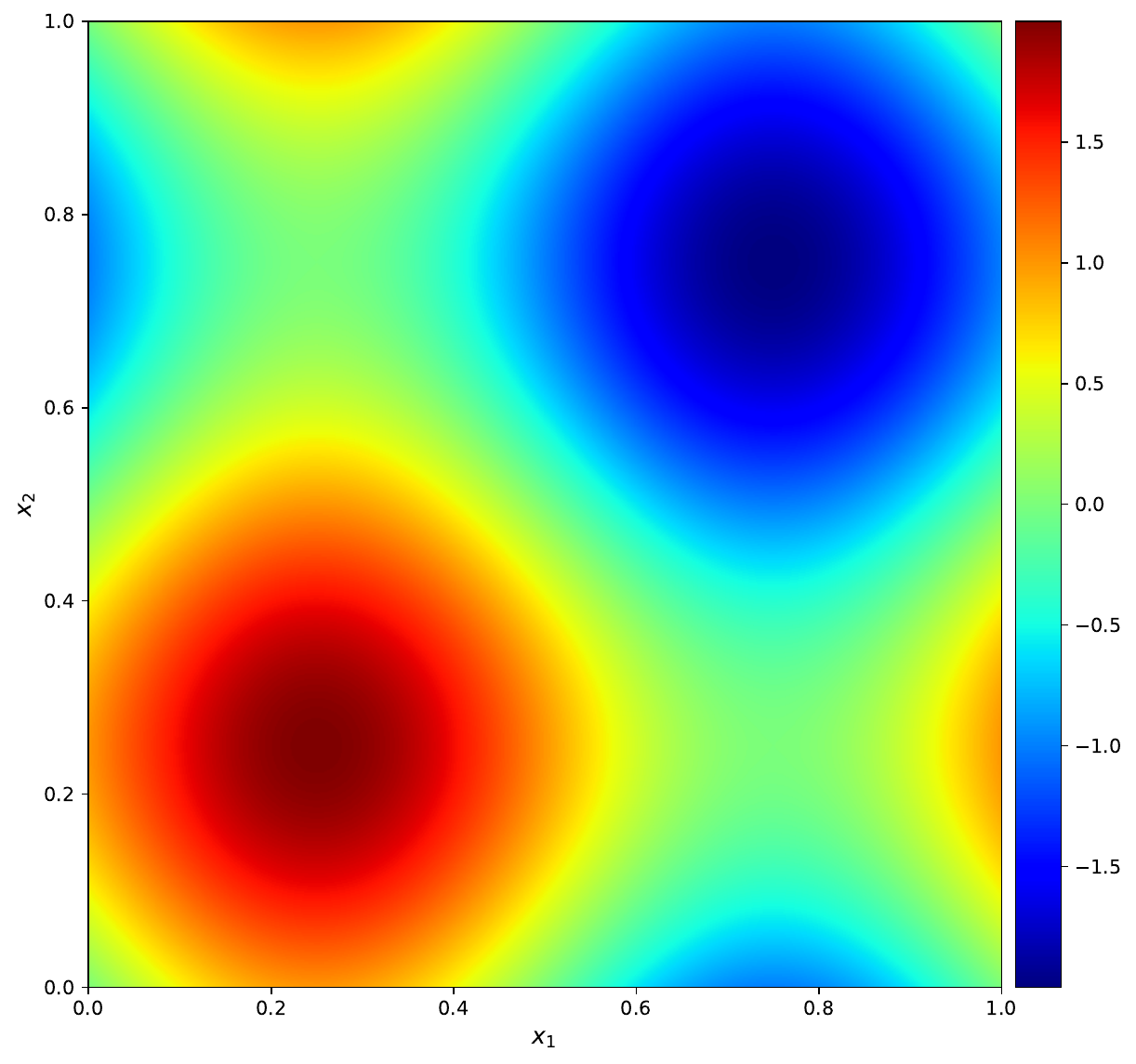}
\includegraphics[width=6cm,height=4.5cm]{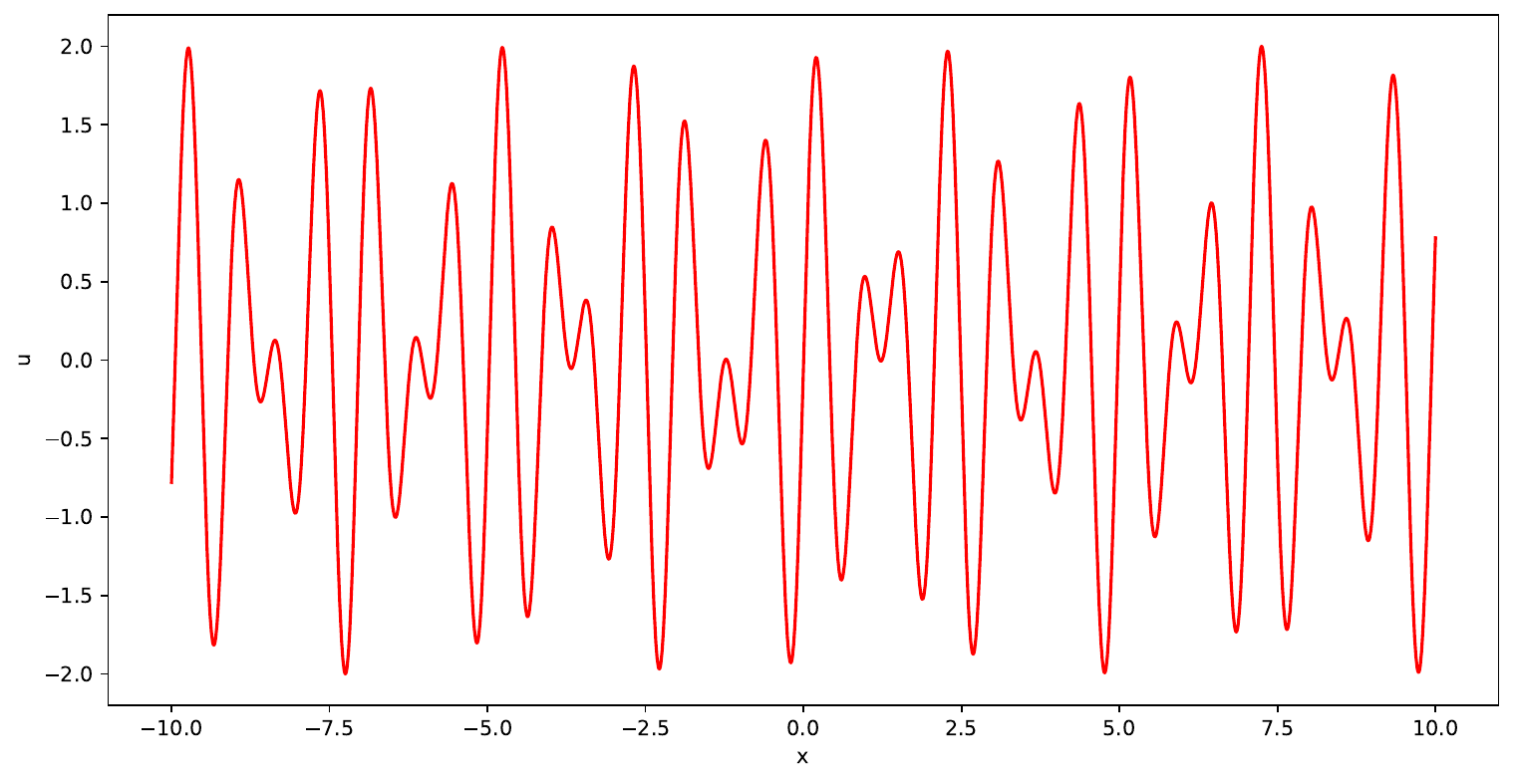}\\
\includegraphics[width=6cm,height=4.5cm]{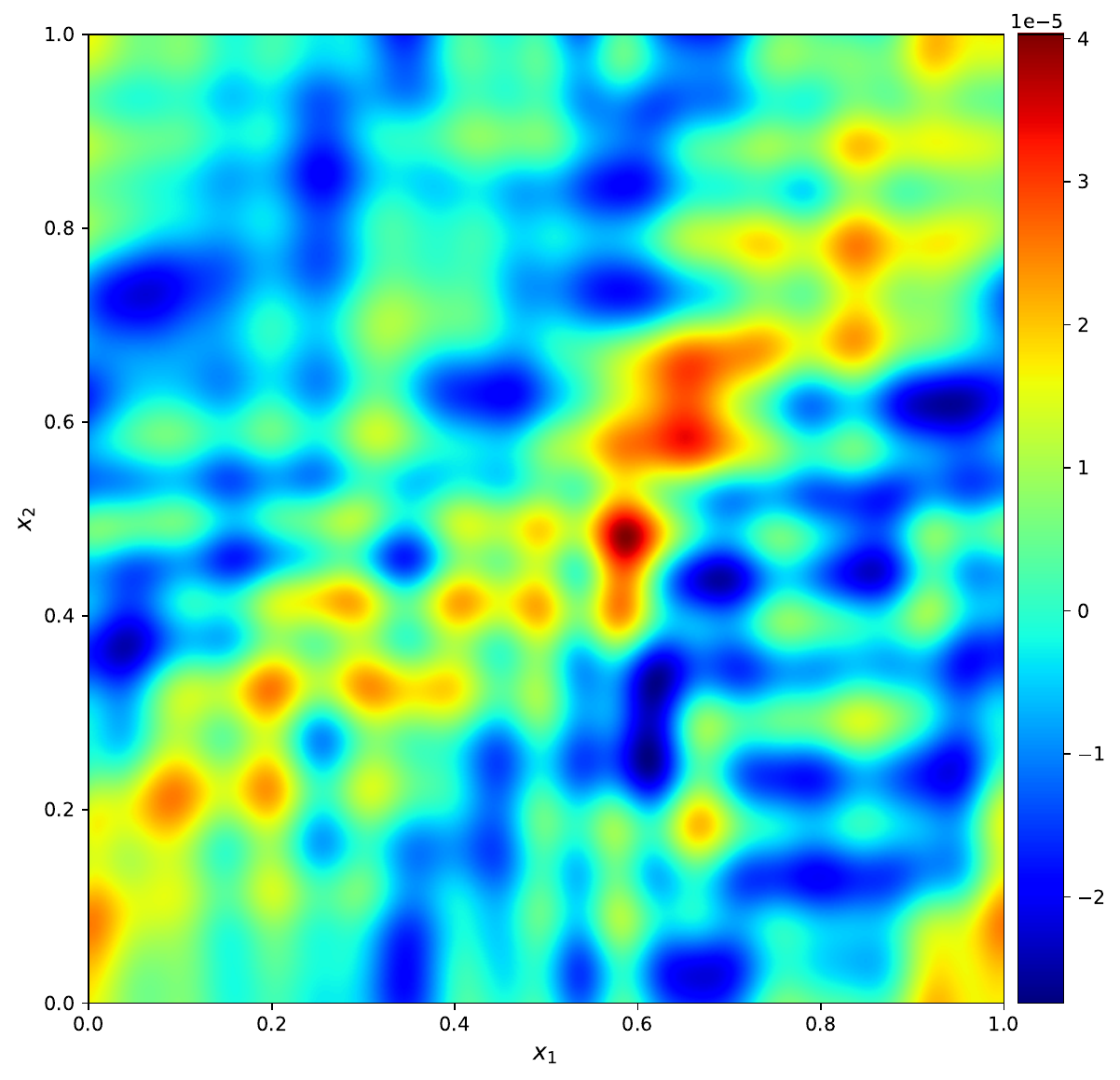}
\includegraphics[width=6cm,height=4.5cm]{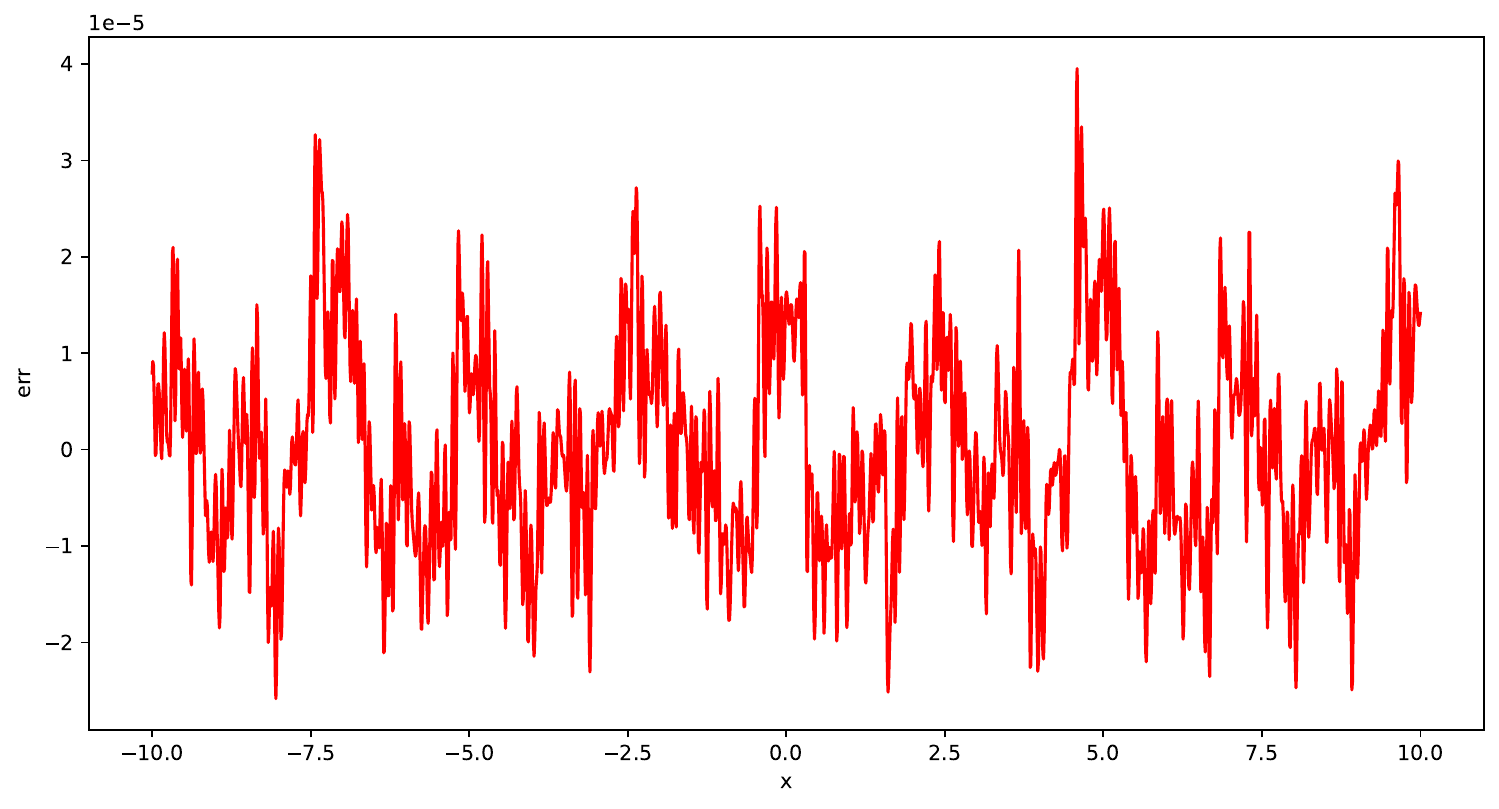}
\caption{The numerical results for Example 1 
with loss function (\ref{loss_Ritz}), top-left shows the 
TNN approximation $\Psi(\boldsymbol y;c,\Theta)$ to $U(\boldsymbol y)$, 
top-right is the figure of the approximation $\Phi(\boldsymbol{x};c,\Theta)$ to 
the quasiperiodic function $u(\boldsymbol x)$, down-left 
shows the error $U(\boldsymbol y)-\Psi(\boldsymbol y;c, \Theta)$ and 
down-right is the error $u(\boldsymbol x)-\Phi(\boldsymbol{x};c,\Theta)$.}\label{fig_errors_ex1_ritz}
\end{figure}

To further improve the accuracy, then we use the loss function defined in (\ref{loss_Residual}) 
for Step 4 of Algorithm \ref{Algorithm_1}, training the TNN for 500 iterations with LBFGS.
The relative $L^2$ norm error of the approximation $\Phi(\boldsymbol{x};c,\Theta)$ 
to the exact solution $u(\boldsymbol x)$ is \(6.8285 \times 10^{-8}\).  
The corresponding numerical results are presented in Figure \ref{fig_errors_ex1}. 
\begin{figure}[ht]
\centering
\includegraphics[width=6cm,height=4.5cm]{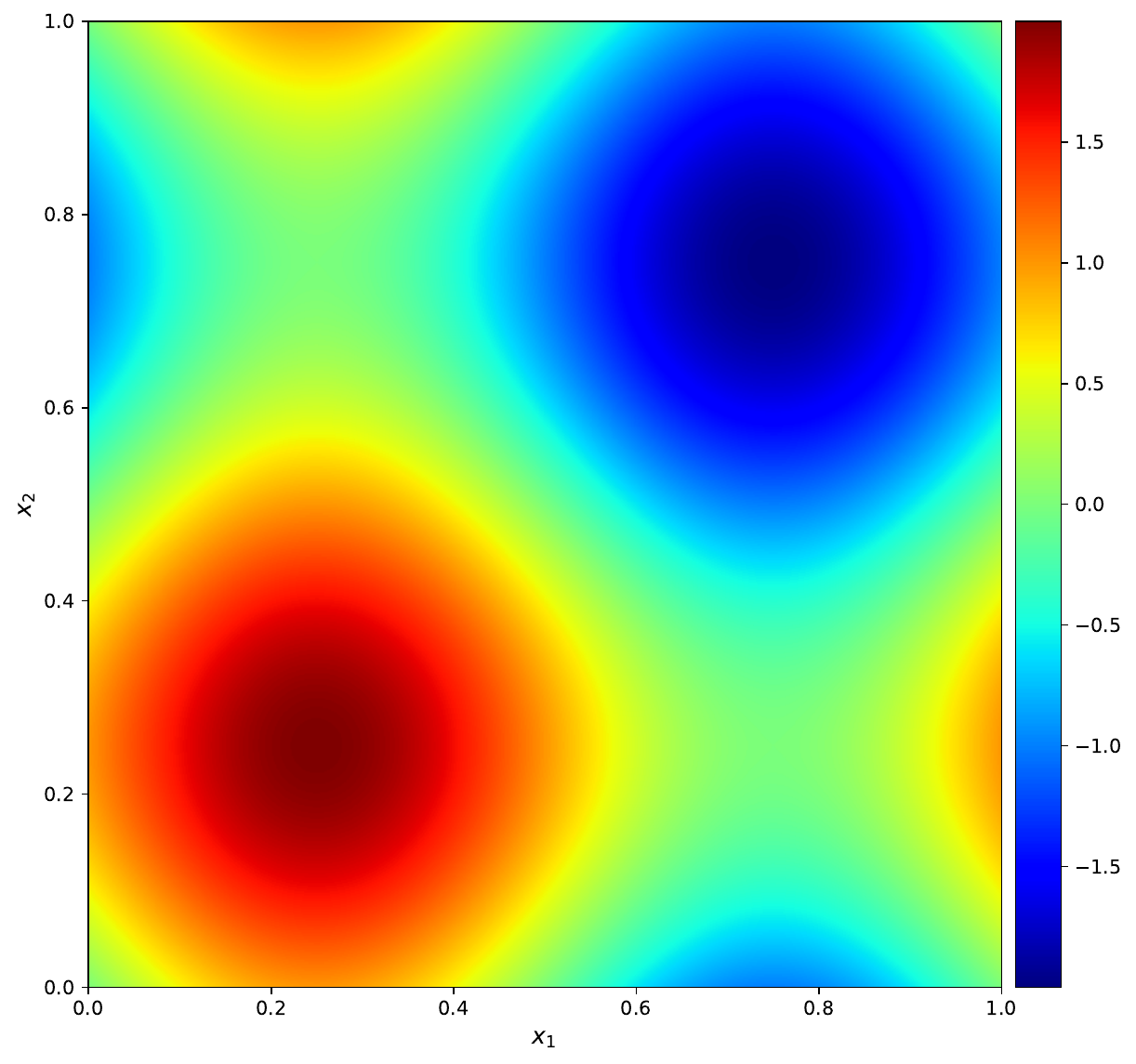}
\includegraphics[width=6cm,height=4.5cm]{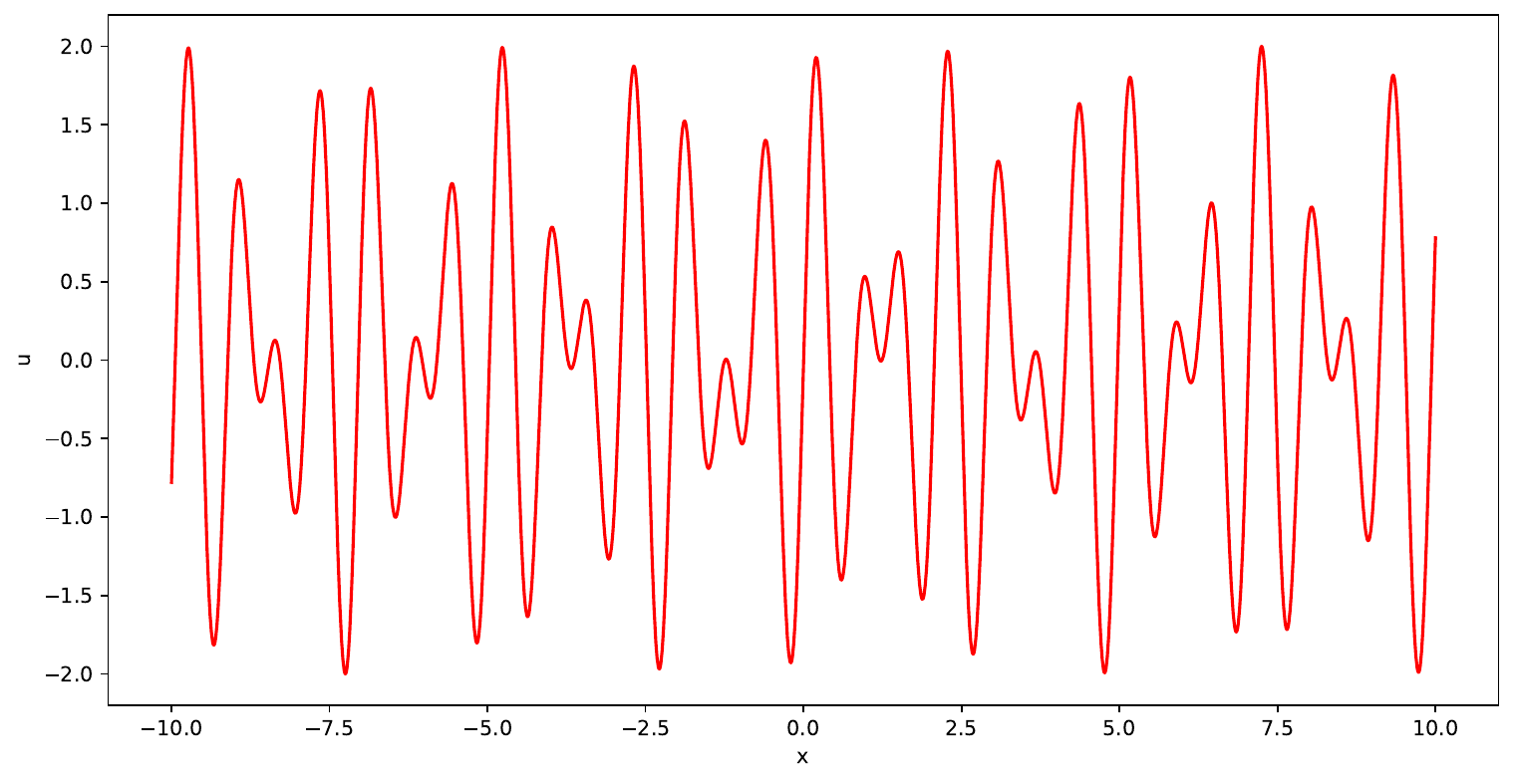}\\
\includegraphics[width=6cm,height=4.5cm]{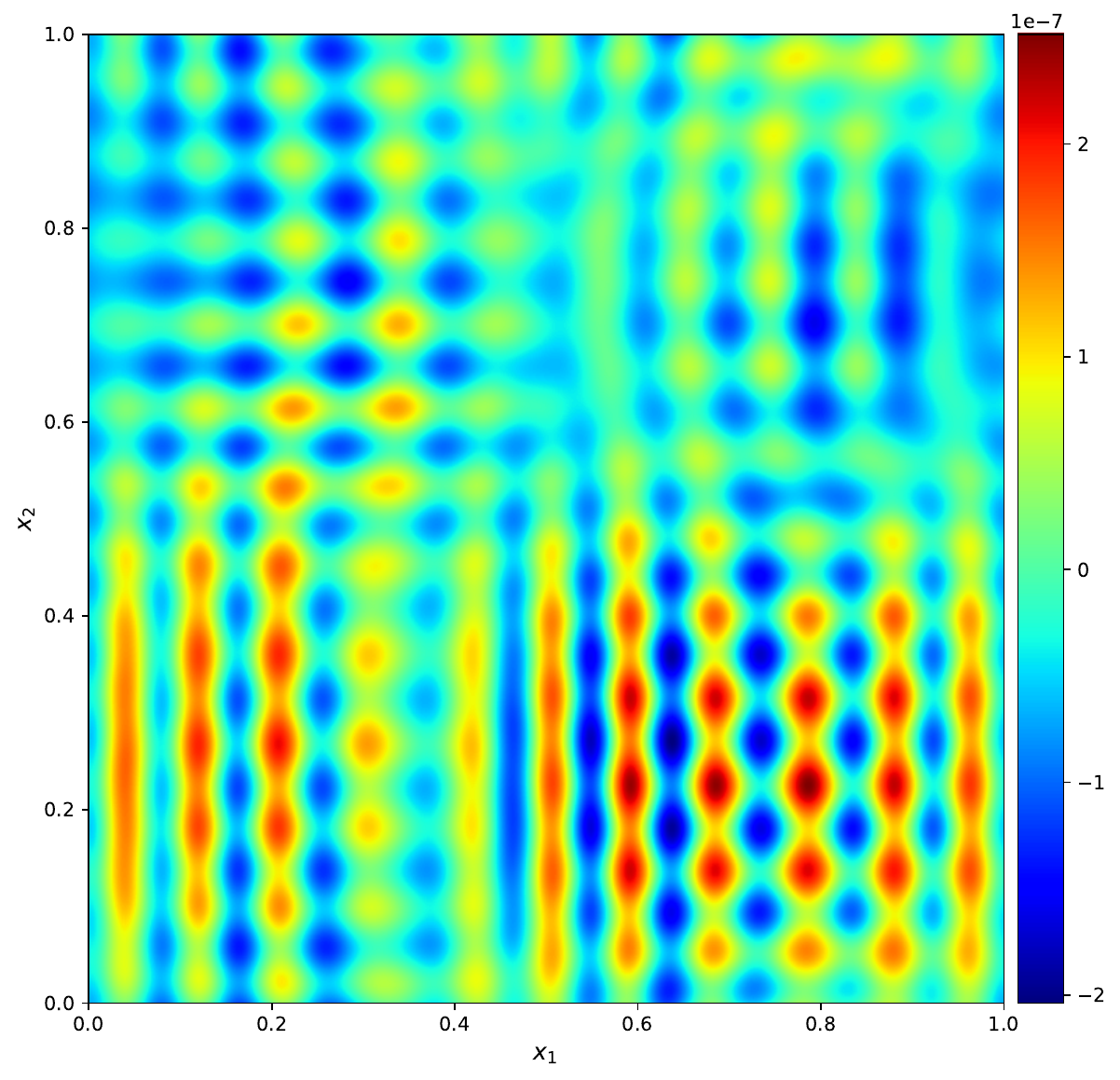}
\includegraphics[width=6cm,height=4.5cm]{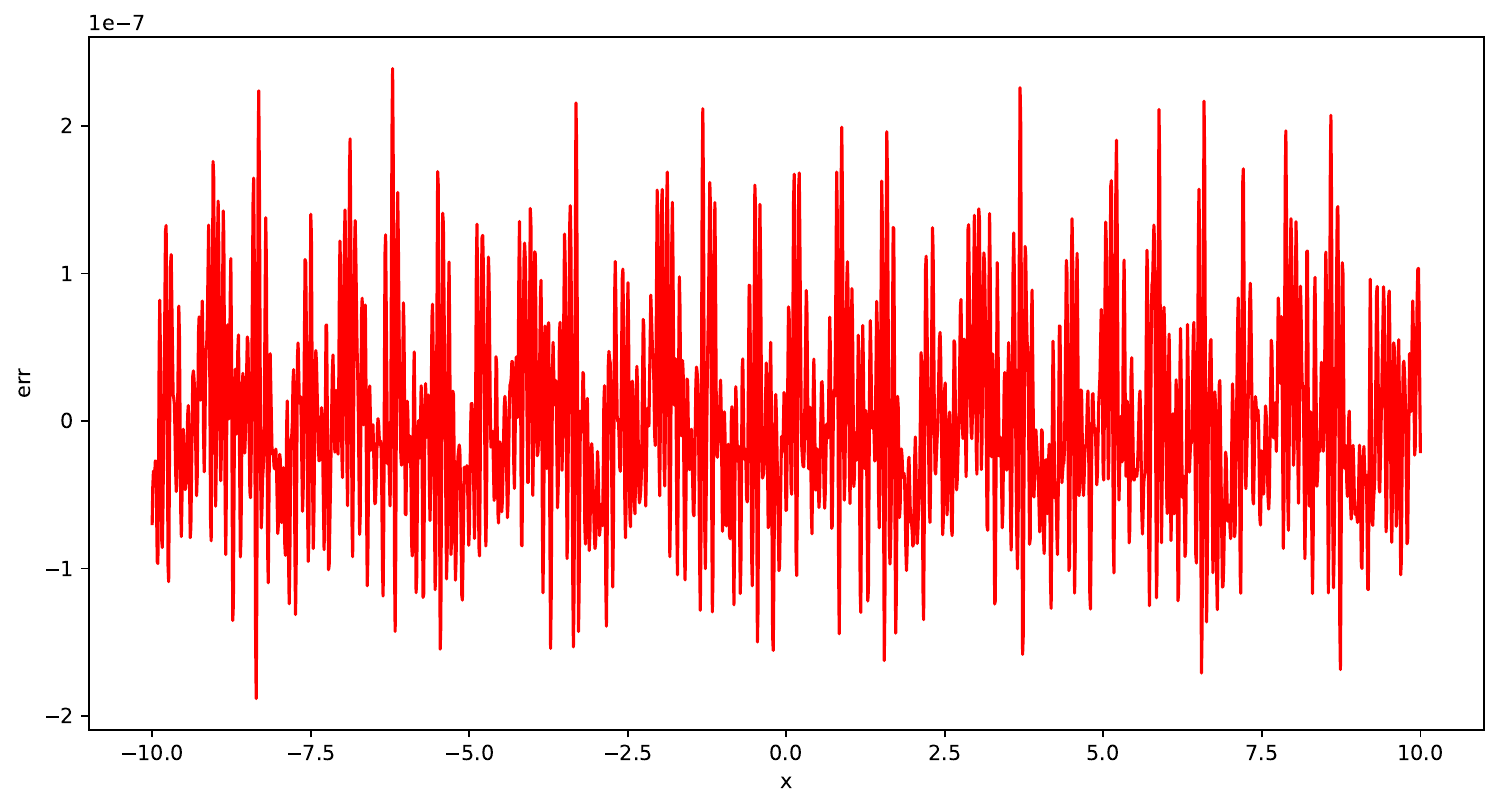}
\caption{The numerical results for Example 1 with loss function (\ref{loss_Residual}), top-left shows the 
TNN approximation $\Psi(\boldsymbol{y};c,\Theta)$ to $U(\boldsymbol{y})$, 
top-right is the figure of the approximation $\Phi(\boldsymbol{x};c,\Theta)$ to 
the quasiperiodic function $u(\boldsymbol{x})$, dow-left 
shows the error $U(\boldsymbol{y})-\Psi(\boldsymbol{y};c, \Theta)$ and 
down-right is the error $u(\boldsymbol{x})-\Phi(\boldsymbol{x};c,\Theta)$.}\label{fig_errors_ex1}
\end{figure}

Figures \ref{fig_errors_ex1_ritz} and \ref{fig_errors_ex1} both shows 
the quasiperiodic property of the solution. 
Based on the numerical experiments, we can observe that employing the residual type of loss function (\ref{loss_Residual}) effective in further reducing the error. Hence, first applying the Ritz type of 
loss function (\ref{loss_Ritz}) and then the residual type of loss function (\ref{loss_Residual}) is a recommended approach for achieving a high-accuracy numerical solution with favorable computational speed.
However, computing the residual loss function (\ref{loss_Residual}) requires more computational effort and memory. For this reason, when the problem include 
many $\mathbb Q$-independent periods, we recommend to only use the Ritz type of 
loss function (\ref{loss_Ritz}) for Step 4 of Algorithm \ref{Algorithm_1}.

\subsection{Example 2}
The second example is concerned with solving the quasiperioric elliptic 
problem (\ref{Elliptic_Equation}) with the following coefficient
\[
\alpha(x) = \cos(2\pi x) + \cos(2\pi^2 x) + 6.
\]
The source term $f(\boldsymbol{x})$ is chosen such that the exact solution is 
\[
u(x) = \sin(2\pi x) + \sin(2\pi^2 x).
\]
In order to use the projection method, we take the following projection matrix
\[
P = \begin{bmatrix} 1 & \pi \end{bmatrix}.
\]
Then the two dimensional coefficient $A(\boldsymbol y)$ 
and periodic function $U(\boldsymbol y)$ can be defined as follows 
\[
A = \cos(2\pi y_1) + \cos(2\pi y_2) + 6,\ \ \ 
U = \sin(2\pi y_1) + \sin(2\pi y_2).
\]
First, we use the loss function defined in (\ref{loss_Ritz}) 
for Step 4 of Algorithm \ref{Algorithm_1}, training the TNN for 1,000 iterations with Adam and subsequently for 500 iterations with LBFGS.
The relative $L^2$ norm error of the approximation $\Phi(\boldsymbol{x};c,\Theta)$ 
to the exact solution $u(\boldsymbol{x})$ is \(5.0018 \times 10^{-6}\).  
The corresponding numerical results are presented in Figure \ref{fig_errors_ex2_ritz}. 
\begin{figure}[ht]
\centering
\includegraphics[width=6cm,height=4.5cm]{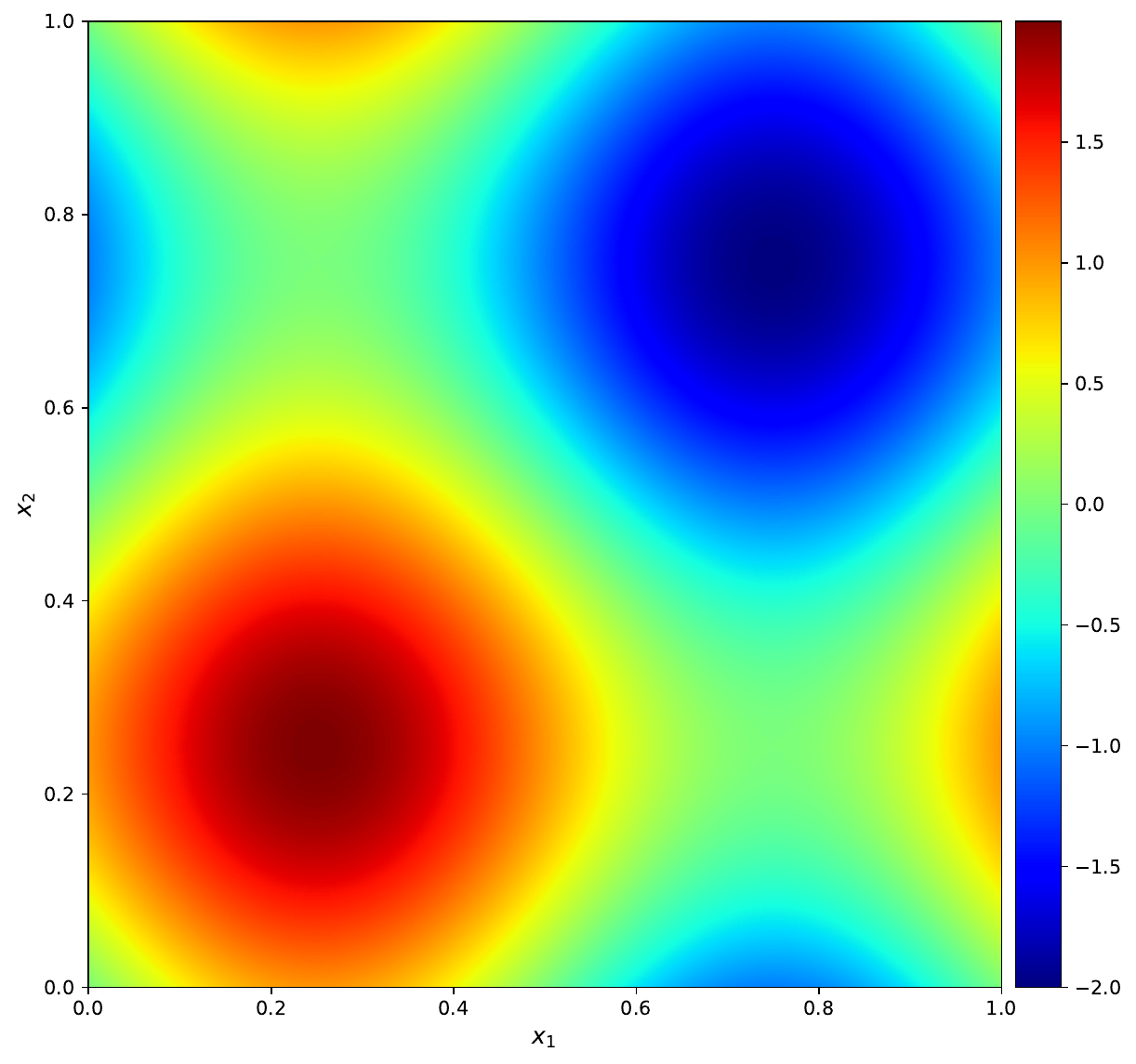}
\includegraphics[width=6cm,height=4.5cm]{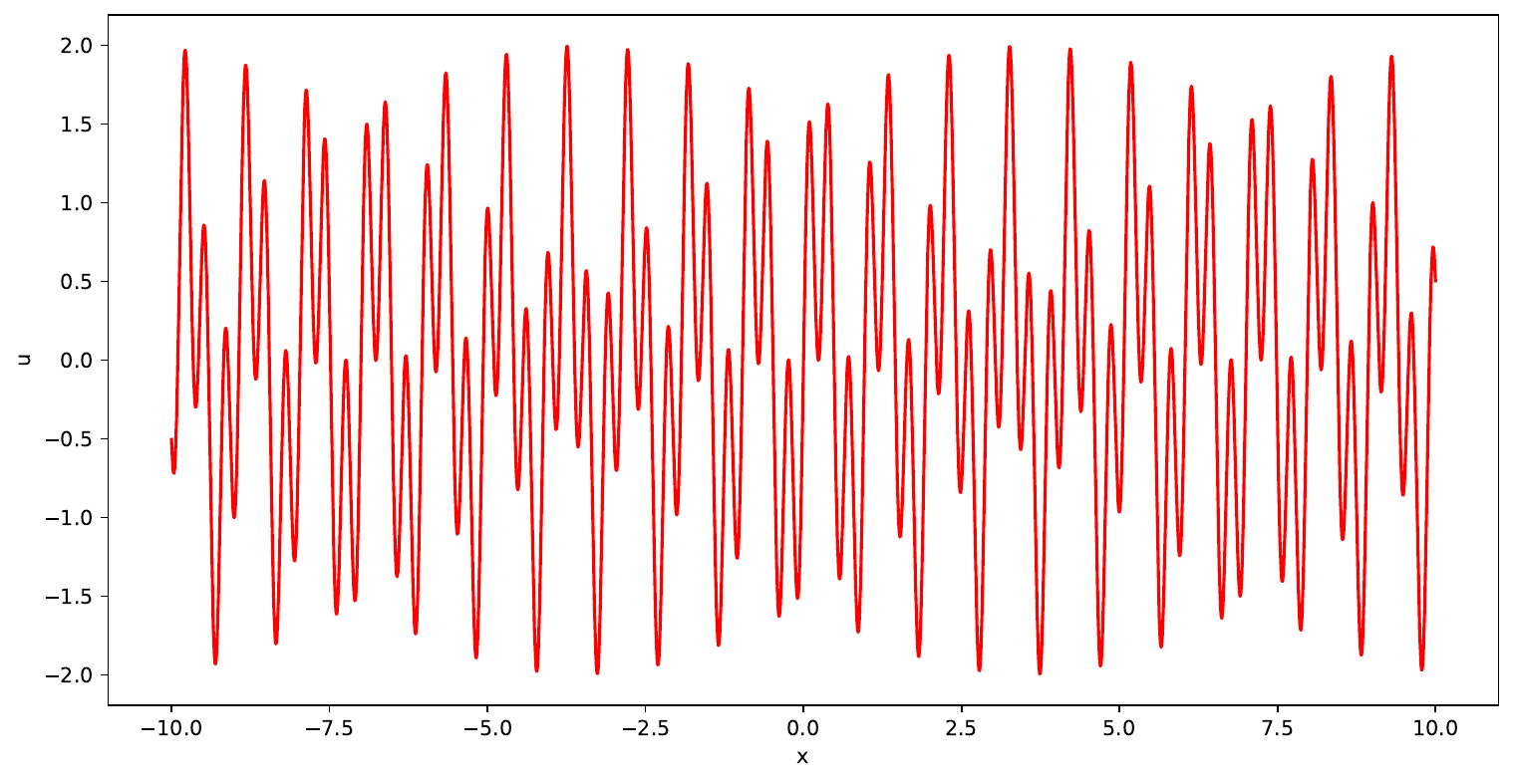}\\
\includegraphics[width=6cm,height=4.5cm]{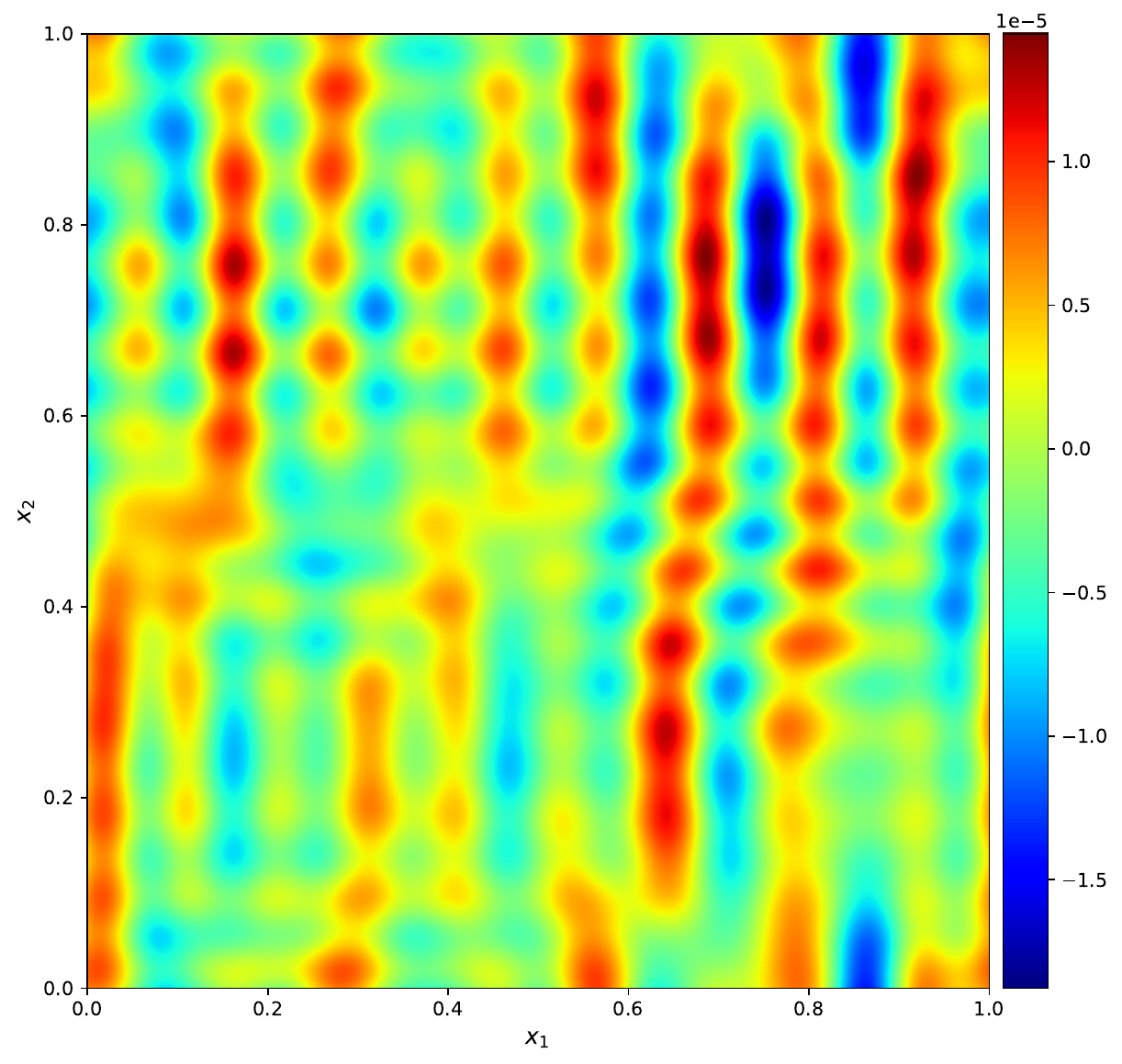}
\includegraphics[width=6cm,height=4.5cm]{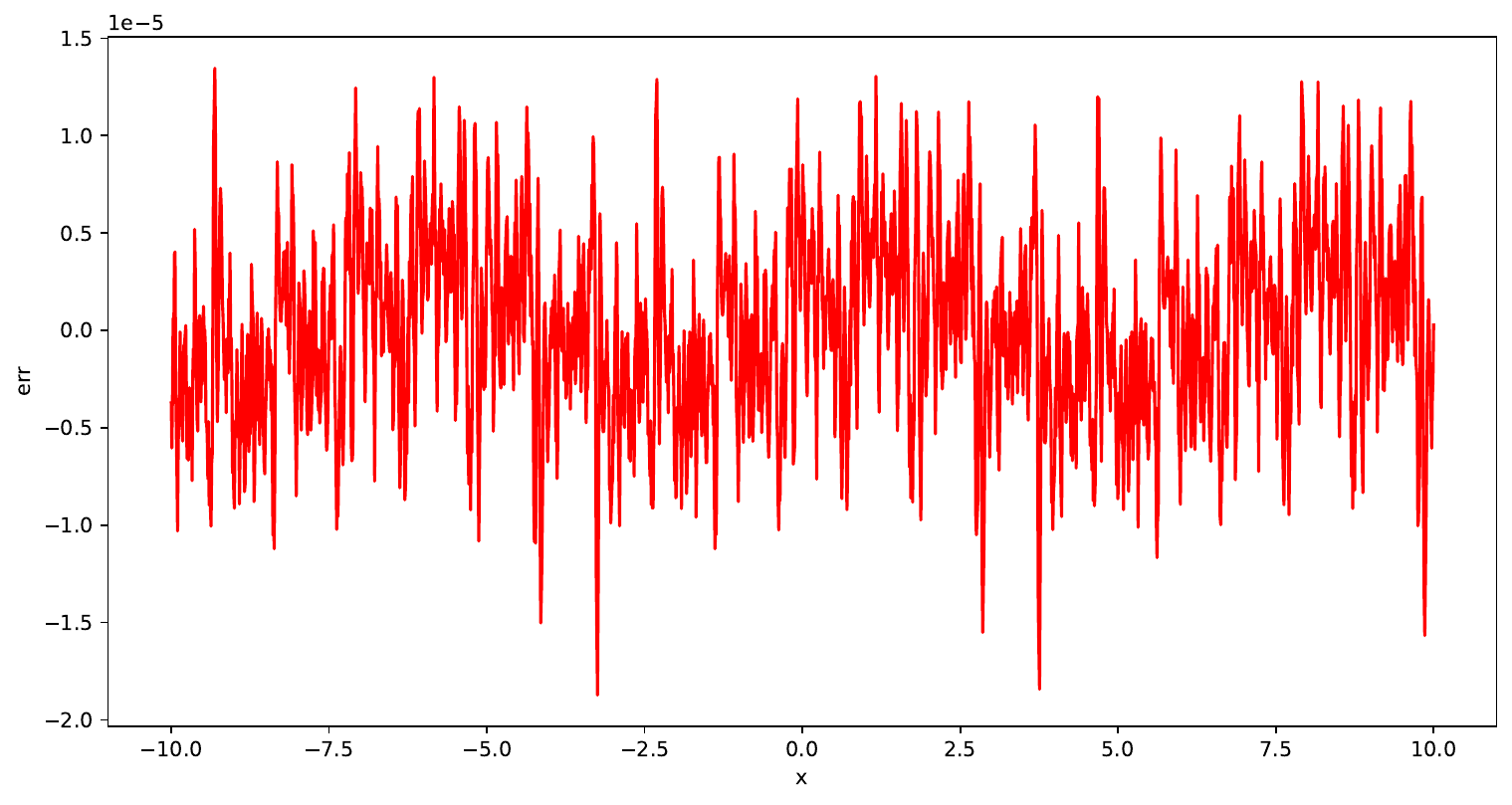}
\caption{The numerical results for Example 2 
with loss function (\ref{loss_Ritz}), top-left shows the 
TNN approximation $\Psi(\boldsymbol{y};c,\Theta)$ to $U(\boldsymbol{y})$, 
top-right is the figure of the approximation $\Phi(\boldsymbol{x};c,\Theta)$ to 
the quasiperiodic function $u(\boldsymbol{x})$, down-left 
shows the error $U(\boldsymbol{y})-\Psi(\boldsymbol{y};c, \Theta)$ and 
down-right is the error $u(\boldsymbol{x})-\Phi(\boldsymbol{x};c,\Theta)$.}\label{fig_errors_ex2_ritz}
\end{figure}

To further improve the accuracy, then we use the loss function defined in (\ref{loss_Residual}) 
for Step 4 of Algorithm \ref{Algorithm_1}, training the TNN for 500 iterations with LBFGS.
The relative $L^2$ norm error of the approximation $\Phi(\boldsymbol{x};c,\Theta)$ 
to the exact solution $u(\boldsymbol{x})$ is \(8.9406 \times 10^{-8}\).  
The corresponding numerical results are presented in Figure \ref{fig_errors_ex2}. 
\begin{figure}[ht]
\centering
\includegraphics[width=6cm,height=4.5cm]{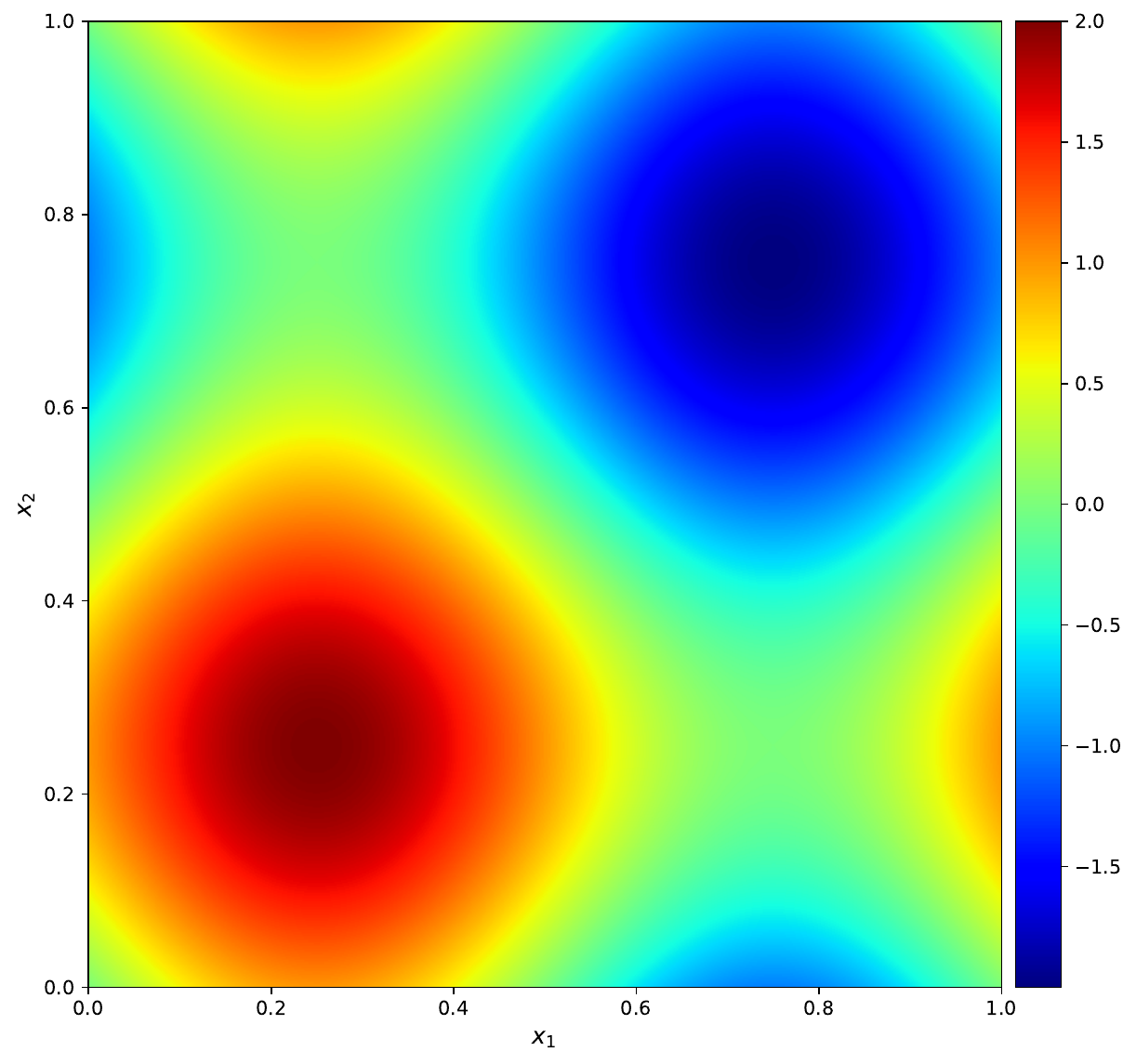}
\includegraphics[width=6cm,height=4.5cm]{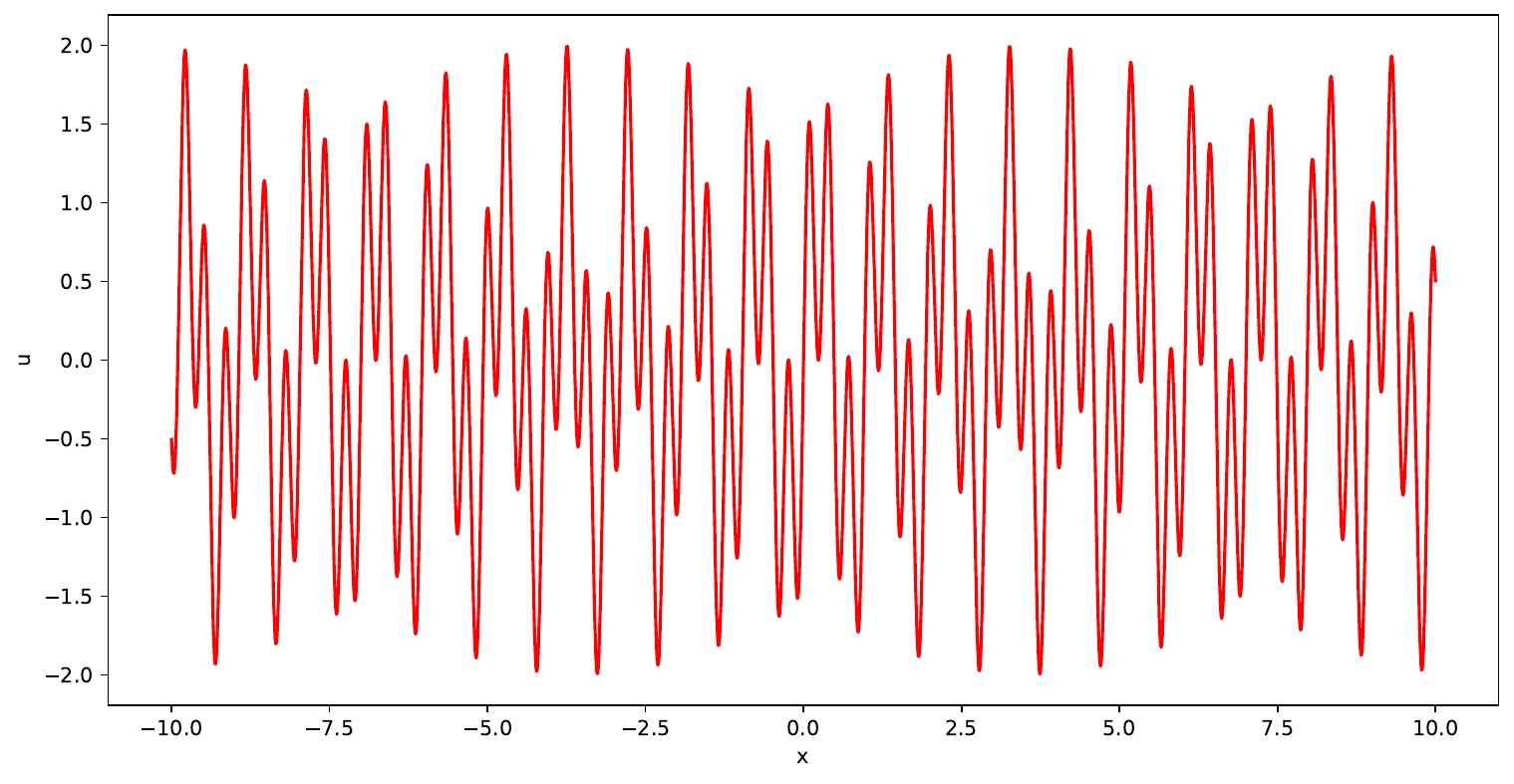}\\
\includegraphics[width=6cm,height=4.5cm]{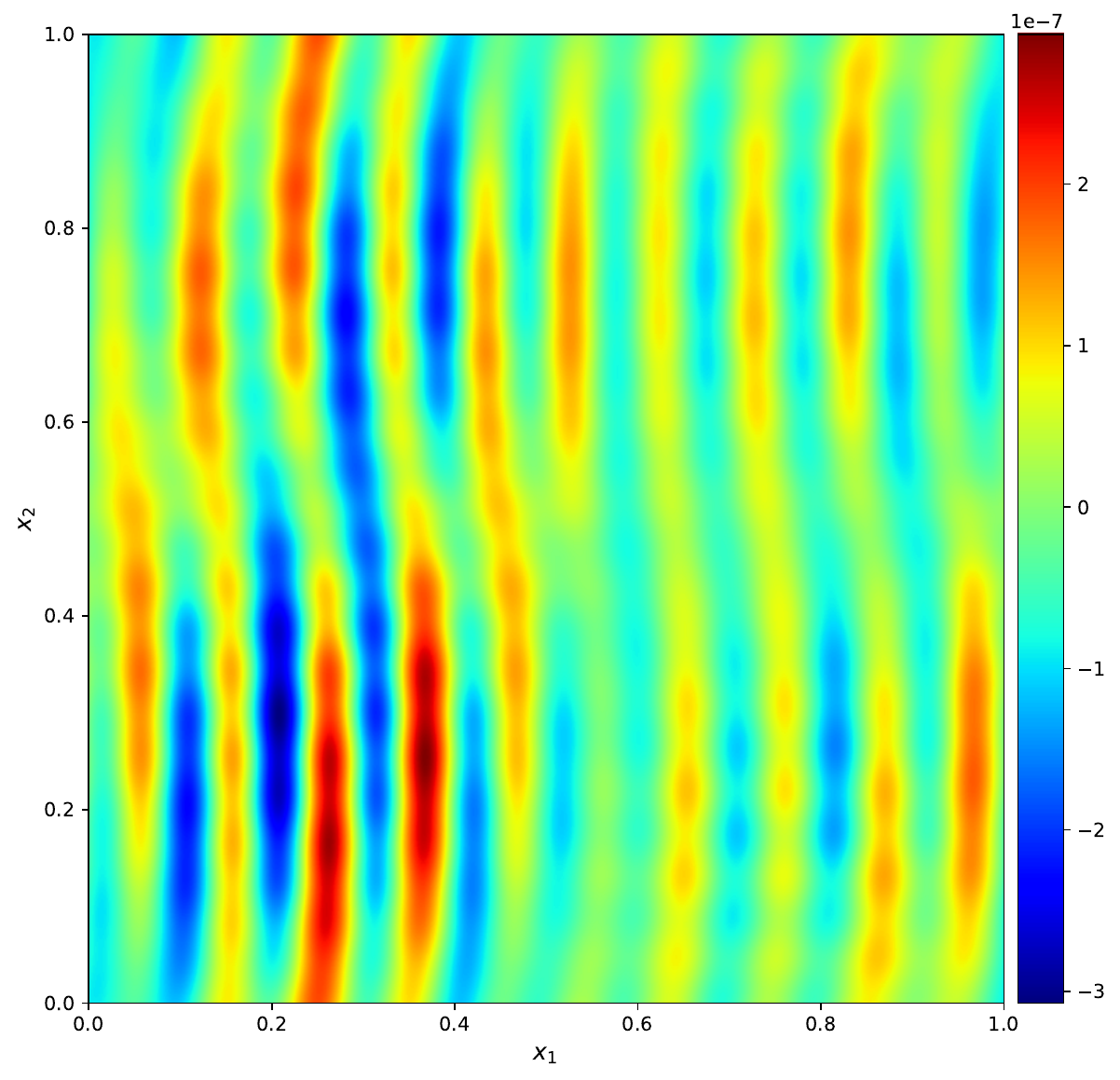}
\includegraphics[width=6cm,height=4.5cm]{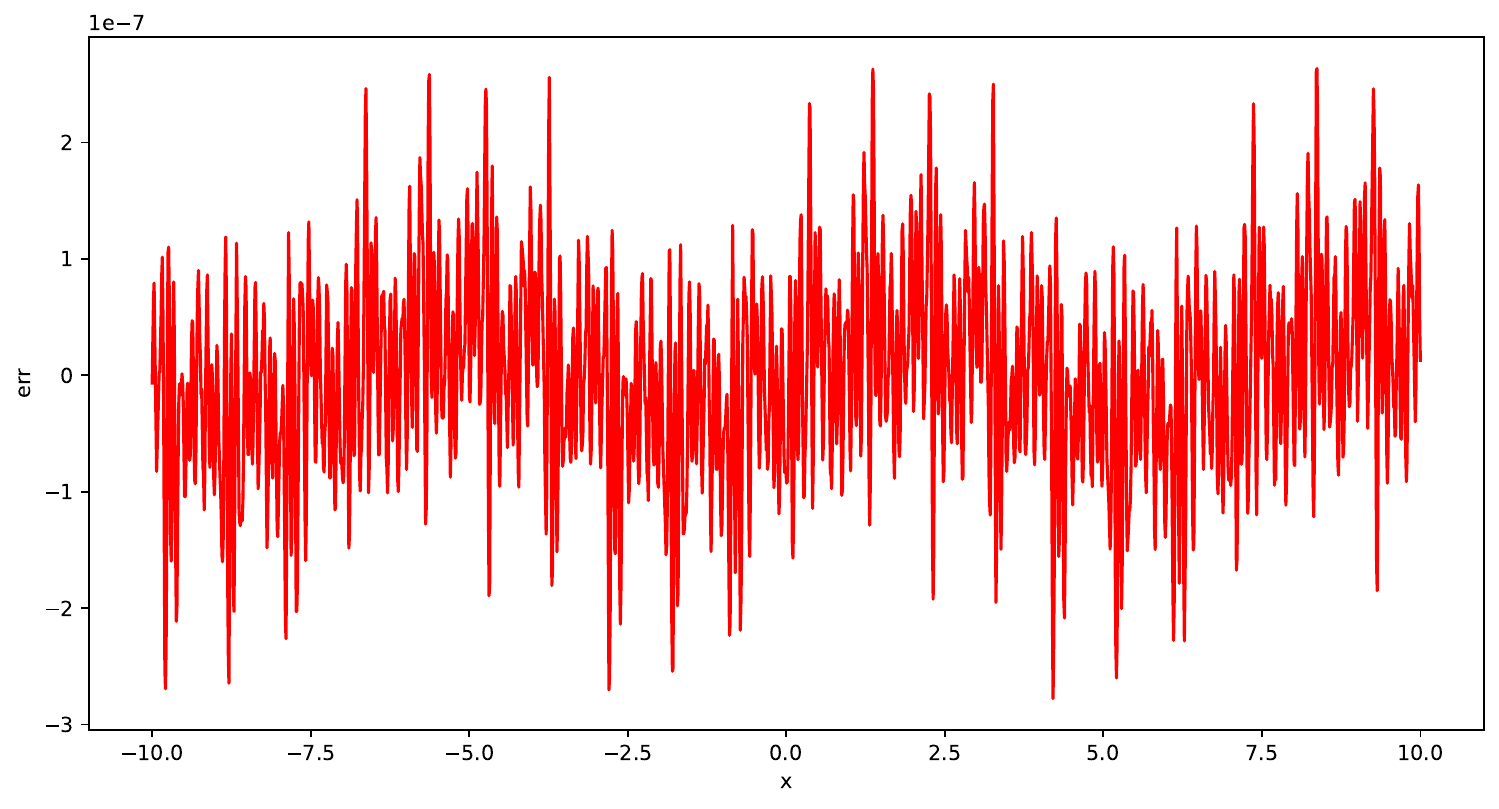}
\caption{The numerical results for Example 2 with loss function (\ref{loss_Residual}), top-left shows the 
TNN approximation $\Psi(\boldsymbol{y};c,\Theta)$ to $U(\boldsymbol{y})$, 
top-right is the figure of the approximation $\Phi(\boldsymbol{x};c,\Theta)$ to 
the quasiperiodic function $u(\boldsymbol{x})$, dow-left 
shows the error $U(\boldsymbol{y})-\Psi(\boldsymbol{y};c, \Theta)$ and 
down-right is the error $u(\boldsymbol{x})-\Phi(\boldsymbol{x};c,\Theta)$.}\label{fig_errors_ex2}
\end{figure}

\subsection{Example 3}
This example is concerned with solving two dimensional 
quasiperiodic elliptic problem with the following 
coefficient 
\[
\alpha(\boldsymbol{x}) = \cos(2\pi x_1) + \cos(2\pi \sqrt{2} x_2) 
+ \cos(2\pi x_2) + \cos(2\pi \sqrt{3} x_2) + 12.
\]
The source term $f(\boldsymbol x)$ is chosen such that the exact solution is 
\[
u = \sin(2\pi x_1) + \sin(2\pi \sqrt{2} x_1) + \sin(2\pi x_2) + \sin(2\pi \sqrt{3} x_2).
\]
Then the projection matrix should be 
\[
P = 
\begin{bmatrix} 
1 & \sqrt{2} & 0 & 0 \\ 
0 & 0 & 1 & \sqrt{3} 
\end{bmatrix},
\]
The higher dimensional coefficient $A(\boldsymbol{y})$ and exact solution $U(\boldsymbol{y})$ are defined as follows 
\begin{eqnarray*}
&&A(\boldsymbol{y}) = \cos(2\pi y_1) + \cos(2\pi y_2) + \cos(2\pi y_3) + \cos(2\pi y_4) + 12,\\
&&U(\boldsymbol{y}) = \sin(2\pi y_1) + \sin(2\pi y_2) + \sin(2\pi y_3) + \sin(2\pi y_4).
\end{eqnarray*}

First, we use the loss function defined in (\ref{loss_Ritz}) 
for Step 4 of Algorithm \ref{Algorithm_1}, training the TNN for 1,000 iterations with Adam and subsequently for 1000 iterations with LBFGS.
The relative $L^2$ norm error of the approximation $\Phi(\boldsymbol{x};c,\Theta)$ 
to the exact solution $u(\boldsymbol{x})$ is \(6.2215 \times 10^{-6}\).  
The corresponding numerical results are presented in Figure \ref{fig_errors_ex3_ritz}.  
\begin{figure}[ht]
\centering
\includegraphics[width=6cm,height=5cm]{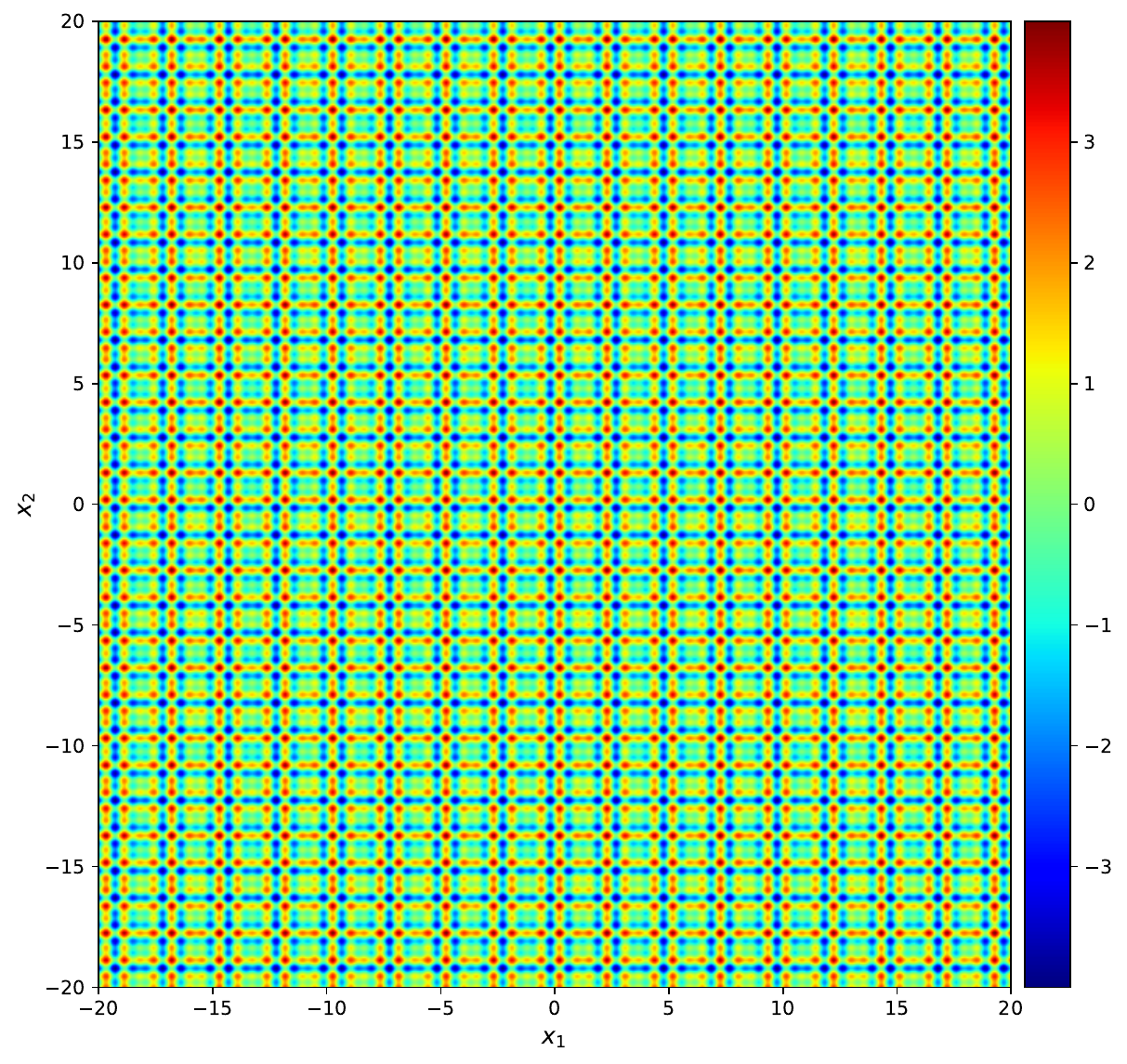}
\includegraphics[width=6cm,height=5cm]{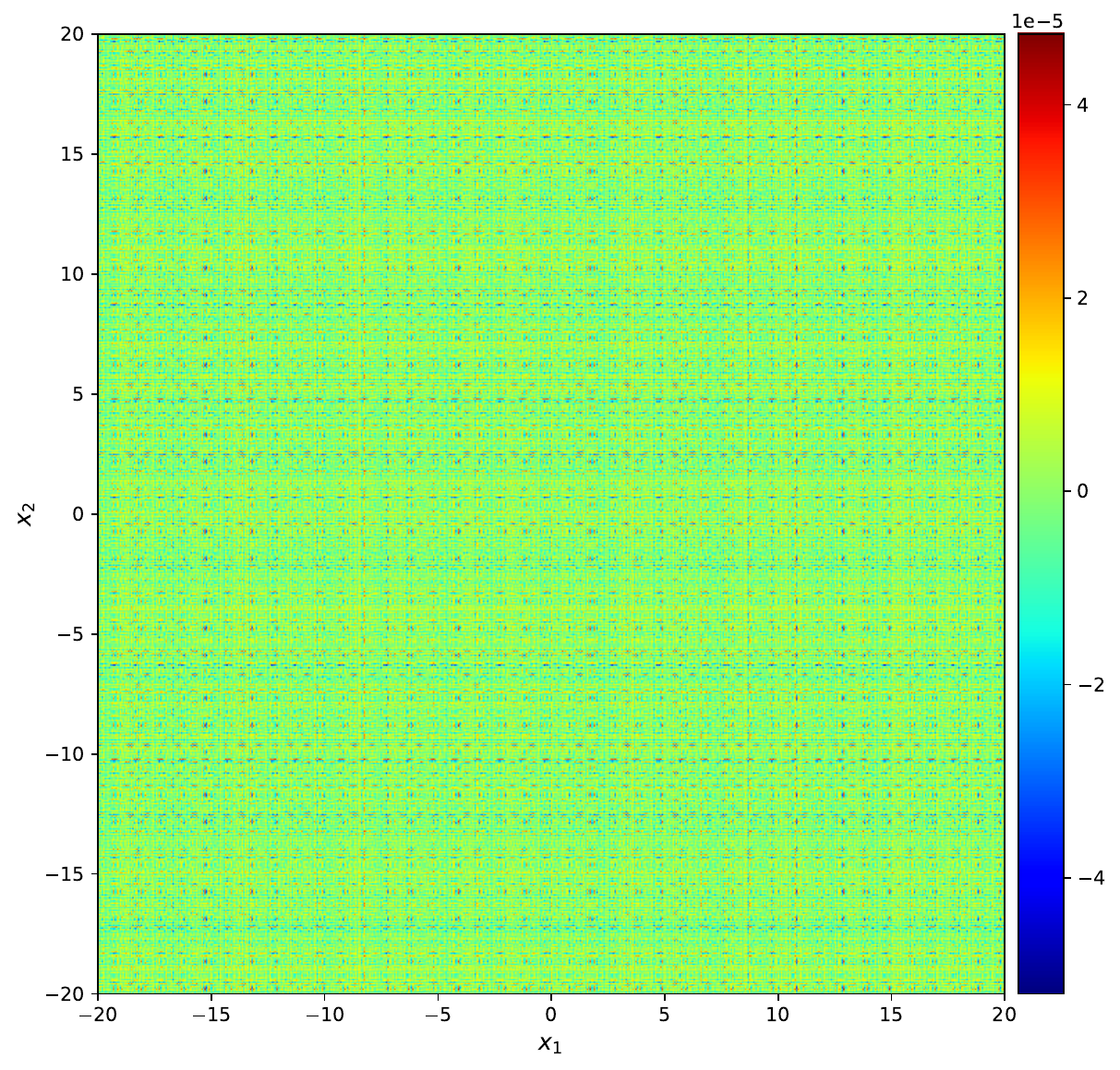}
\caption{The numerical results for Example 3 with loss function (\ref{loss_Ritz}), left shows 
the figure of the approximation $\Phi(P^\top \boldsymbol{x};\Theta)$ to 
the quasiperiodic funtion $u(\boldsymbol{x})$, 
right shows the error $u(\boldsymbol{x})-\Phi(P^\top  \boldsymbol{x})$.}\label{fig_errors_ex3_ritz}
\end{figure}

To further improve the accuracy, then we use the loss function defined in (\ref{loss_Residual}) 
for Step 4 of Algorithm \ref{Algorithm_1}, training the TNN for 3000 iterations with LBFGS.
The relative $L^2$ norm error of the approximation $\Phi(\boldsymbol{x};c,\Theta)$ 
to the exact solution $u(\boldsymbol{x})$ is \(2.1950 \times 10^{-7}\).  
The corresponding numerical results are presented in Figure \ref{fig_errors_ex3}.
\begin{figure}[ht]
\centering
\includegraphics[width=6cm,height=5cm]{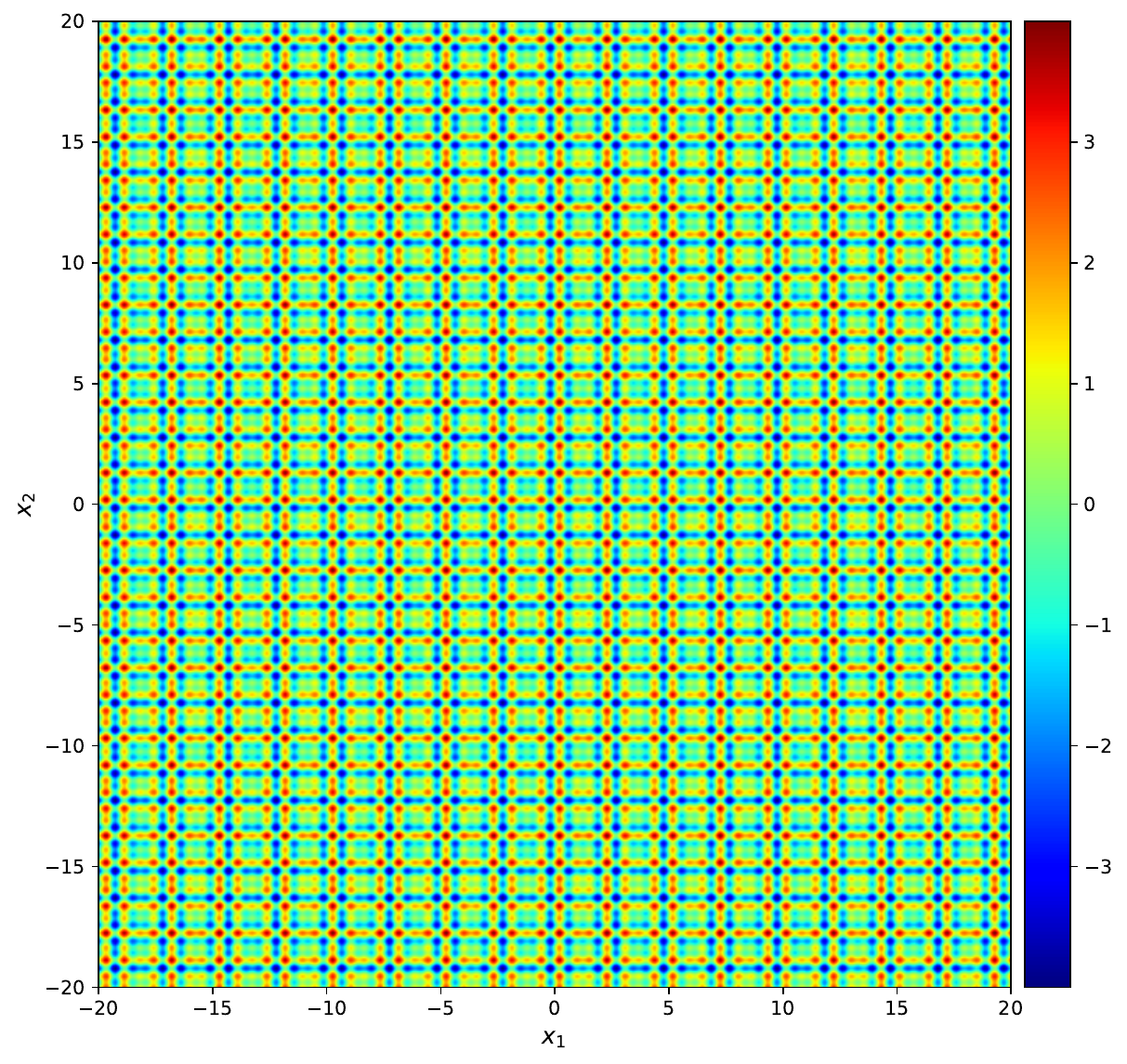}
\includegraphics[width=6cm,height=5cm]{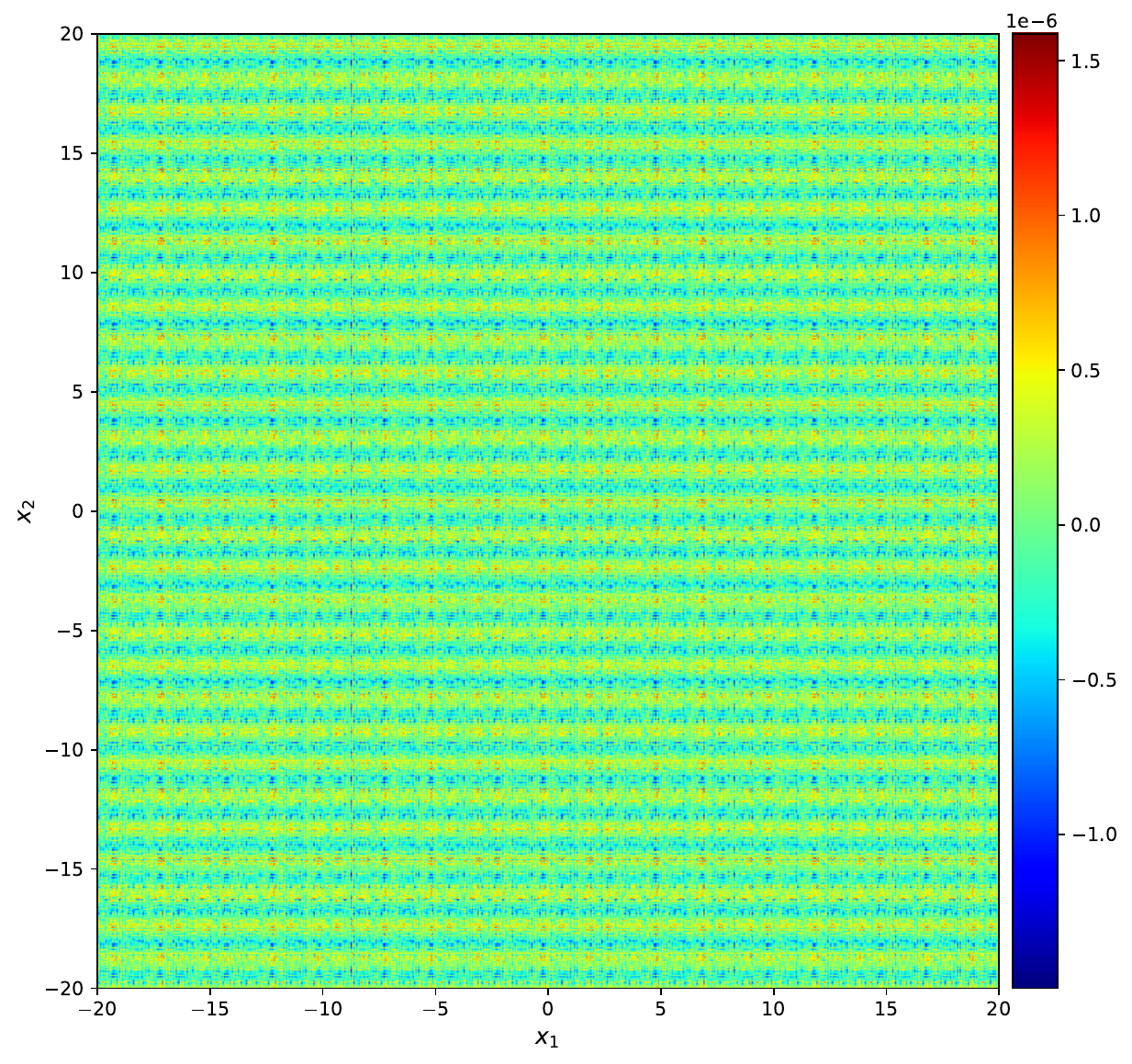}
\caption{The numerical results for Example 3 with loss function (\ref{loss_Residual}), left shows 
the figure of the approximation $\Phi(P^\top \boldsymbol{x};\Theta)$ to 
the quasiperiodic funtion $u(\boldsymbol{x})$, 
right shows the error $u(\boldsymbol{x})-\Phi(P^\top  \boldsymbol{x})$.}\label{fig_errors_ex3}
\end{figure}

\subsection{Example 4}
In this example, we come to investigate the experience of TNN based machine learning method 
for 5 $\mathbb Q$-independent periods. It means that the quasiperiodic 
coefficient $\alpha(x)$ for (\ref{Elliptic_Equation})
is set to be  
\[
\alpha = \cos(2\pi x) + \cos(2\pi \sqrt{2} x) + \cos(2\pi \sqrt{3} x) 
+ \cos(2\pi \sqrt{5} x) + \cos(2\pi \sqrt{7} x) + 12.
\]
Here, the source term $f(\boldsymbol{x})$ is chosen such that the exact solution is 
\[
u = \sin(2\pi x) + \sin(2\pi \sqrt{2} x) + \sin(2\pi \sqrt{3} x) 
+ \sin(2\pi \sqrt{5} x) + \sin(2\pi \sqrt{7} x).
\]
In order to use the projection method, the projection matrix should be 
\[
P = \begin{bmatrix} 1 & \sqrt{2} & \sqrt{3} & \sqrt{5} & \sqrt{7} \end{bmatrix}.
\]
Then the higher dimensional coefficient $A(\boldsymbol{y})$ and 
exact solution $U(\boldsymbol{y})$ are 
\begin{eqnarray*}
&&A = \cos(2\pi y_1) + \cos(2\pi y_2) + \cos(2\pi y_3) + \cos(2\pi y_4) + \cos(2\pi y_5) + 12,\\
&&U = \sin(2\pi y_1) + \sin(2\pi y_2) + \sin(2\pi y_3) + \sin(2\pi y_4) + \sin(2\pi y_5).
\end{eqnarray*}

First, we use the loss function defined in (\ref{loss_Ritz}) 
for Step 4 of Algorithm \ref{Algorithm_1}, training the TNN for 1,000 iterations with Adam and subsequently for 1000 iterations with LBFGS.
The relative $L^2$ norm error of the approximation $\Phi(\boldsymbol{x};c,\Theta)$ 
to the exact solution $u(\boldsymbol{x})$ is \(1.2018 \times 10^{-5}\).  
The corresponding numerical results are presented in Figure \ref{fig_errors_ex4_ritz}.  
\begin{figure}[ht]
\centering
\includegraphics[width=6cm,height=5cm]{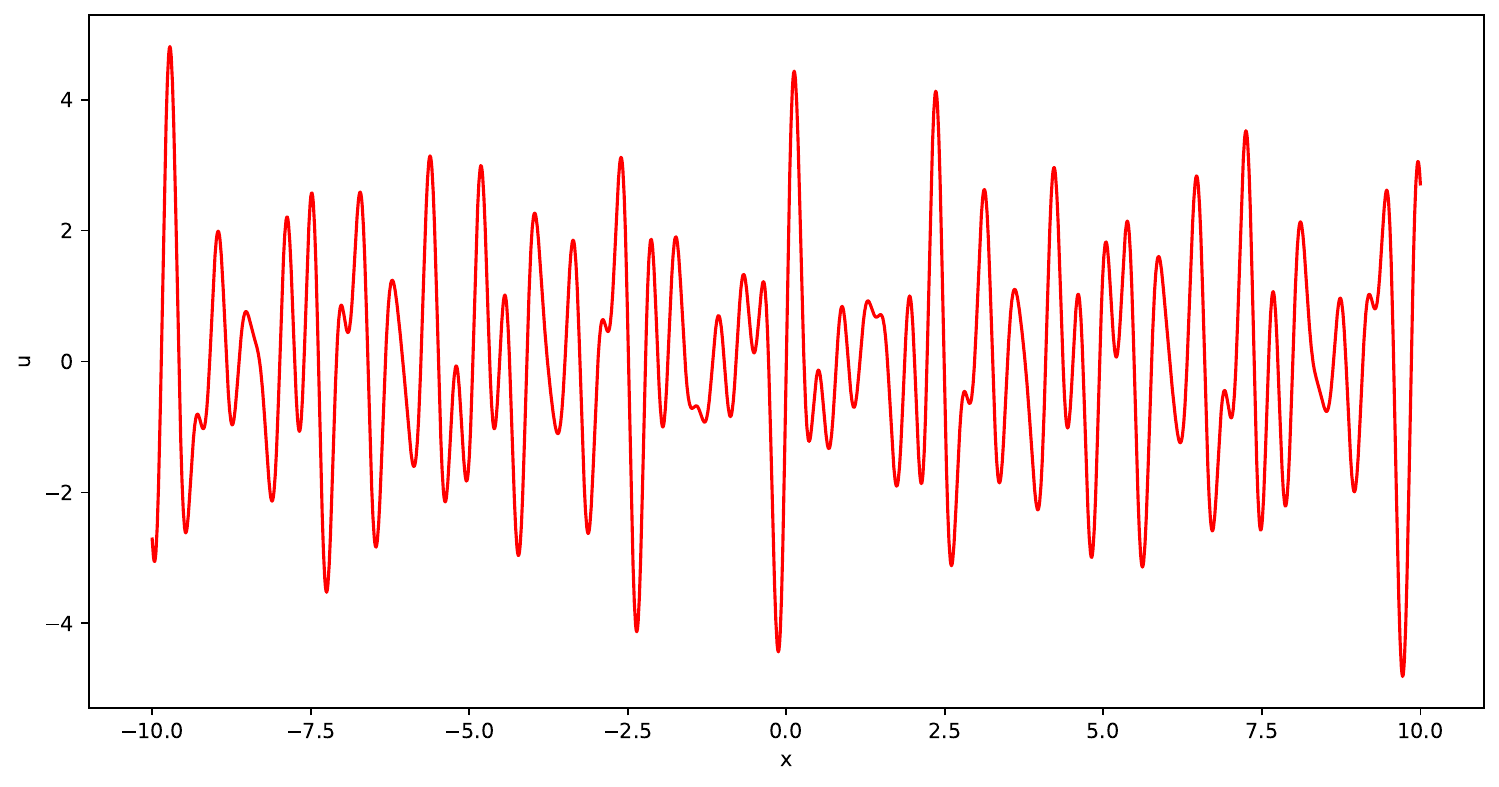}
\includegraphics[width=6cm,height=5cm]{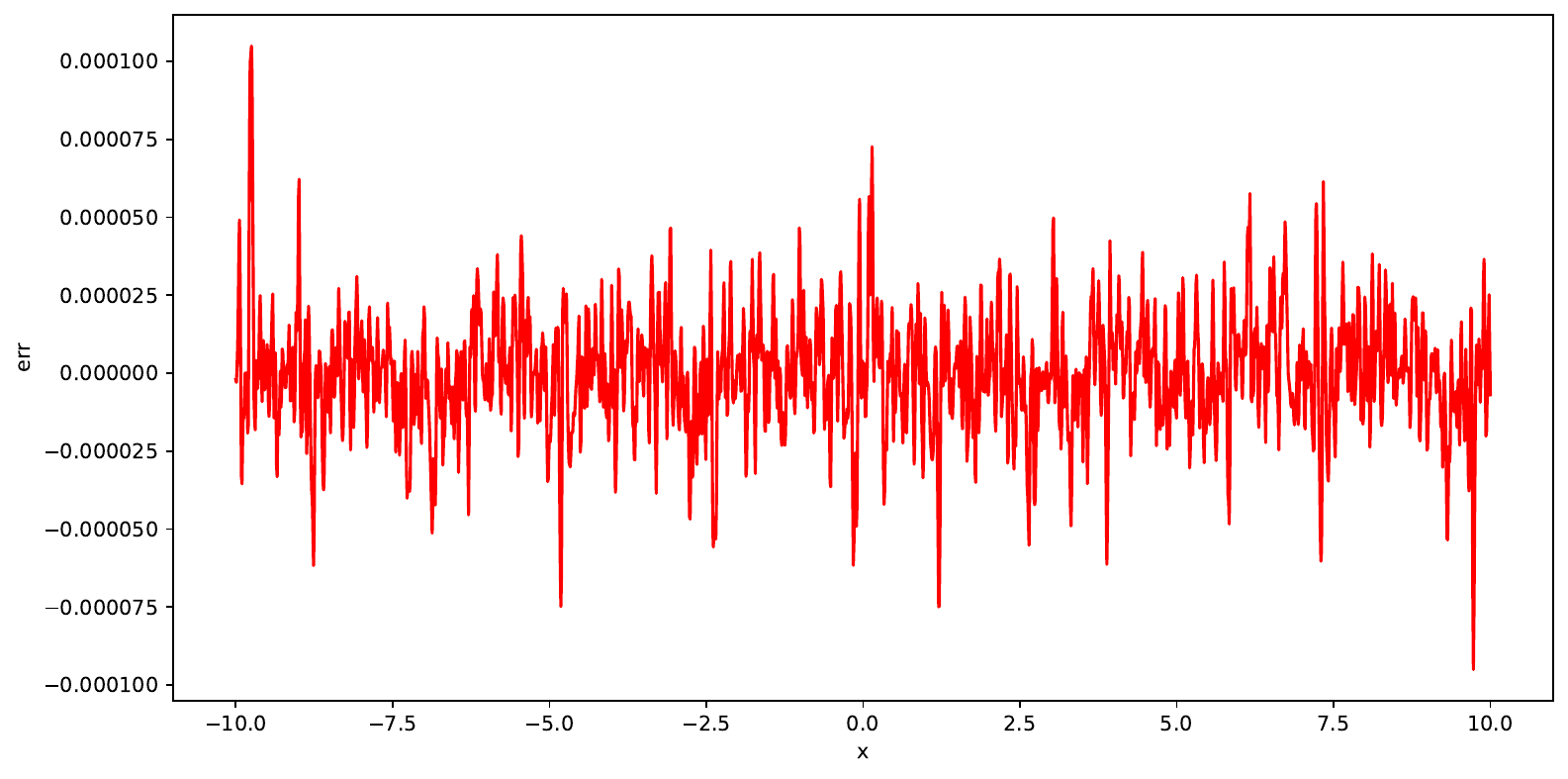}
\caption{The numerical results for Example 4 with loss function (\ref{loss_Ritz}), left shows 
the figure of the approximation $\Phi(\boldsymbol{x};c,\Theta)$ to 
the quasiperiodic funtion $u(\boldsymbol{x})$, 
right shows the error $u(\boldsymbol{x})-\Phi(\boldsymbol{x};c,\Theta)$.}\label{fig_errors_ex4_ritz}
\end{figure}

To further improve the accuracy, then we use the loss function defined in (\ref{loss_Residual}) 
for Step 4 of Algorithm \ref{Algorithm_1}, training the TNN for 3000 iterations with LBFGS.
The relative $L^2$ norm error of the approximation $\Phi(\boldsymbol{x};c,\Theta)$ 
to the exact solution $u(\boldsymbol{x})$ is \(3.2539 \times 10^{-6}\).  
The corresponding numerical results are presented in Figure \ref{fig_errors_ex4}.
\begin{figure}[ht]
\centering
\includegraphics[width=6cm,height=5cm]{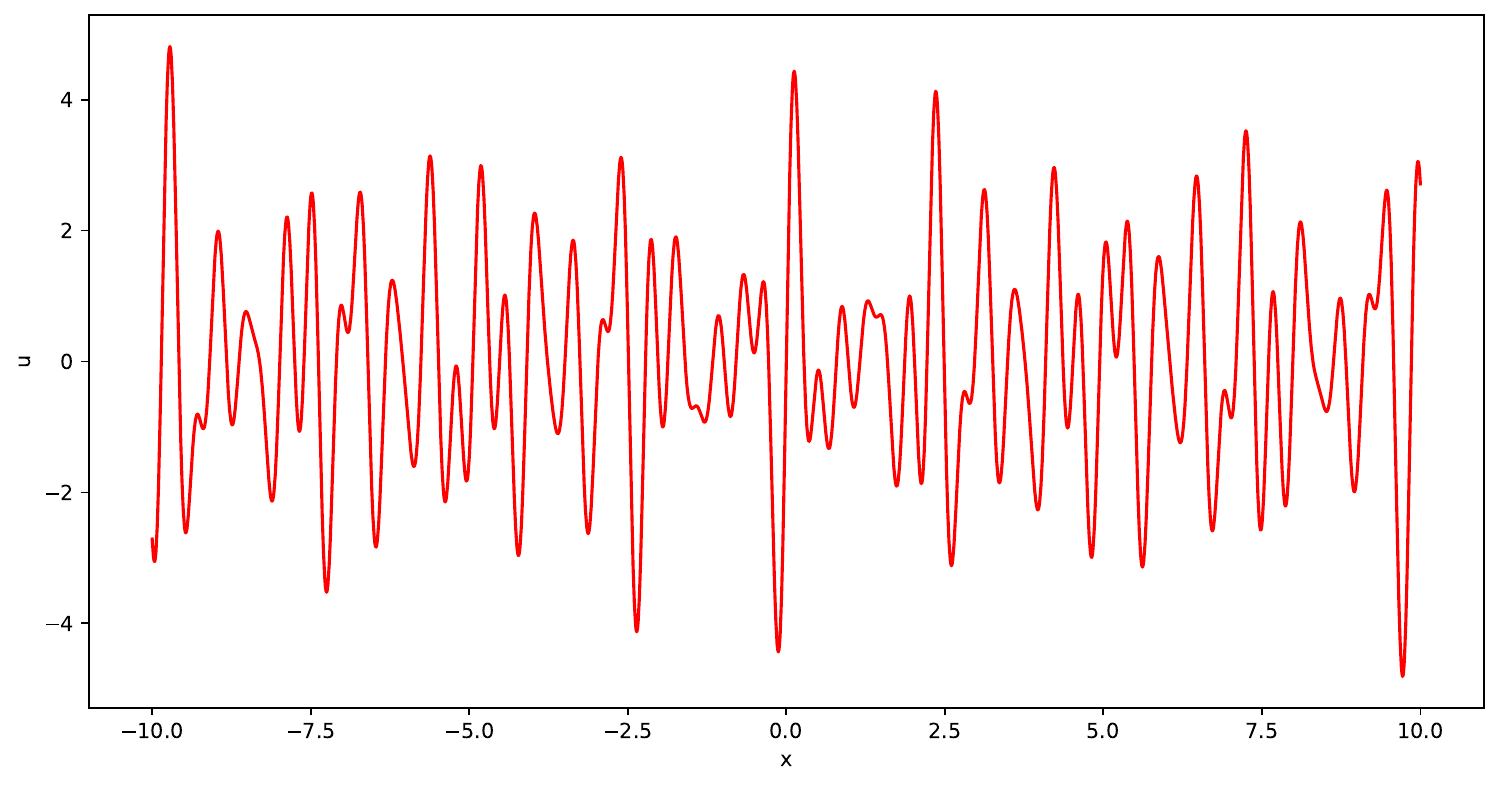}
\includegraphics[width=6cm,height=5cm]{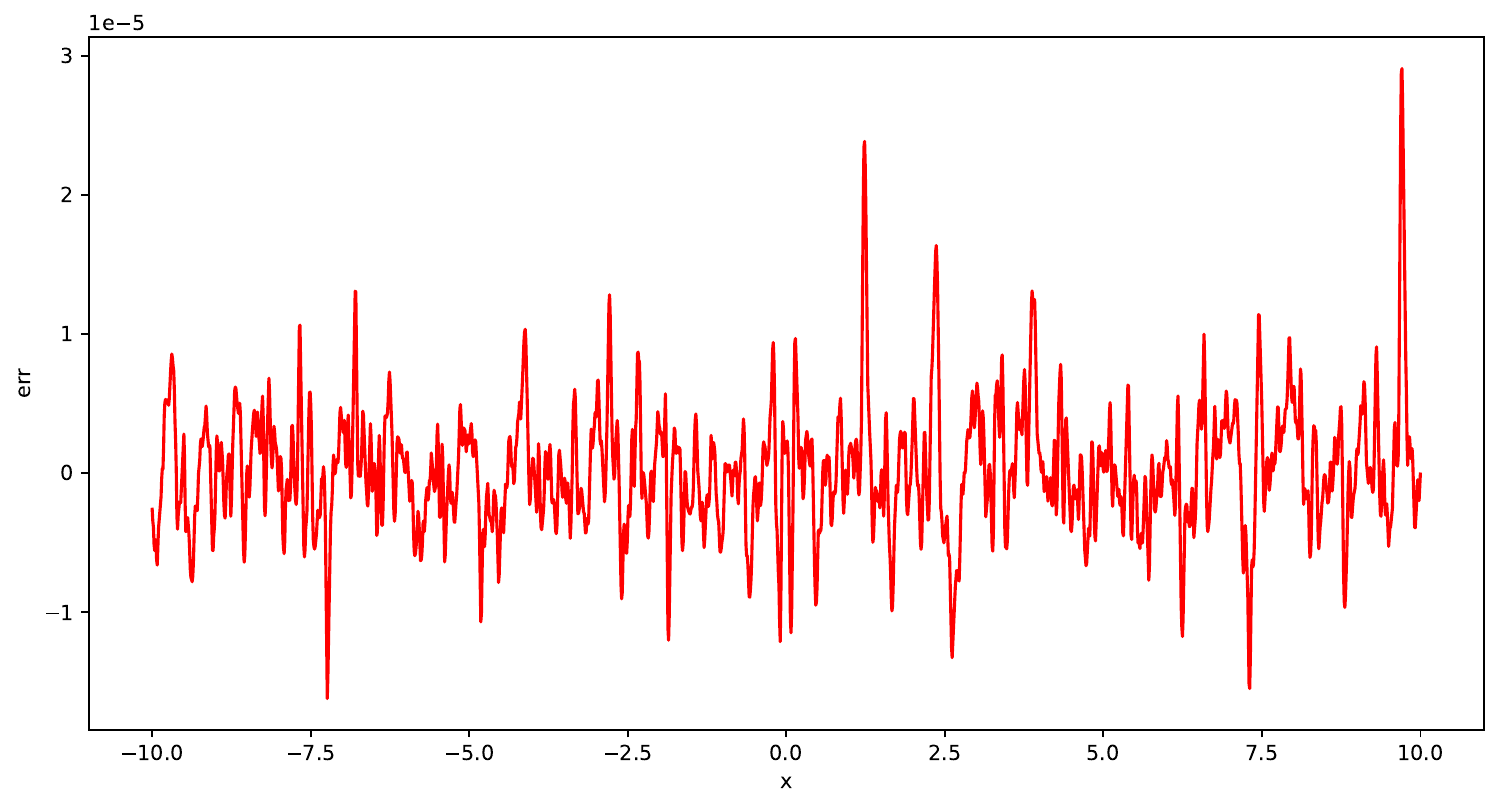}
\caption{The numerical results for Example 4 with loss function (\ref{loss_Residual}), left shows 
the figure of the approximation $\Phi(\boldsymbol{x};c,\Theta)$ to 
the quasiperiodic funtion $u(\boldsymbol{x})$, 
right shows the error $u(\boldsymbol{x})-\Phi(\boldsymbol{x};c,\Theta)$.}\label{fig_errors_ex4}
\end{figure}

\subsection{Example 5}\label{Example2_Ritz}
In this example, we consider the quasiperiodic elliptic problem (\ref{Elliptic_Equation}) 
with $10$ $\mathbb Q$-independent periodic coefficient 
\[
\begin{aligned}
\alpha &= \cos(2\pi x) + \cos(2\pi \sqrt{2} x) + \cos(2\pi \sqrt{3} x) + \cos(2\pi \sqrt{5} x) 
+ \cos(2\pi \sqrt{7} x) \\
&\quad + \cos(2\pi \sqrt{11} x) + \cos(2\pi \sqrt{13} x) + \cos(2\pi \sqrt{17} x) 
+ \cos(2\pi \sqrt{19} x) + \cos(2\pi \sqrt{23} x) + 12. 
\end{aligned}
\]
The source term can be chosen such that the exact solution is 
\[
\begin{aligned}
u &= \sin(2\pi x) + \sin(2\pi \sqrt{2} x) + \sin(2\pi \sqrt{3} x) + \sin(2\pi \sqrt{5} x) 
+ \sin(2\pi \sqrt{7} x) \\
&\quad + \sin(2\pi \sqrt{11} x) + \sin(2\pi \sqrt{13} x) + \sin(2\pi \sqrt{17} x) 
+ \sin(2\pi \sqrt{19} x) + \sin(2\pi \sqrt{23} x).
\end{aligned}
\]
Then the projection matrix should be 
\[
P = \begin{bmatrix} 
1 & \sqrt{2} & \sqrt{3} & \sqrt{5} & \sqrt{7} & \sqrt{11} & \sqrt{13} & \sqrt{17} & \sqrt{19} & \sqrt{23} \end{bmatrix}.
\]
Then the corresponding higher dimensional coefficient $A(\boldsymbol{y})$ and 
exact solution $U(\boldsymbol{y})$ are 
\begin{eqnarray*}
A(\boldsymbol{y}) = \sum_{i=1}^{10} \cos(2\pi y_i) + 12, \ \ \ 
U(\boldsymbol{y}) = \sum_{i=1}^{10} \sin(2\pi y_i).
\end{eqnarray*}

We use the loss function defined in (\ref{loss_Ritz}) 
for Step 4 of Algorithm \ref{Algorithm_1}, training the TNN for 1,000 iterations with Adam and subsequently for 2000 iterations with LBFGS.
The relative $L^2$ norm error of the approximation $\Phi(\boldsymbol{x};c,\Theta)$ 
to the exact solution $u(\boldsymbol{x})$ is \(1.5128 \times 10^{-5}\).  
The corresponding numerical results are presented in Figure \ref{fig_errors_ex5_ritz}.  
\begin{figure}[ht]
\centering
\includegraphics[width=6cm,height=5cm]{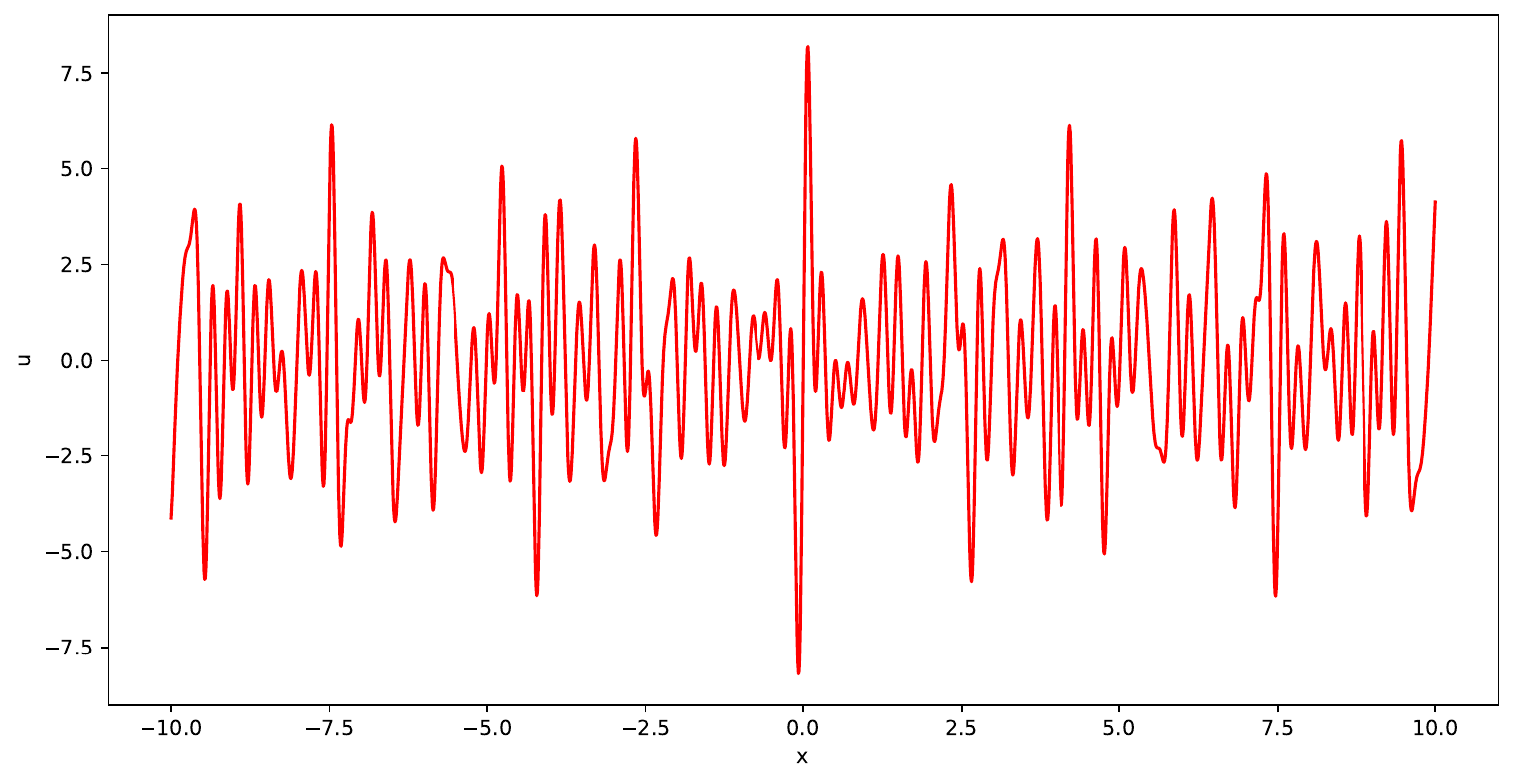}
\includegraphics[width=6cm,height=5cm]{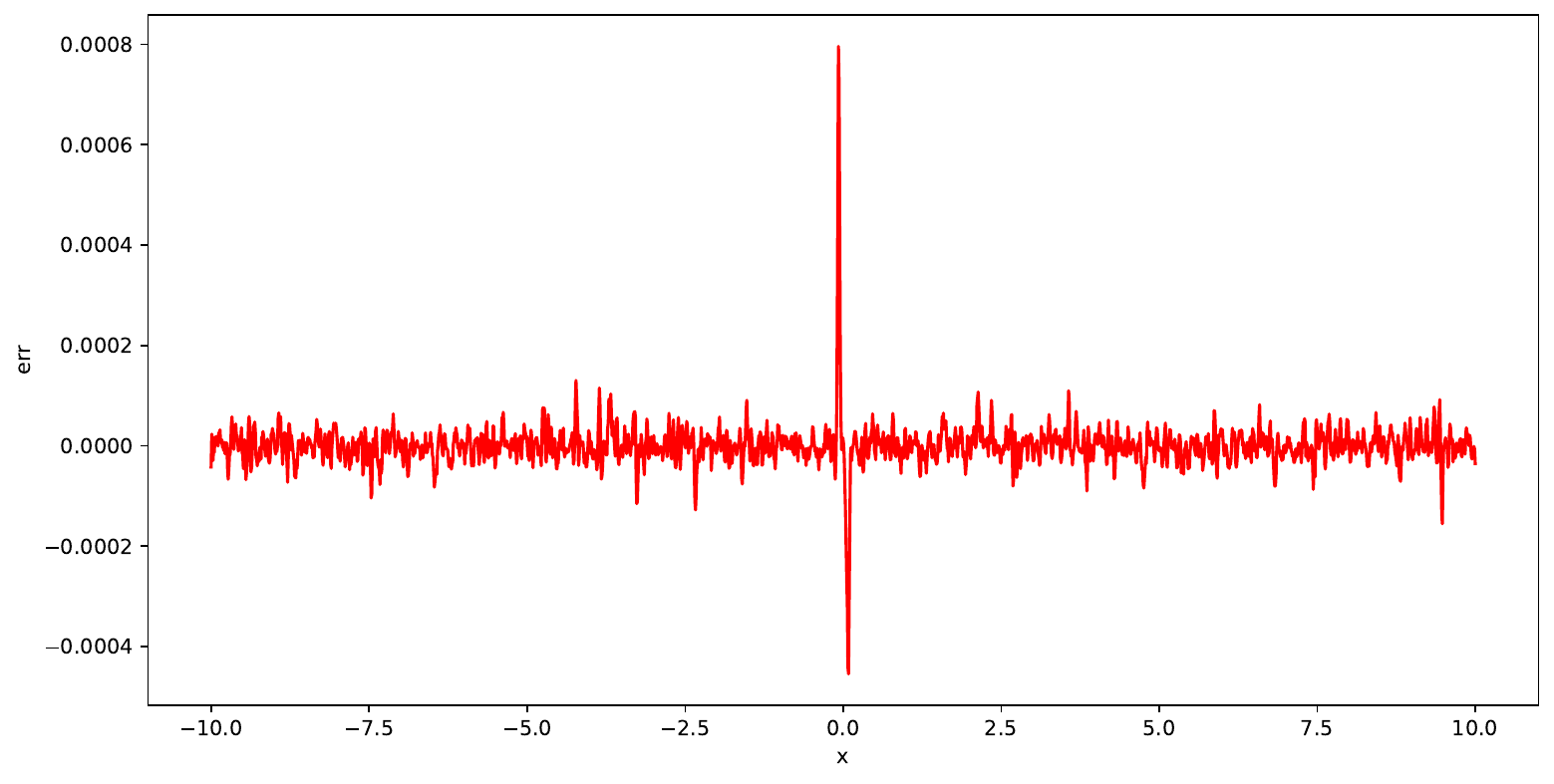}
\caption{The numerical results for Example 5 with loss function (\ref{loss_Ritz}), left shows 
the figure of the approximation $\Phi(\boldsymbol{x};c,\Theta)$ to 
the quasiperiodic funtion $u(\boldsymbol{x})$, 
right shows the error $u(\boldsymbol{x})-\Phi(\boldsymbol{x};c,\Theta)$.}\label{fig_errors_ex5_ritz}
\end{figure}

\subsection{Example 6}
In this example,  we solve the two dimensional quasiperiodic elliptic 
problem with the following 
coefficient 
\[
\alpha = \sum_{i=0}^{5} \cos\left( 2\pi (\cos(i) x_1 + \sin(i) x_2) \right) + 12.
\]
The source term $f(\boldsymbol{x})$ in (\ref{Elliptic_Equation}) is chosen such that the exact solution is 
\[
u = \sum_{i=0}^{5} \sin\left( 2\pi (\cos(i) x_1 + \sin(i) x_2) \right).
\]
Based on the coefficient $\alpha$, we should define the projection matrix as follows 
\[
P = \begin{bmatrix} 1 & \cos(1) & \cos(2) & \cos(3) & \cos(4) & \cos(5) \\ 
0 & \sin(1) & \sin(2) & \sin(3) & \sin(4) & \sin(5) \end{bmatrix},
\]
Then the corresponding higher dimensional coefficient $A(\boldsymbol{y})$ 
and the exact solution $U(\boldsymbol{y})$ are given by 
\[
A = \sum_{i=1}^{6} \cos(2\pi y_i) + 12, \ \ \ 
U = \sum_{i=1}^{6} \sin(2\pi y_i).
\]

First, we use the loss function defined in (\ref{loss_Ritz}) 
for Step 4 of Algorithm \ref{Algorithm_1}, training the TNN for 1,000 iterations with Adam and subsequently for 1000 iterations with LBFGS.
The relative $L^2$ norm error of the approximation $\Phi(\boldsymbol{x};c,\Theta)$ 
to the exact solution $u(\boldsymbol{x})$ is \(2.2296 \times 10^{-5}\).  
The corresponding numerical results are presented in Figure \ref{fig_errors_ex6_ritz}.  
\begin{figure}[ht]
\centering
\includegraphics[width=6cm,height=5cm]{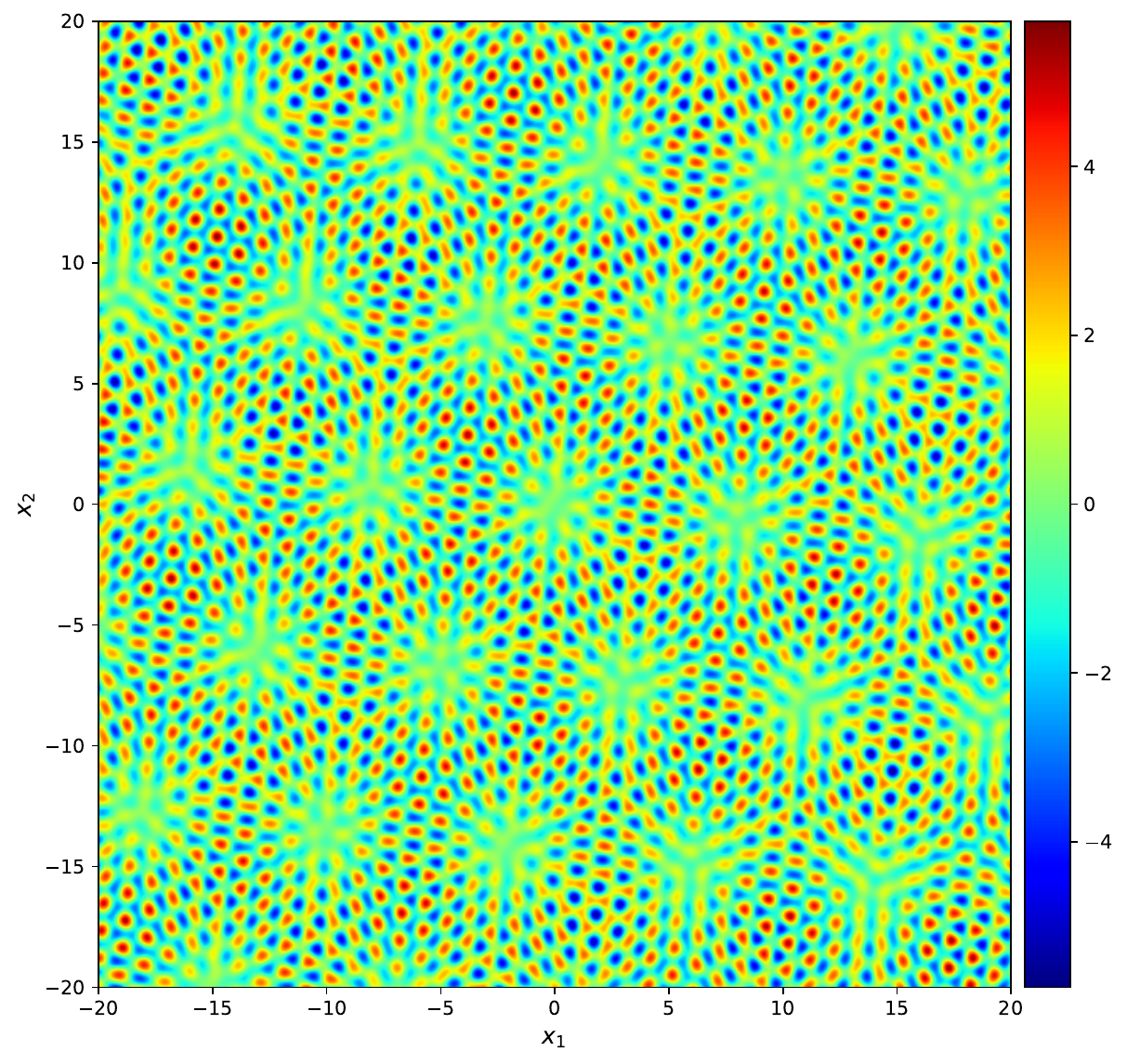}
\includegraphics[width=6cm,height=5cm]{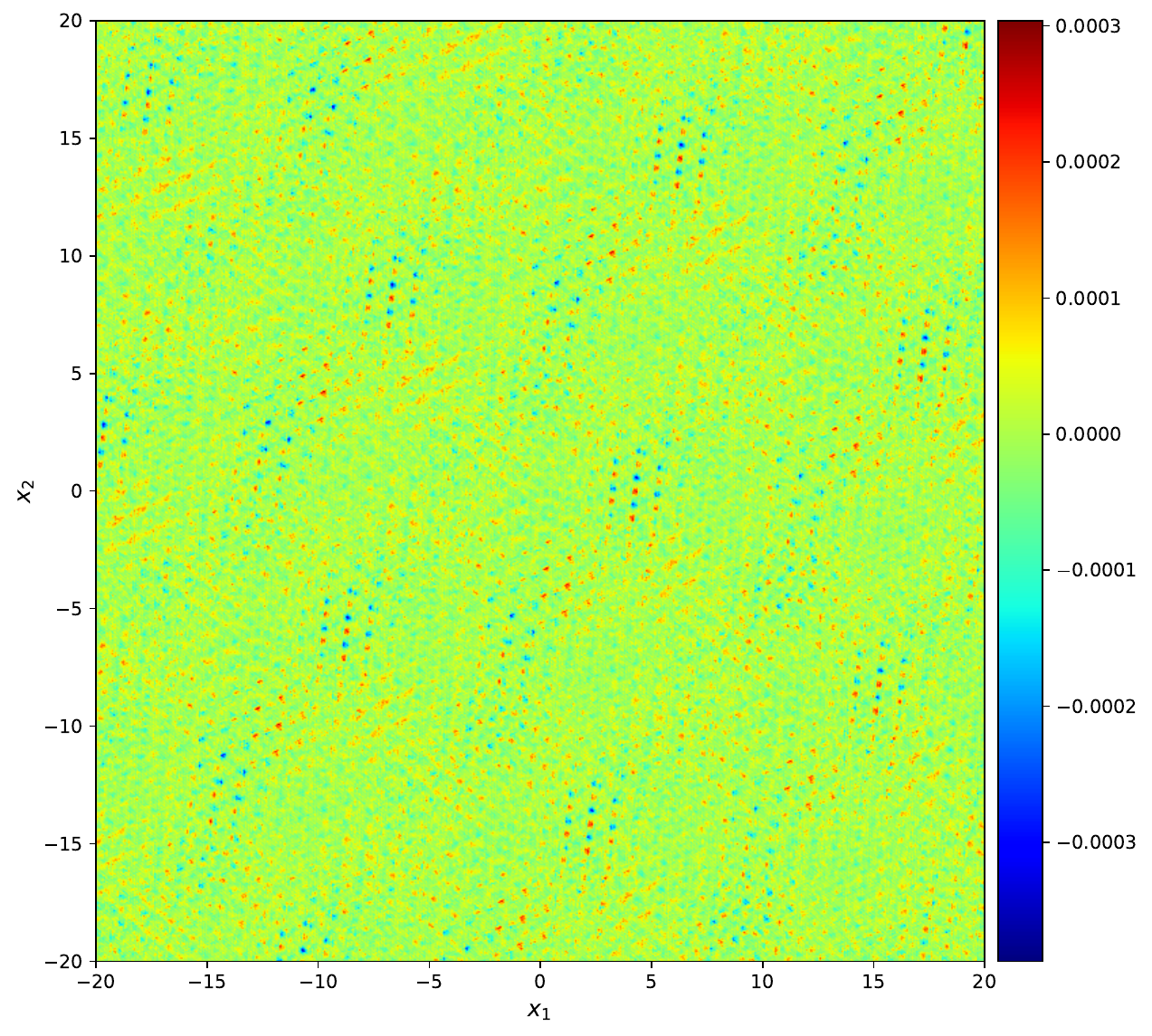}
\caption{The numerical results for Example 6 with loss function (\ref{loss_Ritz}), left shows 
the figure of the approximation $\Phi(\boldsymbol{x};c,\Theta)$ to 
the quasiperiodic funtion $u(\boldsymbol{x})$, 
right shows the error $u(\boldsymbol{x})-\Phi(\boldsymbol{x};c,\Theta)$.}\label{fig_errors_ex6_ritz}
\end{figure}

To further improve the accuracy, then we use the loss function defined in (\ref{loss_Residual}) 
for Step 4 of Algorithm \ref{Algorithm_1}, training the TNN for 2000 iterations with LBFGS.
The relative $L^2$ norm error of the approximation $\Phi(\boldsymbol{x};c,\Theta)$ 
to the exact solution $u(\boldsymbol{x})$ is \(3.4881 \times 10^{-6}\).  
The corresponding numerical results are presented in Figure \ref{fig_errors_ex6}.
\begin{figure}[ht]
\centering
\includegraphics[width=6cm,height=5cm]{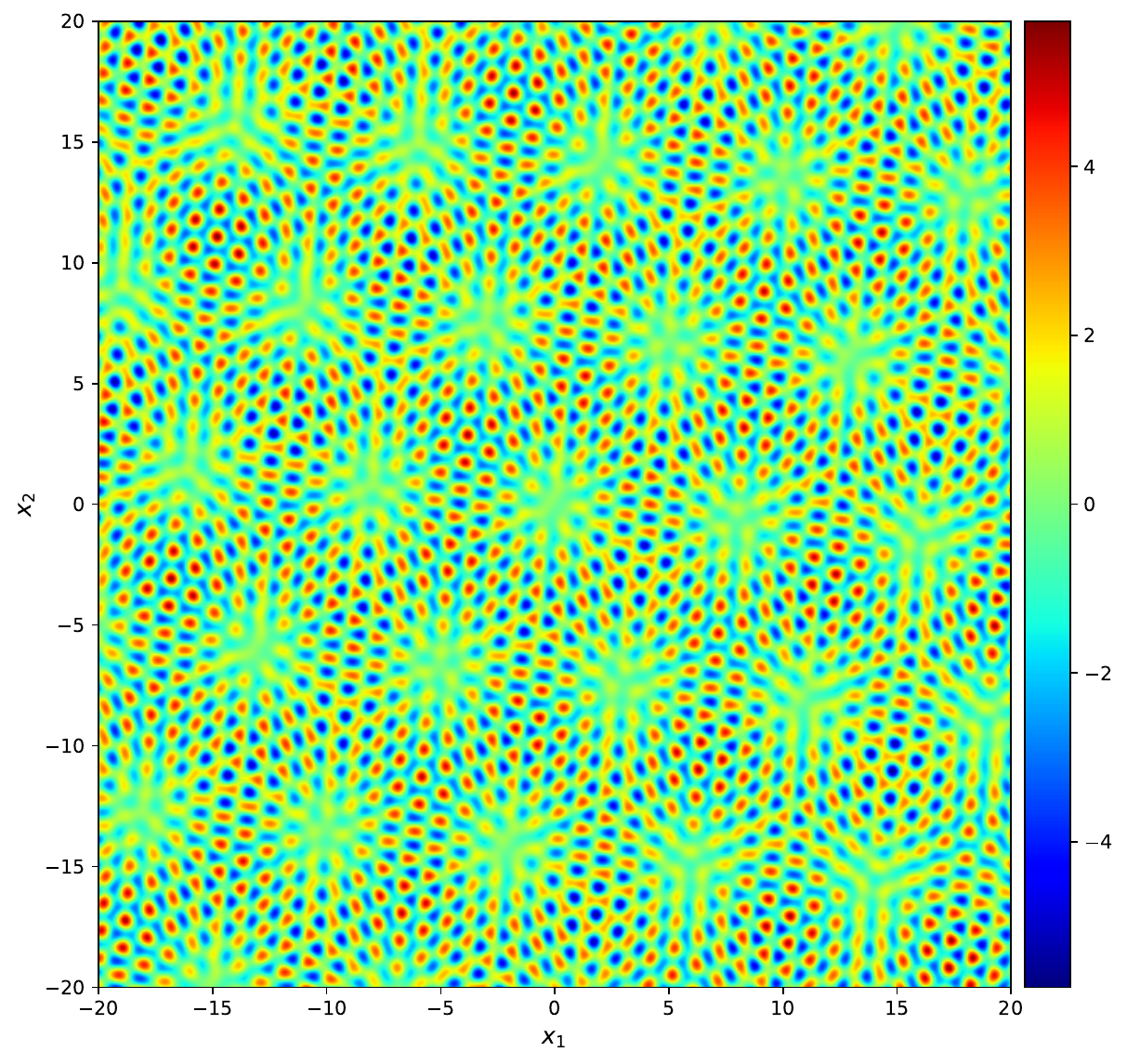}
\includegraphics[width=6cm,height=5cm]{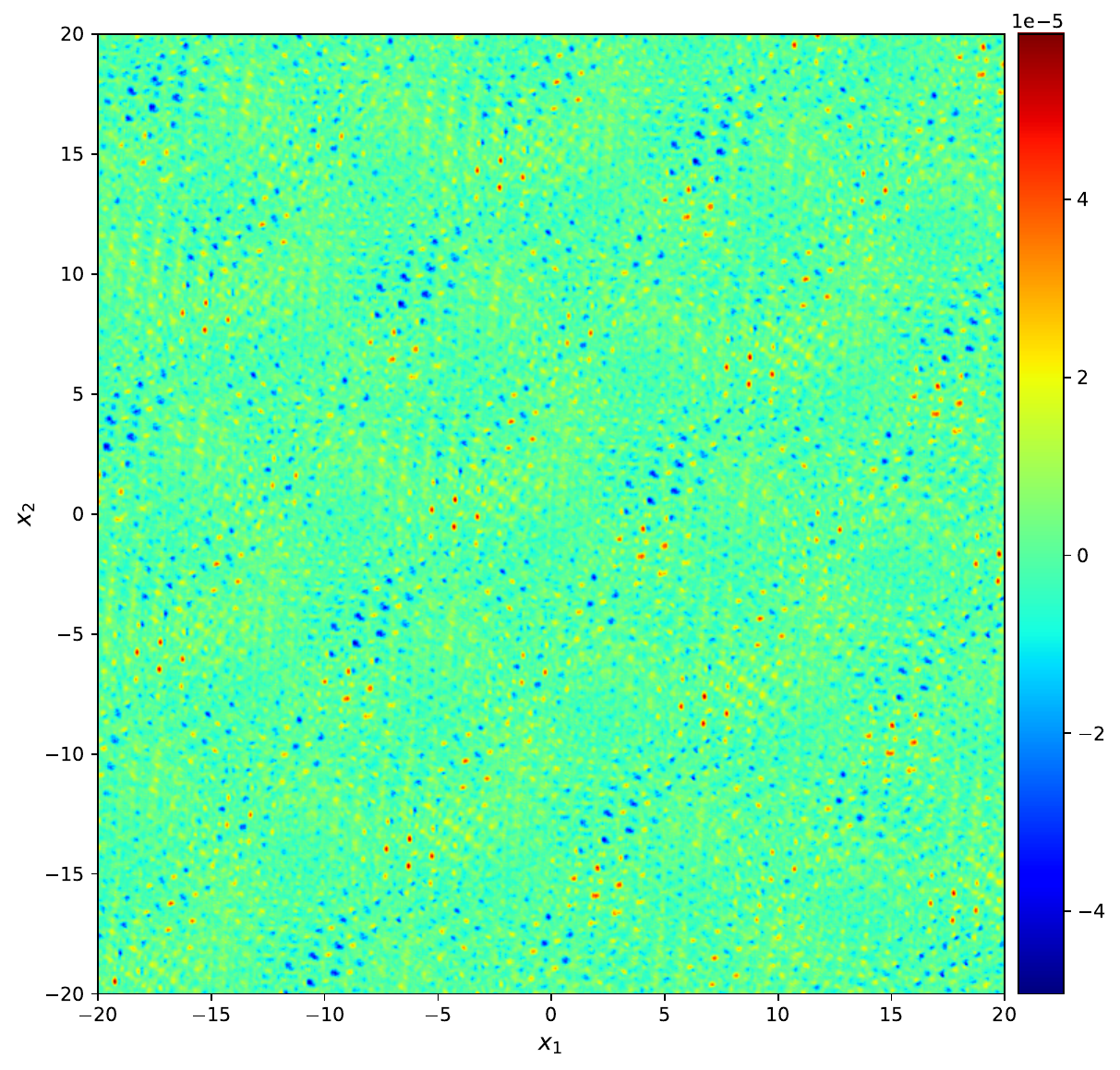}
\caption{The numerical results for Example 6 with loss function (\ref{loss_Residual}), left shows 
the figure of the approximation $\Phi(\boldsymbol{x};c,\Theta)$ to 
the quasiperiodic funtion $u(\boldsymbol{x})$, 
right shows the error $u(\boldsymbol{x})-\Phi(\boldsymbol{x};c,\Theta)$.}\label{fig_errors_ex6}
\end{figure}


\subsection{Example 7}
In this example,  we solve the quasiperiodic elliptic 
problem (\ref{Elliptic_Equation}) with the coefficient 
having  $6$ $\mathbb Q$-independent periods  coefficient 
\[
\alpha = \sum_{i=0}^{5} \cos\left( 2\pi (\cos(i) x_1 + \sin(i) x_2) \right) 
+ \cos\left( 2\pi^2 (\cos(i) x_1 + \sin(i) x_2) \right) + 12.
\]
The source function $f$ is chosen such that 
\[
u = \sum_{i=0}^{5} \sin\left( 2\pi (\cos(i) x_1 + \sin(i) x_2) \right) 
+ \sin\left( 2\pi^2 (\cos(i) x_1 + \sin(i) x_2) \right).
\]
In order to use the projection method, we define the matrices $P_1$ and $P_2$ as 
\[
P_1 = \begin{bmatrix} 
1 & \cos(1) & \cos(2) & \cos(3) & \cos(4) & \cos(5) \\ 
0 & \sin(1) & \sin(2) & \sin(3) & \sin(4) & \sin(5) 
\end{bmatrix},
\]
and 
\[
P_2 = 
\begin{bmatrix} 
\pi & \pi \cos(1) & \pi \cos(2) & \pi \cos(3) & \pi \cos(4) & \pi \cos(5) \\ 
0   & \pi \sin(1) & \pi \sin(2) & \pi \sin(3) & \pi \sin(4) & \pi \sin(5) 
\end{bmatrix}.
\]
Then the projection matrix can be defined as 
\[
P = \begin{bmatrix} 
P_1 & P_2 
\end{bmatrix},
\]
Then it is easy to deduce the higher dimensional coefficient and exact solution 
\[
A = \sum_{i=1}^{12} \cos(2\pi y_i) + 12,\ \ \ \ 
U = \sum_{i=1}^{12} \sin(2\pi y_i).
\]

We use the loss function defined in (\ref{loss_Ritz}) 
for Step 4 of Algorithm \ref{Algorithm_1}, training the TNN for 1,000 iterations with Adam and subsequently for 2000 iterations with LBFGS.
The relative $L^2$ norm error of the approximation $\Phi(\boldsymbol{x};c,\Theta)$ 
to the exact solution $u(\boldsymbol{x})$ is \(1.5906 \times 10^{-5}\).  
The corresponding numerical results are presented in Figure \ref{fig_errors_ex7_ritz}.  
\begin{figure}[ht]
\centering
\includegraphics[width=6cm,height=5cm]{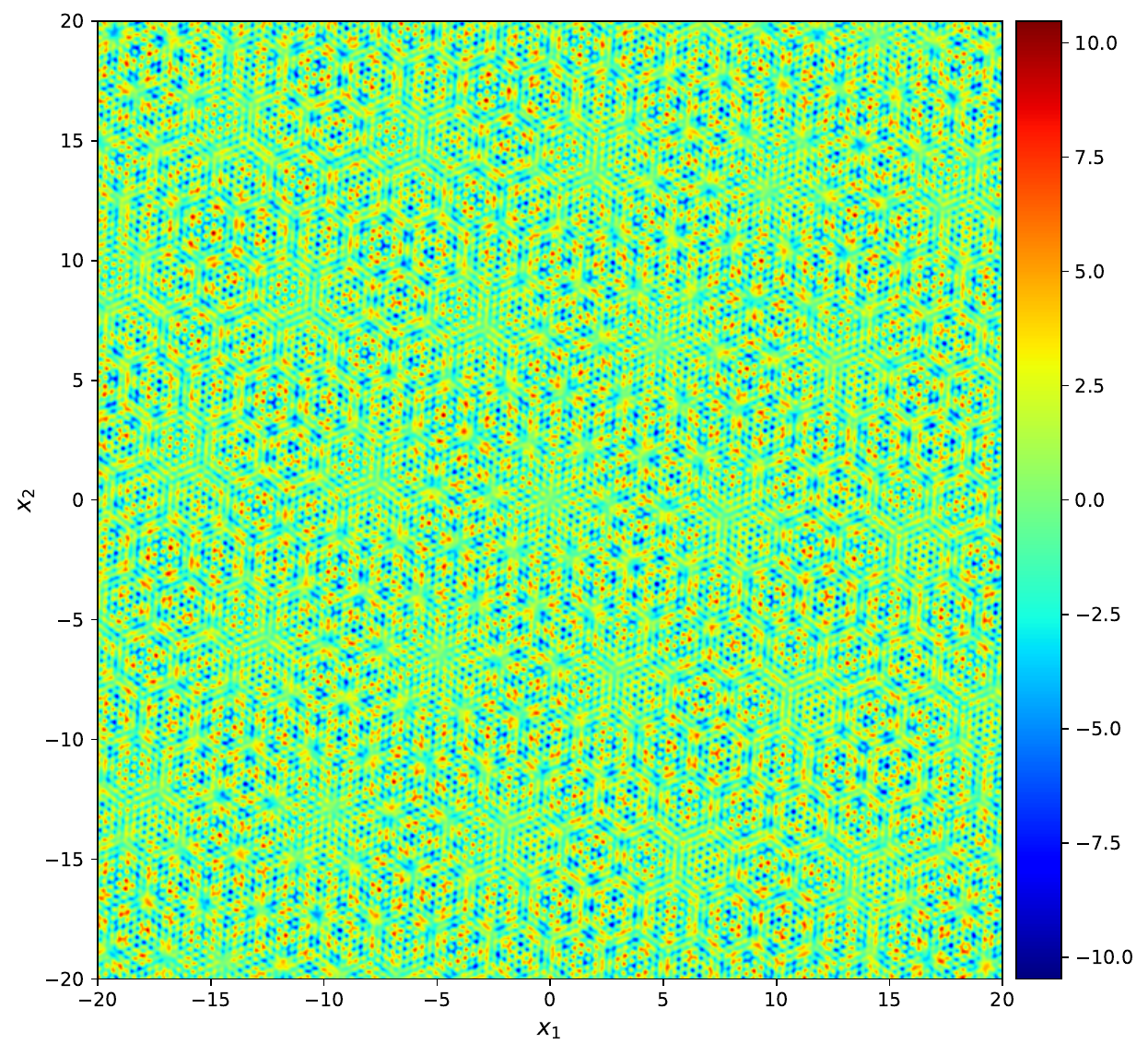}
\includegraphics[width=6cm,height=5cm]{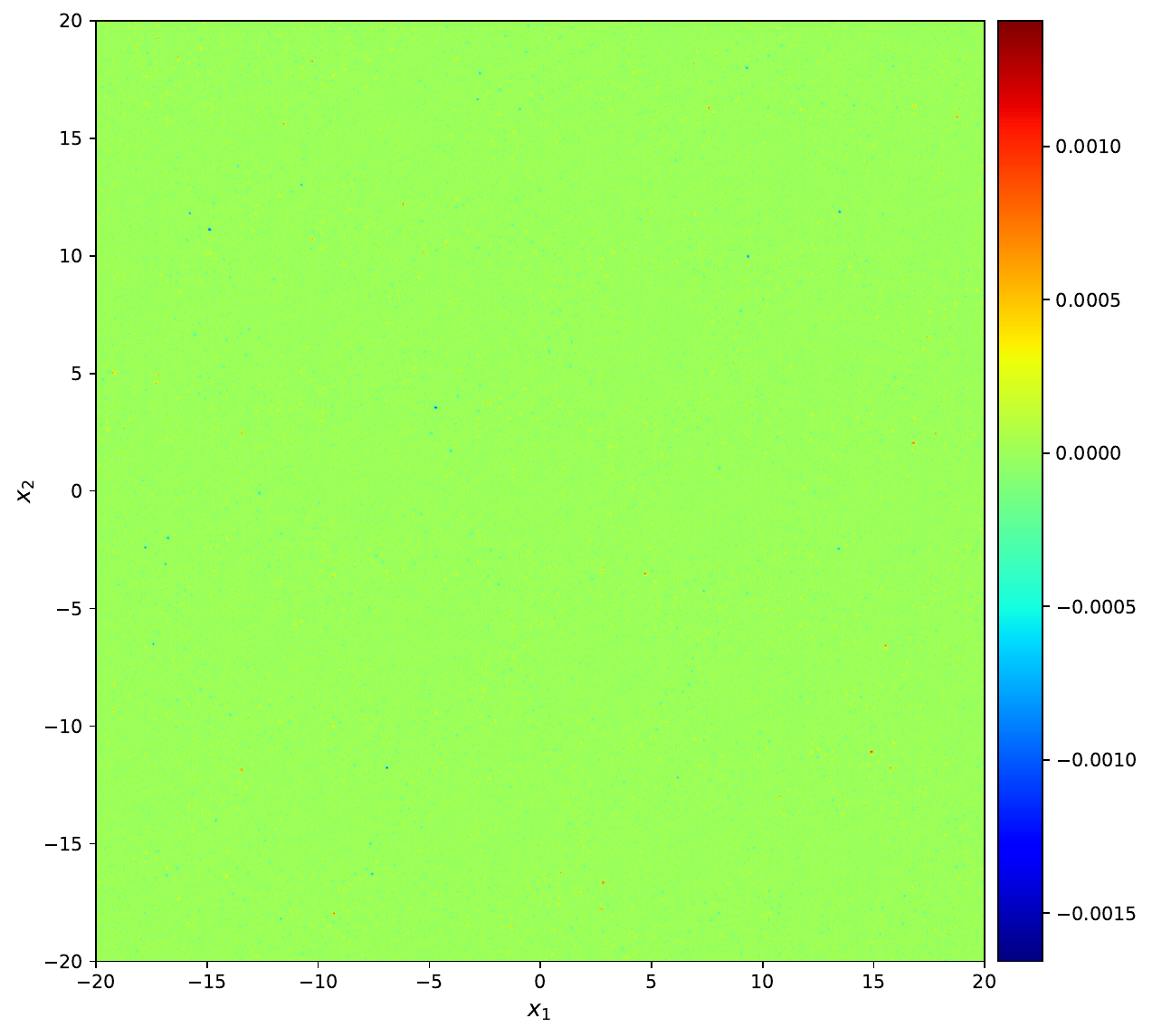}
\caption{The numerical results for Example 7 with loss function (\ref{loss_Ritz}), left shows 
the figure of the approximation $\Phi(\boldsymbol{x};c,\Theta)$ to 
the quasiperiodic funtion $u(\boldsymbol{x})$, 
right shows the error $u(\boldsymbol{x})-\Phi(\boldsymbol{x};c,\Theta)$.}\label{fig_errors_ex7_ritz}
\end{figure}

\subsection{Example 8}
In this example,  we solve the quasiperiodic elliptic 
problem (\ref{Elliptic_Equation}) with the coefficient 
having  $13$ $\mathbb Q$-independent periods  coefficient 
\[
\alpha=\sum_{i=0}^{12} \cos\left(2\pi \left(\cos\left(\frac{i\pi}{13}\right)x+\sin\left(\frac{i\pi}{13}\right)y\right)\right)+12.
\]
The source function $f$ is chosen such that 
\[
u=\sum_{i=0}^{12} \sin\left(2\pi \left(\cos\left(\frac{i\pi}{13}\right)x+\sin\left(\frac{i\pi}{13}\right)y\right)\right).
\]

In order to use the projection method, the projection matrix can be defined as 
\[
P=\begin{bmatrix}
1 & \cos(\frac{\pi}{13}) & \cos(\frac{2\pi}{13}) & \cdots & \cos(\frac{10\pi}{13}) & \cos(\frac{11\pi}{13})& \cos(\frac{12\pi}{13})\\ \\
0 & \sin(\frac{\pi}{13}) & \sin(\frac{2\pi}{13}) & \cdots & \sin(\frac{10\pi}{13}) & \sin(\frac{11\pi}{13})& \sin(\frac{12\pi}{13})
\end{bmatrix},
\]
Then it is easy to deduce the higher dimensional coefficient and exact solution 
\[
A = \sum_{i=1}^{13} \cos(2\pi y_i) + 12,\ \ \ \ 
U = \sum_{i=1}^{13} \sin(2\pi y_i).
\]

We use the loss function defined in (\ref{loss_Ritz}) 
for Step 4 of Algorithm \ref{Algorithm_1}, training the TNN for 1,000 iterations with Adam and subsequently for 2000 iterations with LBFGS.
The relative $L^2$ norm error of the approximation $\Phi(\boldsymbol{x};c,\Theta)$ 
to the exact solution $u(\boldsymbol{x})$ is \(1.9458 \times 10^{-5}\).  
The corresponding numerical results are presented in Figure \ref{fig_errors_ex8_ritz}.  
\begin{figure}[ht]
\centering
\includegraphics[width=6cm,height=5cm]{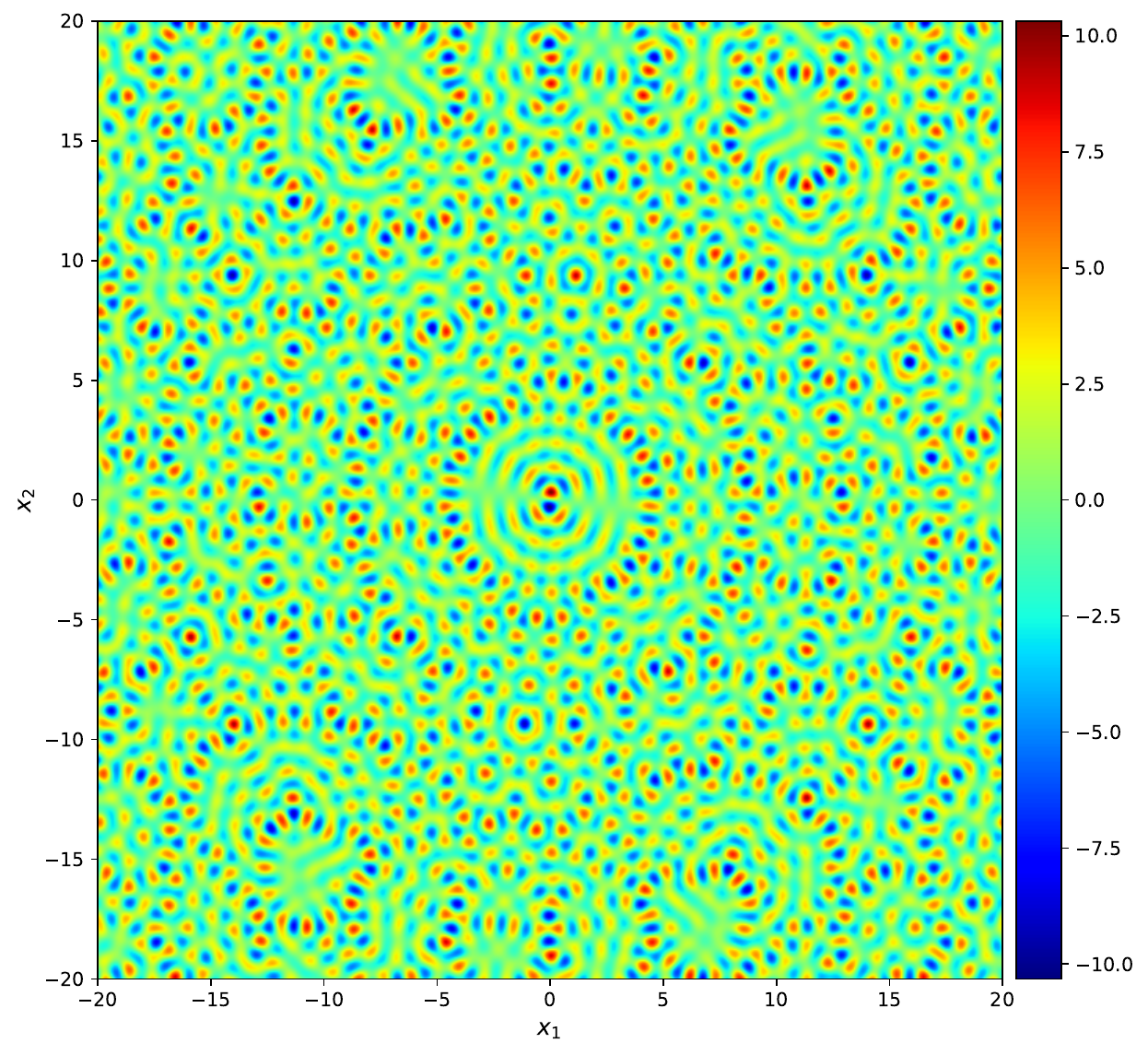}
\includegraphics[width=6cm,height=5cm]{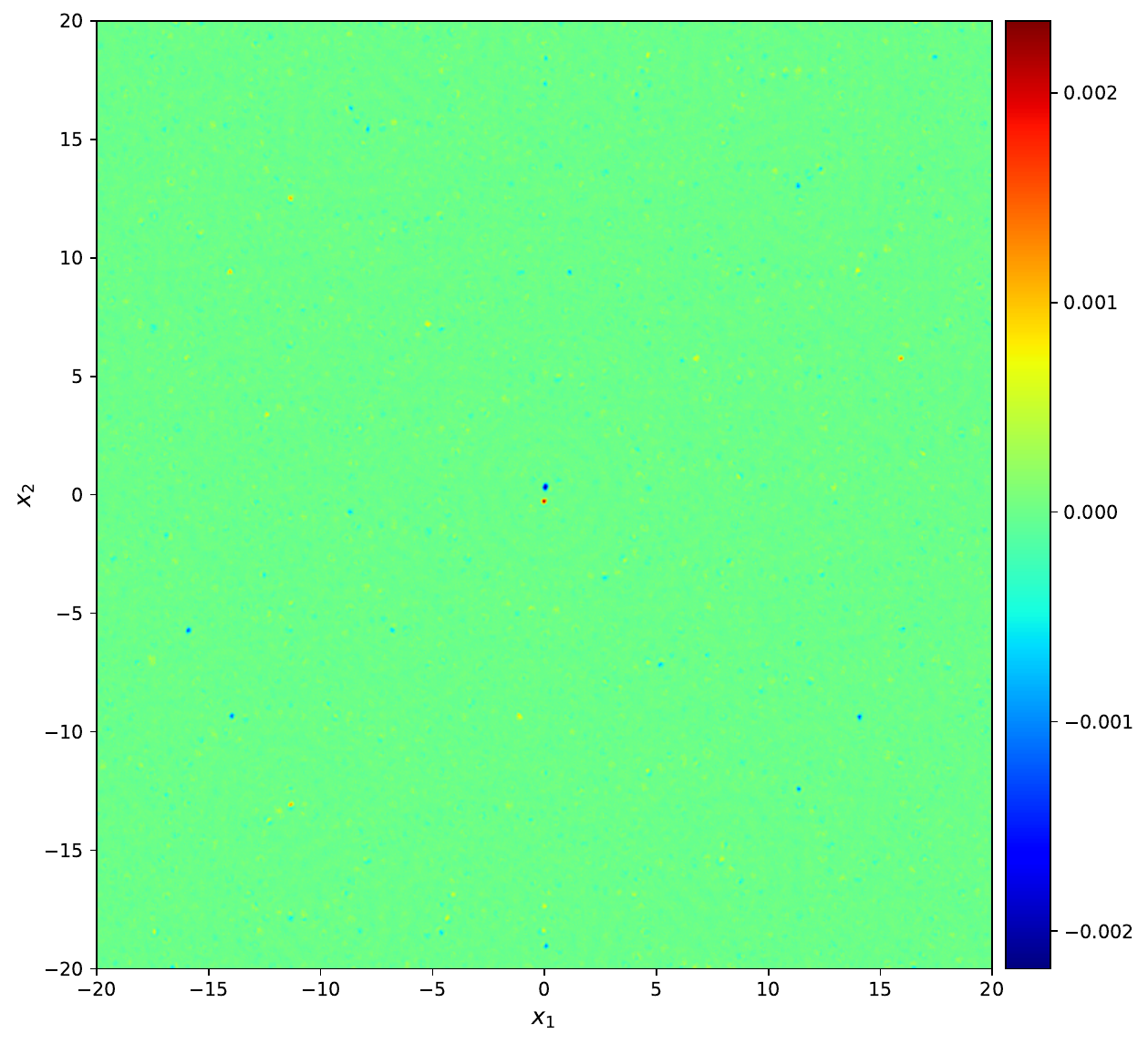}
\caption{The numerical results for Example 8 with loss function (\ref{loss_Ritz}), left shows 
the figure of the approximation $\Phi(\boldsymbol{x};c,\Theta)$ to 
the quasiperiodic funtion $u(\boldsymbol{x})$, 
right shows the error $u(\boldsymbol{x})-\Phi(\boldsymbol{x};c,\Theta)$.}\label{fig_errors_ex8_ritz}
\end{figure}

\section{Conclusion}
In this paper, we develop a TNN-based machine learning method for solving elliptic partial differential equations with quasiperiodic coefficients.
To provide a theoretical foundation for the proposed numerical method, we carry out a detailed analysis of the existence, uniqueness, 
and regularity of solutions to quasiperiodic problems.
In particular, rather than assuming a priori that the exact solution already belongs to a standard Sobolev space, we innovatively 
identify natural sufficient conditions under which the quasiperiodic solution can be lifted to a Sobolev space.
Unlike standard FNN-based machine learning methods, TNNs possess a tensor-product structure, which allows the corresponding 
high-dimensional integrals to be decomposed into one-dimensional integrals and evaluated with high accuracy.
Benefiting from this feature, the proposed TNN-based method is able to solve the resulting high-dimensional periodic problems with high accuracy.
Combined with the projection method, it further yields highly accurate approximations to the original quasiperiodic problems.
We also present detailed numerical experiments, which demonstrate the efficiency and accuracy of the proposed method.
These results indicate the strong potential of TNN-based machine learning methods for broader applications to quasiperiodic problems.
Further investigations along this direction will be carried out in future work.

\bibliographystyle{siamplain}

\end{document}